\documentclass{amsart}

\makeatletter
\@namedef{subjclassname@2020}{%
  \textup{2020} Mathematics Subject Classification}
\makeatother 
\usepackage{amsthm,amssymb,amsfonts,latexsym,mathtools,thmtools,amsmath,lscape}
\usepackage[T1]{fontenc}
\usepackage{tikz-cd} 
\usepackage{enumitem} 
\usepackage{longtable}
\usepackage{multirow}
\usepackage{caption}
\usepackage{array}
\usepackage{mathrsfs} 
\usepackage{hyperref}
\usepackage{float}
\usepackage[all]{xy}
\usepackage{booktabs}

\hypersetup{
    colorlinks=true,
    linkcolor=blue,
    filecolor=blue,      
    urlcolor=blue,
    linktocpage=true
}

\usepackage{fvextra}
\DefineVerbatimEnvironment{MyVerbatim}{Verbatim}{
  fontsize=\small,
  breaklines=true,
  breakanywhere=true,
  breaksymbolleft=\textcolor{gray}{\tiny$\hookrightarrow$},
  xleftmargin=1em
}

\newtheorem{theorem}{Theorem}[section]

\newtheorem{proposition}[theorem]{Proposition}

\theoremstyle{definition}
\newtheorem{definition}[theorem]{Definition}

\newtheorem{remark}[theorem]{Remark}

\theoremstyle{remark}

\numberwithin{equation}{section}

\begin{document}

\title[Normal ordering of 3-dimensional skew polynomial rings]{Combinatorial aspects of normal ordering of \\ 3-dimensional skew polynomial rings}


\author{Andr\'es Rubiano}
\address{Universidad Distrital Francisco Jos\'e de Caldas}
\curraddr{Campus Universitario}
\email{aarubianos@udistrital.edu.co}

\author{Armando Reyes}
\address{Universidad Nacional de Colombia - Sede Bogot\'a}
\curraddr{Campus Universitario}
\email{mareyesv@unal.edu.co}
\thanks{The second author was supported by Faculty of Science, Universidad Nacional de Colombia--Sede Bogot\'a, Colombia [grant number 65488].}



\date{}

\dedicatory{Dedicated to Professor Oswaldo Lezama on the Occasion of His 70th Birthday}

\begin{abstract} 

In this paper, we discuss combinatorial aspects of normal ordering of 3-dimensional skew polynomial rings defined and classified by Bell and Smith \cite{BellSmith1990}. With some help of the Mathematical software \texttt{SageMath}, we are able to reduce the length of computation of PBW forms and normal orderings appearing in commutation rules of these algebras.

\end{abstract}

\maketitle


\section{Introduction}

Just as Schork \cite[Section 1]{Schork2021} asserted, \textquotedblleft Maybe the first explicit results concerning normal ordering were derived by Scherk \cite{Scherk1823} in his dissertation from 1823.\textquotedblright\ In his thesis (and using the current terminology in the literature), he considered the action of two operators $X$ and $D$ on functions of a real variable defined by
\begin{equation}\label{Schork2021(0)}
(Xf)(x) = xf(x) \quad {\rm and} \quad (Df)(x) = \frac{df}{dx}(x),
\end{equation}
respectively. As it can be seen, these two operators do not commute, and if $I$ denotes the identity operator, they satisfy the {\em commutation relation} given by
\begin{equation}\label{Schork2021(1)}
DX - XD = I.
\end{equation}
As is well-known, this algebra is called the {\em first Weyl algebra} and is defined over the field of complex numbers $\mathbb{C}$ as the $\mathbb{C}$-algebra $A_1(\mathbb{C}) = \mathbb{C} \{X, D\} / \langle XD - DX - I\rangle$. History and a detailed study of this noncommutative ring was made by Coutinho \cite{Coutinho1995}.

If we consider the definition of an {\em arbitrary word} in the indeterminates $X$ and $D$, then we mean an object $\omega$ that can be written as $\omega = X^{r_1} D^{s_1} X^{r_2} D^{s_2} \dotsb X^{r_n} D^{s_n} $ for some elements $r_k, s_k$ belonging to the set of natural numbers $\mathbb{N}$ (we consider the zero element as a natural number). Having in mind the relation (\ref{Schork2021(1)}), it is necessary to respect the order of the operators when acting on a function $f$; this motivates the following definition: a word $\omega$ in the letters $X$ and $D$ is said to be in {\em normal ordered form} if $\omega = a_{r, s}X^r D^s$ for $r, s \in \mathbb{N}$ (and arbitrary coefficients $a_{r,s}$ in some kind of ring or field - not necessarily the field of complex numbers). With this, an expression that consists of a sum of words is called {\em normal ordered} if each of the summand is normal ordered. The process of expressing a word (or a sum of words) into a normal ordered form by using the commutation relation is called {\em normal ordering}, which is the key notion of the paper. In this procedure, we get expressions such as
\begin{equation}\label{Schork2021(2)}
\omega = \sum_{r, s \in \mathbb{N}} S_{r,s}(\omega) X^rD^s,
\end{equation}
with the uniquely determined coefficients $S_{r,s}(\omega)$ known as {\em normal ordering coefficients}.

As a matter of fact, Scherk \cite{Scherk1823} also determined the normal ordered form of the words $(XD)^n$ and $(X^pD)^n$ and investigated combinatorial properties of the normal ordering coefficients. For instance, he proved that 
\begin{equation}
(XD)^n = \sum_{k=1}^n S_{1,1}(n, k) X^kD^k,
\end{equation}
where $S_{1,1}(n+1, k) = S_{1,1}(n, k-1) + kS_{1,1}(n, k)$. Today, these normal ordering coefficients are known today as the {\em Stirling numbers of the second kind} $S(n, k)$ \cite[A008277]{Sloane} (of course, Scherk did not recognize them). In the extensive literature about these numbers, we find the relations between them and the notion of {\em set partition}. In fact, Scherk also considered the analogous normal ordering coefficients $S_{p, 1}(n, k)$ - which are particular examples of the {\em generalized Stirling numbers} - These objects written as $S_{p, q}(n, k)$ are defined by considering normal orderings in the expression $(X^{p} D^{q})^n$. 

A major theme in subsequent work is that \emph{changing the commutation rule} (or moving to deformations and variants of the Weyl algebra) produces new families of Stirling-type numbers with rich combinatorics. A striking example is the interpretation of normal order coefficients in terms of rook numbers on Ferrers boards, leading to explicit formulas and factorization phenomena as shown by Varvak \cite{Varvak2005}. Deformations also connect normal ordering to $q$-analogues and $q$-Stirling families \cite{Schork2006}, while more refined operator strings yield generalized Bell-Stirling structures with graph-theoretic models and additional combinatorial interpretations \cite{MendezBlasiakPenson2005,CodaraDAntonaHell2014}. These contributions establish normal ordering as a robust \emph{combinatorial machine} that turns commutation relations into computable, structured coefficient systems.

In the literature, we can find a large number of works that have addressed the study of normal orders for various families of non-commutative algebras such as the {\em quantum plane}, the $q$-{\em deformed Weyl algebra}, the algebras in physical models of quantum theory defined by the {\em creation operator} and the {\em annihilation operator}, and other generalizations of the Weyl algebra (e.g. \cite{BlasiakPensonSolomon2003, BlasiakFlajolet2011, BlasiakHorzelaPensonSolomonDuchamp2007, HsuShiue1998, MansourSchork2015Book, Maslov1976, Nazaikinskiietal1995, Schork2021}, and references therein). Recently, Briand et al. \cite{BriandLopesRosas2020} studied normally ordered forms for powers of differential operators, deriving generalized Stirling numbers and arithmetic properties of the coefficients, while Mansour and Schork \cite{MansourSchork2023} developed \emph{Ore-Stirling numbers} as normal ordering coefficients in a more general setting than Patrias and Pylyavskyy \cite{PatriasPylyavskyy2018}. As it can be seen in all these works, normal ordering becomes an efficient bridge between algebraic identities and concrete combinatorial classes (set partitions, contractions, graphs), yielding recurrences, generating functions and closed forms for the associated coefficient arrays. It is important to note that all structures considered above have only one nontrivial commutation relation.

The algebras of interest in this paper are the 3-dimensional skew polynomial rings introduced by Bell and Smith \cite{BellSmith1990} as a class of noncommutative algebras on three generators admitting a PBW basis of standard monomials which are defined by three nontrivial commutation relations. Beyond their intrinsic ring-theoretic interest, $3$-dimensional skew polynomial algebras naturally capture familiar structures from Lie theory and quantum groups such as the universal enveloping algebra $U(\mathfrak{sl}(2))$, the dispin algebra $U(\mathfrak{osp}(1,2))$ and Woronowicz-type algebras associated to noncommutative differential calculi appear within this framework \cite{ReyesSuarez20173D, Woronowicz1987}. 

Our purpose in this paper is to develop the algebraic combinatorics underlying PBW normal forms in the class of $3$-dimensional skew polynomial algebras. Fixing the PBW order $x\prec y\prec z$, we regard PBW reduction as a rewriting process that eliminates the forbidden adjacent pairs $yx$, $zy$ and $zx$ via the affine relations (as expected, our ideas are based on Bergman's paper \cite{Bergman1978}, see also \cite{ReyesSuarez20173D}). This viewpoint turns the computation of products and powers into a systematic \emph{normal ordering problem}: given a noncommutative expression, determine its unique expansion in the PBW basis consisting of the monomials $x^i y^j z^{l}$ with $i, j, l \in \mathbb{N}$. Our approach is uniform and computable: we work in a parameterized $3$-skew family and then specialize the parameters to cover the $15$ isomorphism types in the Bell and Smith's classification. In our treatment, we keep in mind Schork's words about algebras with more than two generators: \textquotedblleft When considering algebras with three or more generators where none of the generators is central (hence, one has at least three nontrivial commutation relations) one can clearly define normal ordering in an analogous fashion but the considerations quickly become tedious\textquotedblright\ \cite[p. 4]{Schork2021}. These words motivated us to look for computational alternatives to reduce calculations, and it is the reason why we developed a code for it in the mathematical software \texttt{SageMath}.

The article is organized as follows. Section \ref{Preliminaries} presents the definition and classification of 3-dimensional skew polynomial rings. Next, Section \ref{Normalorderingcases} contains the original results of the paper. With some help of the Mathematical software \texttt{SageMath}, we are able to compute PBW forms and hence normal orderings appearing in commutation rules. More exactly, for each of the fifteen 3-dimensional skew polynomial rings (Proposition \ref{3-dimensionalClassification}), we establish the fundamental two letter normal ordering identities (such as $y^n x^m$, $z^n x^m$, and $z^n y^m$) that serve as basic rewriting rules. These identities are then used as building blocks to derive the higher-order expansions required in our combinatorial results, including powers of blocks and mixed words (e.g.\ $(x^n y^m)^s$, $(xyz)^s$, $(x^n y^m z^t)^s$) and the corresponding coefficient recursions. We want to remark that the software \texttt{SageMath} is used in an elementary way only to reduce the
length of the computation and all computation can be done by hand without assistance of it. Precisely, Section \ref{Computationalsetup} contains the code we implemented in \texttt{SageMath} that allowed us to compute the normal ordered forms of every algebra. Finally, Section \ref{Futurework} presents some ideas for possible future work related to normal ordering issues.

Throughout the paper, $\mathbb{N}$ denotes the set of natural numbers including the zero element. The term ring means an associative ring with identity, not necessarily commutative. All vector spaces and algebras (always associative and unital) are over a fixed field $\Bbbk$, which is assumed to have characteristic zero. As usual, $\Bbbk^{\times} := \Bbbk \ \backslash \ \{0\}$.

\section{Definitions and Preliminaries}\label{Preliminaries}

\begin{definition}[{\cite{BellSmith1990}, \cite[Definition C4.3]{Rosenberg1995}}]\label{3dimensionaldimension}
\textit{A 3-dimensional skew polynomial algebra} $A$ is a $\Bbbk$-algebra generated by the indeterminates $x,y,z$ restricted to relations 
$$
yz - \alpha zy = \lambda,\quad zx - \beta xz = \mu \quad {\rm and} \quad xy - \gamma yx = \nu,
$$ 
such that
\begin{enumerate}
\item [\rm (i)] $\lambda, \mu, \nu\in \Bbbk+\Bbbk x+\Bbbk y+\Bbbk z$ and $\alpha, \beta, \gamma \in \Bbbk\ \backslash\ \{0\}$ and

\item [\rm (ii)] standard monomials $\left\{x^iy^jz^l\mid i,j,l\ge 0\right\}$ are a $\Bbbk$-basis of the algebra $A$.
\end{enumerate}
\end{definition}

\begin{proposition}[{\cite[Theorem C.4.3.1]{Rosenberg1995}}]\label{3-dimensionalClassification}
If $A$ is a 3-dimensional skew polynomial algebra, then $A$
is one of the following algebras:
\begin{enumerate}
\item [\rm (1)] If\ $|\{\alpha, \beta, \gamma\}|=3$, then $A$ is given by the relations 
$$
yz - \alpha zy = 0,\quad zx - \beta xz = 0\quad {\rm and} \quad xy - \gamma yx = 0.
$$

\item [\rm (2)] If $|\{\alpha, \beta, \gamma\}|=2$ and $\beta\neq \alpha =\gamma =1$, then $A$ is one of the following algebras:
\begin{enumerate}
\item [\rm (i)] $yz-zy=z,\ zx-\beta xz=y$ and $xy - yx = x$ {\rm (}if $\beta = -1$, then we get the Dispin algebra{\rm )}.

\item [\rm (ii)] $yz-zy=z,\ zx-\beta xz=b$ and $xy-yx=x${\rm ;}

\item [\rm (iii)] $yz-zy=0,\ zx-\beta xz=y$ and $xy - yx = 0${\rm ;}

\item [\rm (iv)] $yz-zy=0,\ zx-\beta xz=b$ and $xy - yx = 0${\rm ;}

\item [\rm (v)] $yz-zy=az,\ zx-\beta xz = 0$ and $xy - yx = x${\rm ;}

\item [\rm (vi)] $yz-zy=z,\ zx-\beta xz=0$ and $xy - yx = 0$,
\end{enumerate}
where $a, b$ are any elements of\ $\Bbbk$. All non-zero values of $b$ give isomorphic algebras.

\item [\rm (3)] If $|\{\alpha, \beta, \gamma\}|=2$ and $\beta\neq \alpha=\gamma\neq 1$, then $A$ is one of the following algebras:
\begin{enumerate}
\item [\rm (i)] $yz-\alpha zy=0,\ zx-\beta xz = y+b$ and $xy - \alpha yx=0${\rm ;}

\item [\rm (ii)] $yz-\alpha zy=0,\ zx - \beta xz=b$ and $xy - \alpha yx = 0$.
\end{enumerate}
In this case, $b$ is an arbitrary element of $\Bbbk$. Again, any
non-zero values of $b$ give isomorphic algebras.

\item [\rm (4)] If $\alpha=\beta=\gamma\neq 1$, then $A$ is the algebra defined by the relations
\begin{align*}
yz - \alpha zy = &\ a_1x+b_1, \\
zx - \alpha xz = &\ a_2y+b_2, \ {\rm and} \\
xy - \alpha yx = &\ a_3z+b_3.
\end{align*}
If $a_i = 0$ for all $i$, then all non-zero values of $b_i$ give isomorphic
algebras.

\item [\rm (5)] If $\alpha=\beta=\gamma=1$, then $A$ is isomorphic to one of the following algebras:
\begin{enumerate}
\item [\rm (i)] $yz-zy=x,\ zx-xz=y$ and $xy - yx = z${\rm ;}

\item [\rm (ii)] $yz-zy=0,\ zx-xz=0$ and $xy - yx = z${\rm ;}

\item [\rm (iii)] $yz-zy=0,\ zx-xz=0$ and $xy - yx = b${\rm ;}

\item [\rm (iv)] $yz-zy=-y,\ zx-xz=x+y$ and $xy-yx=0${\rm ;}

\item [\rm (v)] $yz-zy=az,\ zx-xz=z$ and $xy - yx = 0${\rm ;}
\end{enumerate}
Parameters $a,b\in \Bbbk$ are arbitrary,  and all non-zero values of
$b$ generate isomorphic algebras.
\end{enumerate}
\end{proposition}

\begin{remark}\label{Rosenberg}
In \cite[Theorem C4.3.1, case (e)(v), p. 101]{Rosenberg1995}, the relation is stated as
\[
yz - zy = az, \quad zx - xz = x \quad {\rm and} \quad xy - yx = 0.
\]
However, this is a typo. Indeed, Rosenberg later realized the same algebra as the Ore extension
\[
A = \Bbbk[x,y][z;\varphi],
\]
where $\varphi$ is an automorphism of $\Bbbk[x,y]$ given by $\varphi(x)=x+1$ and $\varphi(y)=y-a$ \cite[p. 108]{Rosenberg1995}. By the defining rule $za=\varphi(a)z$ for all $a\in \Bbbk[x,y]$, we obtain that 
\begin{align*}
zx = \varphi(x)z = (x+1)z & \quad {\rm whence} \quad zx - xz = z, \\
zy = \varphi(y)z = (y-a)z & \quad {\rm and\ so} \quad yz - zy = az.
\end{align*}
Therefore, the correct presentation of case (e)(v) is
\[
yz - zy = az, \quad zx - xz = z \quad {\rm and} \quad xy- yx = 0.
\]
\end{remark}

Ring, geometric, and differential properties of 3-dimensional skew polynomial rings have been considered in different papers (e.g. \cite{RedmanPhDThesis1996} - \cite{RubianoReyes2026DS3D}).

\section{Normal ordering identities}\label{Normalorderingcases}

This section contains the original results of the paper. We present explicit PBW computations for each one of the $3$-dimensional skew polynomial algebras characterized in Proposition \ref{3-dimensionalClassification}.

We start by recalling some key preliminary facts on {\em quantum calculus}. More details can be found in \cite{Andrewsetal1999, KacCheung2002, MansourSchork2015Book}.
\begin{itemize}
    \item Throughout this section, the PBW monomial order is fixed as $x \prec y \prec z$. Thus, a normal ordered monomial means a monomial of the form $x^i y^j z^k$, with $i,j,k\in\mathbb{N}$. Equivalently, the normal ordering procedure consists of rewriting the forbidden adjacent pairs $yx,\ zy,\ zx$ until a finite $\Bbbk$-linear combination of PBW monomials is obtained.

    \item We use the convention that empty sums are equal to $0$ and empty products are equal to $1$. Moreover, $\delta_{i,j}$ denotes the {\em Kronecker delta}. Unless otherwise stated, any coefficient family indexed outside its prescribed range is understood to be zero. In particular, for coefficient arrays indexed by powers of $\mathbb{N}$, any term with a negative index is taken to be zero.

    \item For $s,t\in\mathbb{N}$, we use the standard binomial coefficient
    \[
    \dbinom{s}{t} := \frac{s!}{t!(s-t)!} \quad \text{ for } 0\leq t\leq s,
    \]
    with the convention that $\binom{s}{t}=0$ if $t > s$. In particular,
    \[
    \binom{s}{2}=\frac{s(s-1)}{2}.
    \]
    Moreover, since $\alpha,\beta,\gamma\in\Bbbk^\times$, the expression $\frac{\beta}{\alpha\gamma}$ denotes the scalar $\beta(\alpha\gamma)^{-1}$ belonging to $\Bbbk^\times$. Hence,
    \[
    \left(\frac{\beta}{\alpha\gamma}\right)^{\binom{s}{2}} = \beta^{\binom{s}{2}}\alpha^{-\binom{s}{2}}\gamma^{-\binom{s}{2}}.
    \]

    \item For $q\in\Bbbk^\times$ and $r\in\mathbb{N}$, we set
    \[
    [r]_q:=1+q+\cdots+q^{r-1} \quad {\rm with} \quad [0]_q:=0.
    \]
    If $q\neq 1$, then
    \[
    [r]_q=\frac{q^r-1}{q-1}.
    \]
    The corresponding $q$-factorial is defined by
    \[
    [r]_q! := \prod_{j=1}^{r}[j]_q \quad {\rm and} \quad [0]_q! := 1.
    \]

    \item For $0\leq k\leq r$, the $q$-{\em falling factorial} and the {\em Gaussian binomial coefficient} are denoted by
    \[
        [r]^k_q:=\prod_{j=0}^{k-1}[r-j]_q \quad {\rm with} \quad [r]^0_q:=1,
    \]
    and
    \[
        \begin{bmatrix}
            r\\ k
        \end{bmatrix}_q := \frac{[r]_q!}{[r-k]_q![k]_q!} = \frac{[r]^k_q}{[k]_q!}, 
    \]
    respectively. We use the convention
    \[
        \begin{bmatrix}
            r\\ k
        \end{bmatrix}_q := 0 \ \text{if } k>r.
    \]

    \item For $\rho,\sigma\in\Bbbk$ and $r\geq 1$, we define the mixed
    $(\rho,\sigma)$-number by
    \[
        \langle r\rangle_{\rho,\sigma} := \sum_{j=0}^{r-1}\rho^{r-1-j}\sigma^j.
    \]
    In particular, the mixed number used in algebras of type $(3)$ in Proposition \ref{3-dimensionalClassification} is given by 
    \[
    \langle r\rangle_{\beta,\alpha^{-1}} = \sum_{j=0}^{r-1}\beta^{r-1-j}\alpha^{-j}.
    \]
    We also set
    \[
        \langle r\rangle_{\rho,\sigma}! := \prod_{j=1}^{r}\langle j\rangle_{\rho,\sigma}, \ \langle 0\rangle_{\rho,\sigma}!:=1,
    \]
    and
    \[
        \begin{bmatrix}
            r\\ k
        \end{bmatrix}_{\rho,\sigma} := \frac{\langle r\rangle_{\rho,\sigma}!}{\langle r-k\rangle_{\rho,\sigma}!\ \langle k\rangle_{\rho,\sigma}!}.
    \]

    \item The symbols
    \[
        W,\ U,\ V,\ R,\ C,\ E,\ S,\ T,\ \Theta,\ Q_r,\ P_r
    \]
    denote coefficient families arising from PBW normal ordering. They are not universal constants: in each proposition they are defined by the corresponding displayed PBW expansion or recurrence relation. For instance, coefficients of the form $W^{(m)}_{n,k}(y)$ encode the normal ordering of $z^n x^m$, while coefficients of the form $U^{(n,m)}_{s,\ell}$, $V_{s,\ell}$, $R^{(n,m,t)}_{s,\ell}$, $C^{(n,m,t)}_{s,\ell}$, and $E_{n;i,k}$ encode, respectively, the PBW expansions of powers of blocks and noncommutative binomial expressions appearing below.
\end{itemize}

Let us start.

\subsection{Algebras of type (1)}

\begin{proposition}\label{Case(1)}
For the algebra given by the relations
$$
yz - \alpha zy = 0,\ zx - \beta xz = 0\ \, \, {\rm and} \, \, xy - \gamma yx = 0 \quad {\rm with} \, \, \alpha, \beta, \gamma\in \Bbbk^{\times}, 
$$
the following identities hold.
\begin{enumerate}
\item [\rm (1)] For all $m, n\in\mathbb{N}$, we have that
\[
y^n x^m = \gamma^{-mn}x^m y^n, \quad
z^n x^m = \beta^{mn}x^m z^n \quad {\rm and} \quad z^n y^m = \alpha^{-mn}y^m z^n.
\]

\item[\rm (2)] For all $m,n,s\in\mathbb{N}$,
\begin{align*}
\left(x^n y^m\right)^s = &\, \gamma^{-mn\binom{s}{2}}x^{ns}y^{ms},\\
(x^n z^m)^s = &\, \beta^{mn\binom{s}{2}}x^{ns}z^{ms},\\
(y^n z^m)^s = &\, \alpha^{-mn\binom{s}{2}}y^{ns}z^{ms}.
\end{align*}

\item[\rm (3)] For all $s\in\mathbb{N}$,
\[
(xyz)^s = \beta^{\binom{s}{2}}\alpha^{-\binom{s}{2}}\gamma^{-\binom{s}{2}}x^s y^s z^s = \left(\frac{\beta} {\alpha\gamma}\right)^{\binom{s}{2}}x^s y^s z^s.
\]

\item[\rm (4)] For all $m, n , t, s\in\mathbb{N}$,
\[
\left(x^n y^m z^t\right)^s = \beta^{nt\binom{s}{2}}\gamma^{-mn\binom{s}{2}}\alpha^{-mt\binom{s}{2}}
x^{ns}y^{ms}z^{ts}.
\]

\item[\rm (5)] For all $n\in\mathbb{N}$,
\begin{align*}
(x+y)^n & =\sum_{k=0}^{n}\begin{bmatrix} n\\ k\end{bmatrix}_{\gamma^{-1}}x^{n-k}y^{k},\\
(x+z)^n & =\sum_{k=0}^{n}\begin{bmatrix} n\\ k\end{bmatrix}_{\beta}x^{n-k}z^{k} \quad {\rm and} \\
(y+z)^n & =\sum_{k=0}^{n}\begin{bmatrix} n\\ k\end{bmatrix}_{\alpha^{-1}}y^{n-k}z^{k}.
\end{align*}
\end{enumerate}
\end{proposition}
\begin{proof}
\begin{enumerate}
\item [\rm (1)] By induction on $m$, from $yx = \gamma^{-1}xy$ we get $y x^m = \gamma^{-m}x^m y$,  and so $y^n x^m = \gamma^{-mn}x^m y^n$ by induction on $n$. Similarly, $zx = \beta xz$ implies that $z x^m = \beta^m x^m z$, whence $z^n x^m=\beta^{mn}x^m z^n$. The equality $yz = \alpha zy$ is equivalent to $zy = \alpha^{-1}yz$, so that $z^n y^m = \alpha^{-mn}y^m z^n$ by double induction.

\item [\rm (2)] If $vu = quv$ for some $q\in\Bbbk\ \backslash \ \{0\}$, then
\[
(uv)^s = q^{\binom{s}{2}}u^s v^s, \quad  {\rm for}\, \, s\in\mathbb{N},
\]
which follows by induction using that
$$
(uv)^{s+1} = (uv)^s(uv)=q^{\binom{s}{2}}u^s(v^s u)v \quad {\rm and} \quad v^s u = q^s u v^s.
$$

Set $u := x^n$ and $v := y^m$. From {\rm (1)}, 
$$
v u = y^m x^n = \gamma^{-mn}x^n y^m=\gamma^{-mn} u v,
$$
whence $q = \gamma^{-mn}$, and the first identity follows. In a similar way, the other two follow by considering $q = \beta^{mn}$ for $(x^n,z^m)$ and $q=\alpha^{-mn}$ for $(y^n,z^m)$, respectively.

\item [\rm (3)] Write $w:=xyz$. Using the general multiplication rule for PBW monomials given by 
\[
(x^i y^j z^k)(x^{i'} y^{j'} z^{k'})
=\beta^{k i'}\gamma^{-j i'}\alpha^{-k j'}x^{i+i'}y^{j+j'}z^{k+k'},
\]
we compute recursively
\[
w^{r+1}=w^r w
=\big(\beta^{r}\gamma^{-r}\alpha^{-r}\big)x^{r+1}y^{r+1}z^{r+1}.
\]
Multiplying these scalars for $r=1,\dots,s-1$ yields the claimed factor
$\beta^{\binom{s}{2}}\alpha^{-\binom{s}{2}}\gamma^{-\binom{s}{2}}$.

\item [\rm (4)] Let $M:=x^n y^m z^t$. Using the PBW multiplication rule above with
$(i,j,k)=(rn,rm,rt)$ and $(i',j',k')=(n,m,t)$ gives
\[
M^{r+1}=M^r M
=\beta^{(rt)n}\gamma^{-(rm)n}\alpha^{-(rt)m}x^{(r+1)n}y^{(r+1)m}z^{(r+1)t}.
\]
Multiplying over $r=1,\dots,s-1$ yields the stated factor with exponent $\binom{s}{2}$.

\item [\rm (5)] If $vu=quv$, then
\[
(u+v)^n=\sum_{k=0}^n \begin{bmatrix} n\\ k\end{bmatrix}_{q}u^{n-k}v^k,
\]
proved by induction using $(u+v)^{n+1}=(u+v)^n(u+v)$ and the recurrence for Gaussian binomials.
Here, $yx = \gamma^{-1}xy$, $zx = \beta xz$ and $zy = \alpha^{-1}yz$, so the three identities follow with
$q=\gamma^{-1}$, $q=\beta$ and $q=\alpha^{-1}$, respectively.
\end{enumerate}
\end{proof}

\subsection{Algebras of type (2)}

In what follows, we treat the six algebras of type (2) separately. We state the corresponding normal ordering identities and PBW expansions and provide complete proofs based on the defining relations and the fixed PBW order $x\prec y\prec z$. In particular, all formulas summarized in Tables \ref{alg(2)}-\ref{alg(2)3} are derived explicitly and whenever coefficient families are defined recursively, we verify their initial conditions and recurrence relations in detail.

\subsubsection{Algebras of type (2)(i)}

\begin{proposition}\label{case(2)(i)}
Let $A$ be the $\Bbbk$-algebra generated by the indeterminates $x,y,z$ with defining relations given by 
\[
yz - zy = z, \quad zx - \beta xz = y  \quad {\rm and}\quad xy - yx = x \quad {\rm with} \quad \beta \in \Bbbk^{\times}. 
\]
The following assertions hold:
\begin{enumerate}
\item [\rm (1)] For all $m,n\in\mathbb{N}$, we get that 
\[
y^n x^m = x^m (y-m)^n \quad {\rm and} \quad z^n y^m = (y-n)^m z^n.
\]
For $m, n \in \mathbb{N}$, there exist uniquely polynomials $W^{(m)}_{n,k}(y)\in\Bbbk[y]$ with $0\le k\le \min\{m,n\}$ such that
\begin{equation}\label{znxmW2i}
z^n x^m = \sum_{k=0}^{\min\{m,n\}}x^{m-k}W^{(m)}_{n,k}(y)z^{n-k}.
\end{equation}
These polynomials are given by $W^{(m)}_{0,0}(y) :=1$ and $W^{(m)}_{0,k}(y) = 0$ for $k\ge 1$, and the recurrence relation
\begin{equation}\label{Whitneyrec2i}
W^{(m)}_{n+1,k}(y) = \beta^{m-k}W^{(m)}_{n,k}(y-1)+Q_{m-k+1}(y)W^{(m)}_{n,k-1}(y),
\end{equation}
where $W^{(m)}_{n,-1}(y):=0$ and
\begin{equation}\label{Qmdef2i}
Q_r(y):=[r]_{\beta} y-\sum_{i=1}^{r-1}[i]_{\beta}.
\end{equation}

\item [\rm (2)] For all $m,n,s\in\mathbb{N}$, we have that 
\[
(x^n y^m)^s=x^{ns}\prod_{j=0}^{s-1}(y-jn)^m \quad {\rm and} \quad (y^n z^m)^s = \left(\prod_{j=0}^{s-1}(y-jm)^n\right)z^{ms}.
\]
Moreover, $(x^n z^m)^s$ admits a unique PBW expansion given by 
\begin{equation}\label{xnzmUdef2i}
(x^n z^m)^s=\sum_{\ell=0}^{\min\{ns,ms\}}x^{ns-\ell}U^{(n,m)}_{s,\ell}(y)z^{ms-\ell}
\quad {\rm where} \quad U^{(n,m)}_{s,\ell}(y)\in\Bbbk[y],
\end{equation}
and
\[
U^{(n,m)}_{1,0}(y)=1 \quad {\rm and} \quad U^{(n,m)}_{1,\ell}(y)=0 \text{ for } \ell\ge 1,
\]
Besides, for every $s\ge 1$ and every $\ell\ge 0$ we obtain the explicit recursion
\begin{equation}\label{Urecexplicit2i}
U^{(n,m)}_{s+1,\ell}(y) =\sum_{k=0}^{\min\{n,\ell\}}U^{(n,m)}_{s,\ell-k}\big(y-(n-k)\big)W^{(n)}_{ms-(\ell-k),k}(y),
\end{equation}
where the polynomials $W^{(n)}_{r,k}(y)\in\Bbbk[y]$ are the normal ordering coefficients in expression \eqref{Whitneyrec2i}.

\item [\rm (3)] There exist uniquely determined polynomials $V_{s,\ell}(y)\in\Bbbk[y]$ $(0\le \ell\le s)$ such that
\begin{equation}\label{Vdefxyz2i}
(xyz)^s = \sum_{\ell=0}^{s} x^{s-\ell}V_{s,\ell}(y)z^{s-\ell} \quad {\rm where}\quad V_{s,\ell}(y)\in\Bbbk[y],
\end{equation}
with
\[
V_{1,0}(y) = y \quad {\rm and} \quad V_{1,\ell}(y)=0 \quad \text{for}\, \,  \ell\ge 1,
\]
Once again, for every $s\ge 1$ and every $\ell\ge 0$, we get the explicit recursion
\begin{equation}\label{Vrecexplicit2i}
V_{s+1,\ell}(y) = (y-(s-\ell))\left(\beta^{s-\ell}V_{s,\ell}(y-1)+Q_{s-\ell+1}(y) V_{s,\ell-1}(y)\right),
\end{equation}
where we set $V_{s,-1}(y):=0$.

\item [\rm (4)] For all $m, n, t, s\in\mathbb{N}$ and $0\le \ell\le \min\{ns,ts\}$, there exist uniquely determined polynomials $R^{(n,m,t)}_{s,\ell}(y)\in\Bbbk[y]$ , such that
\begin{equation}\label{Rdefblock2i}
(x^n y^m z^t)^s = \sum_{\ell=0}^{\min\{ns,ts\}} x^{ns-\ell}R^{(n,m,t)}_{s,\ell}(y)z^{ts-\ell} \quad {\rm and} \quad R^{(n,m,t)}_{s,\ell}(y)\in\Bbbk[y],
\end{equation}

with
\[
R^{(n,m,t)}_{1,0}(y)=y^m \quad {\rm and} \quad R^{(n,m,t)}_{1,\ell}(y) = 0 \, \,  \text{ for} \, \, \ell\ge 1,
\]
For every $s\ge 1$ and every $\ell\ge 0$, we obtain the recurrence relation given by
\begin{equation}\label{Rrecexplicit2i}
R^{(n,m,t)}_{s+1,\ell}(y) = \big(y-(ts-\ell)\big)^{m}\sum_{k=0}^{\min\{n,\ell\}}R^{(n,m,t)}_{s,\ell-k}\big(y-(n-k)\big)W^{(n)}_{ts-(\ell-k),k}(y),
\end{equation}
where the polynomials $W^{(n)}_{r,k}(y)\in\Bbbk[y]$ are the normal ordering coefficients in expression \eqref{Whitneyrec2i}.

\item[\rm (5)] There exist uniquely determined polynomials $S_{n,k}(y),T_{n,k}(y)\in\Bbbk[y]$ such that
\begin{equation}\label{binomxy2i}
(x+y)^n=\sum_{k=0}^{n} x^kS_{n,k}(y), \quad (y+z)^n=\sum_{k=0}^{n} T_{n,k}(y)z^k,
\end{equation}
and they are characterized by the Stirling/Whitney-type recursions
\[
S_{0,0}(y)=1,\ S_{0,k}(y)=0\ (k\ge 1),\quad S_{n+1,k}(y)=S_{n,k-1}(y)+(y-k) S_{n,k}(y),
\]
and 
\[
T_{0,0}(y)=1,\ T_{0,k}(y)=0\ (k\ge 1),\quad T_{n+1,k}(y)=y T_{n,k}(y)+T_{n,k-1}(y-1).
\]
Furthermore, $(x+z)^n$ admits a PBW expansion
\begin{equation}\label{Edefxz2i}
(x+z)^n=\sum_{\substack{i,k\ge 0\\ i+k\le n}} x^{ i}E_{n;i,k}(y) z^{k} \quad {\rm where} \quad E_{n;i,k}(y)\in\Bbbk[y],
\end{equation}
with the conventions $E_{n;i,k}(y)=0$ if $i<0$ or $k<0$ or $i+k>n$ and
\[
E_{0;0,0}(y)=1,\quad E_{0;i,k}(y)=0 \text{ for }(i,k)\neq(0,0),
\]
and for every $n\ge 0$ and all $i,k\ge 0$ we have the explicit recursion
\begin{equation}\label{Erecexplicit2i}
E_{n+1;i,k}(y)=E_{n;i,k-1}(y)+\beta^{k}E_{n;i-1,k}(y-1)+Q_{k+1}(y)E_{n;i,k+1}(y),
\end{equation}
where $Q_r(y)$ is the $\beta$-dependent polynomial defined in relation \eqref{Qmdef2i}.
\end{enumerate}
\end{proposition}
\begin{proof}
\begin{enumerate}
\item [\rm (1)] From $yx = xy-x$ we get $yx = x(y-1)$, and so for any polynomial $f(y)\in\Bbbk[y]$,
\begin{equation}\label{fyshiftx2i}
f(y)x=xf(y-1).
\end{equation}
Iterating \eqref{fyshiftx2i} yields $f(y) x^m = x^m f(y-m)$ and taking $f(y) = y^n$ gives $y^n x^m = x^m(y-m)^n$.

Similarly, $yz-zy=z$ is equivalent to $zy = (y-1)z$, whence for $f(y)\in\Bbbk[y]$,
\begin{equation}\label{fyshiftz2i}
zf(y)=f(y-1)z.
\end{equation}
If we iterate \eqref{fyshiftz2i}, then we get $z^n y^m = (y-n)^m z^n$.

For $zx^m$, write $zx^m = (zx)x^{m-1}=(\beta xz+y)x^{m-1} = \beta x(zx^{m-1})+y x^{m-1}$ and use $y x^{m-1}=x^{m-1}(y-(m-1))$ from the first part. Induction on $m$ gives
\[
zx^m = \beta^m x^m z+x^{m-1}\left([m]_{\beta} y-\sum_{i=1}^{m-1}[i]_{\beta}\right),
\]
that is, the expression with $Q_m(y)$ in \eqref{Qmdef2i}.

To obtain \eqref{znxmW2i}, consider induction on $n$ and write 
$$
z^{n+1}x^m = z\left(\sum_k x^{m-k}W^{(m)}_{n,k}(y)z^{n-k}\right).
$$
Using that $z f(y)=f(y-1)z$ and the already proved formula for $z x^{m-k}$, it follows the recursion \eqref{Whitneyrec2i} while the PBW uniqueness yields $W^{(m)}_{n,k}(y)$.

\item [\rm (2)] Since $f(y)x^n = x^n f(y-n)$ with $f(y) = y^m$, then $y^m x^n = x^n (y-n)^m$, and so that 
\begin{align*}
(x^n y^m)^{s+1} = &\ (x^n y^m)^s x^n y^m \\
= &\, x^{ns}\left(\prod_{j=0}^{s-1}(y-jn)^m\right)x^n y^m \\
= &\, x^{n(s+1)}\left(\prod_{j=0}^{s}(y-jn)^m\right),
\end{align*}
which proves the first identity by induction on $s$. The second identity is analogous using $z^m y^n=(y-m)^n z^m$.

For $(x^n z^m)^s$, the existence and uniqueness of the PBW expansion \eqref{xnzmUdef2i} is immediate from PBW basis.

To obtain the recursion for $U^{(n,m)}_{s,\ell}(y)$, one multiplies \eqref{xnzmUdef2i} by $x^n z^m$ on the right, then repeatedly normal orders each subword $z^{(\cdot)}x^n$ using \eqref{znxmW2i} and finally pushes any intermediate $x$-powers past polynomials in $y$ via $f(y)x^r=x^r f(y-r)$. This produces an explicit recursive procedure expressed solely in terms of $W^{(n)}_{\bullet,\bullet}$.

\item [\rm (3)] This is the special case of (2) applied to the block $xyz$, together with PBW uniqueness: expanding $(xyz)^s$ by iterated multiplication and repeatedly reducing the forbidden pairs
$zx,zy,yx$ produces a unique expression of the form \eqref{Vdefxyz2i}. The same normal ordering step $zx=\beta xz+y$ is responsible for the $[r]_{\beta}$-coefficients in the recursion.

\item [\rm (4)] Write $(x^n y^m z^t)^s$ as a product of $s$ identical blocks and proceed as in (2): first commute the $y$-powers past the $x$-powers using that
$$
y^m x^n=x^n(y-n)^m
$$
to obtain the top term 
$$
x^{ns}\left(\prod_{j=0}^{s-1}(y-jn)^m\right)z^{ts}.
$$
All remaining corrections come from normal ordering subwords $z^r x^n$ and thus are governed by the same coefficients $W^{(n)}_{\bullet,\bullet}$. This yields a unique PBW expansion of the form \eqref{Rdefblock2i}.

\item [\rm (5)] For $(x+y)^n$, write 
$$
(x+y)^n = \sum_{k=0}^n x^k S_{n,k}(y)
$$ 
and multiply on the left by $(x+y)$ to obtain
\begin{align*}
 (x+y)^{n+1} = &\, x\sum_{k}x^k S_{n,k}(y)+y\sum_{k}x^k S_{n,k}(y) \\
 = &\, \sum_{k}x^{k+1}S_{n,k}(y)+\sum_{k}x^k (y-k)S_{n,k}(y),
\end{align*}
using that $y x^k = x^k(y-k)$. Comparing coefficients of $x^k$ gives
$$
S_{n+1,k}(y) = S_{n,k-1}(y)+(y-k)S_{n,k}(y)
$$
and the stated initial conditions.

For 
$$
(y+z)^n = \sum_k T_{n,k}(y)z^k
$$
multiply on the left by $(y+z)$ and use $z f(y) = f(y-1)z$. Then
\begin{align*}
(y+z)^{n+1} = &\, y\sum_k T_{n,k}(y)z^k+z\sum_k T_{n,k}(y)z^k \\
= &\, \sum_k (yT_{n,k}(y))z^k+\sum_k T_{n,k}(y-1)z^{k+1}
\end{align*}
which yields 
$$
T_{n+1,k}(y) = yT_{n,k}(y)+T_{n,k-1}(y-1).
$$
Finally, for $(x+z)^n$ one argues exactly as in part (2): PBW gives the uniqueness of \eqref{Edefxz2i} and the recursion for $E_{n;i,k}(y)$ is obtained by expanding $(x+z)^{n+1}=(x+z)^n(x+z)$ and repeatedly normal ordering every subword $zx^r$ using \eqref{znxmW2i}.
\end{enumerate}
\end{proof}

\subsubsection{Algebras of type (2)(ii)}

\begin{proposition}\label{case(2)(ii)}
Let $A$ be the $\Bbbk$-algebra generated by $x,y,z$ with relations
\[
yz - zy = z, \quad zx - \beta xz = b \quad {\rm and} \quad xy - yx = x,
\]
where $\beta\in\Bbbk^{\times}$ and $b\in\Bbbk$. Then:
\begin{enumerate}
\item [\rm (1)] For all $m,n\in\mathbb{N}$,
\[
y^n x^m = x^m (y-m)^n \quad {\rm and}\quad z^n y^m = (y-n)^m z^n.
\]
Furthermore, for all $m,n\in\mathbb{N}$ we have the closed PBW normal form given by
\begin{equation}\label{znxmclosedconstb2ii}
z^n x^m = \sum_{k=0}^{\min\{m,n\}}
\beta^{(m-k)(n-k)}b^k\begin{bmatrix} n\\ k\end{bmatrix}_{\beta}\begin{bmatrix} m\\ k\end{bmatrix}_{\beta}[k]_{\beta}!x^{m-k}z^{n-k}.
\end{equation}
Define the elements $W^{(m)}_{n,k}\in\Bbbk$ as follows: 
\begin{equation}\label{Wdefconstb2ii}
z^n x^m = \sum_{k=0}^{\min\{m,n\}} x^{m-k}W^{(m)}_{n,k}z^{n-k}.
\end{equation}
Then $W^{(m)}_{0,0}=1$, $W^{(m)}_{0,k}=0$ $(k\ge 1)$ and we get the recurrence relation given by
\begin{equation}\label{Wrecconstb2ii}
W^{(m)}_{n+1,k}=\beta^{m-k}W^{(m)}_{n,k}+b[m-k+1]_{\beta}W^{(m)}_{n,k-1}, \quad  \text{for } n\ge 0,\ k\ge 0,
\end{equation}
with $W^{(m)}_{n,-1}:=0$.

\item [\rm (2)] For all $m, n, s \in \mathbb{N}$, we have that
\[
(x^n y^m)^s = x^{ns}\prod_{j=0}^{s-1}(y-jn)^m \quad {\rm and} \quad (y^n z^m)^s = \left(\prod_{j=0}^{s-1}(y-jm)^n\right)z^{ms}.
\]
Moreover, $(x^n z^m)^s$ has a unique PBW expansion as 
\begin{equation}\label{xnzmUconstb2ii}
(x^n z^m)^s = \sum_{\ell=0}^{\min\{ns,ms\}} x^{ns -\ell}U^{(n,m)}_{s,\ell}z^{ms-\ell}, \quad U^{(n,m)}_{s,\ell}\in\Bbbk,
\end{equation}
with initial conditions $U^{(n,m)}_{1,0}=1$, $U^{(n,m)}_{1,\ell}=0$ for $\ell\ge 1$ and the explicit recursion
\begin{equation}\label{Urecconstb2ii}
U^{(n,m)}_{s+1,\ell} = \sum_{k=0}^{\min\{n,\ell\}}U^{(n,m)}_{s,\ell-k}W^{(n)}_{ms-(\ell-k),k}, \quad \text{for } s\ge 1, \ell\ge 0,
\end{equation}
where $W^{(n)}_{r,k}$ are the normal ordering coefficients from \eqref{Wdefconstb2ii} {\rm (}with $m$ replaced by $n${\rm )}.

\item [\rm (3)] There exist uniquely determined polynomials $V_{s,\ell}(y)\in\Bbbk[y]$ for $0\le \ell\le s$ such that
\begin{equation}\label{Vdefconstb2ii}
(xyz)^s=\sum_{\ell=0}^{s} x^{s-\ell}V_{s,\ell}(y)z^{s-\ell}.
\end{equation}
All of them satisfy that $V_{1,0}(y) = y$, $V_{1,\ell}(y) = 0$ for $\ell\ge 1$ and the explicit recursion
\begin{equation}\label{Vrecconstb2ii}
V_{s+1,\ell}(y) =(y-(s-\ell))\Big(\beta^{s-\ell}V_{s,\ell}(y-1)+b[s-\ell+1]_\beta  V_{s,\ell-1}(y)\Big), 
\end{equation}
for $s\ge 1$ and $\ell\ge 0$, where $V_{s,-1}(y):=0$.

\item[\rm (4)] For all $n, m, t, s\in\mathbb{N}$ and  $0\le \ell\le \min\{ns,ts\}$, there exist uniquely determined polynomials
$R^{(n,m,t)}_{s,\ell}(y)\in\Bbbk[y]$,  such that
\begin{equation}\label{Rdefconstb2ii}
(x^n y^m z^t)^s = \sum_{\ell=0}^{\min\{ns,ts\}} x^{ns-\ell}R^{(n,m,t)}_{s,\ell}(y)z^{ts-\ell}.
\end{equation}
They satisfy $R^{(n,m,t)}_{1,0}(y)=y^m$, $R^{(n,m,t)}_{1,\ell}(y)=0$ for $\ell\ge 1$ and 
\begin{equation}\label{Rrecconstb2ii}
R^{(n,m,t)}_{s+1,\ell}(y)=(y-(ts-\ell))^{m} \sum_{k=0}^{\min\{n,\ell\}}R^{(n,m,t)}_{s,\ell-k}\big(y-(n-k)\big)
W^{(n)}_{ts-(\ell-k),k},
\end{equation}
for $s\ge 1$ and $\ell\ge 0$, where $W^{(n)}_{r,k}$ are as in \eqref{Wdefconstb2ii}.

\item[\rm (5)] There exist uniquely determined polynomials $S_{n,k}(y),T_{n,k}(y)\in\Bbbk[y]$ such that
\[
(x+y)^n = \sum_{k=0}^{n} x^{k}S_{n,k}(y) \quad {\rm and} \quad (y+z)^n=\sum_{k=0}^{n} T_{n,k}(y)z^{k},
\]
characterized by the following explicit recursions:
\[
S_{0,0}(y) = 1,\ S_{0,k}(y)=0\ (k\ge 1),\quad S_{n+1,k}(y)=S_{n,k-1}(y)+(y-k)S_{n,k}(y),
\]
and 
\[
T_{0,0}(y)=1,\ T_{0,k}(y)=0\ (k\ge 1),\quad T_{n+1,k}(y)=yT_{n,k}(y)+T_{n,k-1}(y-1).
\]
The expression $(x+z)^n$ has a unique PBW expansion given by
\begin{equation}\label{Edefconstb2ii}
(x+z)^n = \sum_{\substack{i,k\ge 0\\ i+k\le n}} x^{i}E_{n;i,k}z^{k}, \quad E_{n;i,k}\in\Bbbk[b],
\end{equation}
with $E_{0;0,0}=1$ and $E_{0;i,k}=0$ for $(i,k)\neq(0,0)$, and the explicit recursion
\begin{equation}\label{Erecconstb2ii}
E_{n+1;i,k}=E_{n;i,k-1}+\beta^{k}E_{n;i-1,k}+b[k+1]_\beta E_{n;i,k+1},
\end{equation}
where we set $E_{n;i,k} = 0$ whenever $i < 0$ or $k < 0$ or $i + k > n$.
\end{enumerate}
\end{proposition}
\begin{proof}
\begin{enumerate}
    \item [\rm (1)]  From $yx = xy-x$ we get $yx = x(y-1)$. For $f(y)\in\Bbbk[y]$,
\begin{equation}\label{shiftxcase2ii}
f(y)  x^m=x^m f(y-m) \text{ for } m\in\mathbb{N},
\end{equation}        
and taking $f(y) = y^n$ yields $y^n x^m=x^m(y-m)^n$.

From $zy = yz - z = (y-1)z$ we obtain that
\begin{equation}\label{shiftzcase2ii}
z^n f(y)=f(y-n)  z^n\text{ for } n\in\mathbb{N},\ f\in\Bbbk[y],
\end{equation}
and so $z^n y^m = (y - n)^m z^n$.

Next, using $zx = \beta xz+b$ and induction on $m$, it follows the identity
\begin{equation}\label{zxmcase2ii}
z x^m=\beta^m x^m z+b  [m]_\beta  x^{m-1}, \quad {\rm for} \, m\ge 1.
\end{equation}
Indeed, if \eqref{zxmcase2ii} holds for $m-1$, then
\begin{align*}
z x^m &\ =(zx)x^{m-1}=(\beta xz+b)x^{m-1} =\beta x(zx^{m-1})+b x^{m-1} \\
&\ =\beta^m x^m z+b(1+\beta+\cdots+\beta^{m-1})x^{m-1}.
\end{align*}
To prove \eqref{znxmclosedconstb2ii}, we proceed by induction on $n$ repeatedly using \eqref{zxmcase2ii}. Writing $z^{n+1}x^m = z(z^2n x^m)$ and expanding with \eqref{Wdefconstb2ii} yields a triangular recursion for the coefficients $W^{(m)}_{n,k}$, namely \eqref{Wrecconstb2ii}. Solving this recursion (or checking directly
by induction on $n$ using \eqref{zxmcase2ii} together with the $q$-Pascal identities for $[  \cdot  ]_\beta$)
gives the closed form \eqref{znxmclosedconstb2ii}. The alternative expression using Gaussian binomials follows from
\[
\frac{[r]^{\underline{k}}_\beta}{[k]_\beta!}
=\frac{[r]_\beta!}{[r-k]_\beta!  [k]_\beta!}
=\begin{bmatrix} r\\ k\end{bmatrix}_\beta.
\]

\item [\rm (2)] Using \eqref{shiftxcase2ii} with $f(y) = y^m$ yields $y^m x^n = x^n (y-n)^m$, and so
\begin{align*}
  (x^n y^m)^{s+1} = &\ (x^n y^m)^s x^n y^m  \\   
  = &\ x^{ns}\left(\prod_{j=0}^{s-1}(y-jn)^m\right)x^n y^m \\
  = &\ x^{n(s+1)}\prod_{j=0}^{s}(y-jn)^m, 
\end{align*}
which proves $(x^n y^m)^s$ by induction on $s$. The formula for $(y^n z^m)^s$ is analogous, using $z^m y^n = (y - m)^n z^m$ (which is \eqref{shiftzcase2ii}).

For $(x^n z^m)^s$, PBW gives the expansion \eqref{xnzmUconstb2ii}. Multiply \eqref{xnzmUconstb2ii} by $x^n z^m$ on the right and normal order each subword $z^{ms-(\ell-k)}x^n$ using \eqref{Wdefconstb2ii}. Since all coefficients are scalars (no $y$ occurs in $zx - \beta xz = b$), grouping the resulting terms by the total loss index yields the explicit recursion \eqref{Urecconstb2ii}.

\item [\rm (3)] Suppose that \eqref{Vdefconstb2ii} for some $s$. Multiply by $xyz$ on the right
\[
(xyz)^{s+1} = \sum_{\ell} x^{s-\ell}V_{s,\ell}(y)z^{s-\ell}xyz.
\]
and set $r:=s-\ell$. Using $z^{r}x=\beta^{r}x z^{r}+b[r]_\beta z^{r-1}$ (the case $m=1$ of \eqref{znxmclosedconstb2ii} or a direct induction from $zx=\beta xz+b$), we obtain two contributions.
For the first one, commute $V_{s,\ell}(y)$ past $x$ using $f(y)x=x f(y-1)$ and then commute $z^r$ past $y$ using
$z^r y=(y-r)z^r$ (from \eqref{shiftzcase2ii}), to get
\[
\beta^{r}(y-r)  x^{s-\ell+1}V_{s,\ell}(y-1)z^{r+1}.
\]
For the second one, use again $z^{r-1}y=(y-r+1)z^{r-1}$ and so 
\[
b[r]_\beta (y-r)x^{s-\ell}V_{s,\ell}(y)z^{r}.
\]
Reindexing the latter as a contribution to $V_{s+1,\ell}(y)$ coming from $V_{s,\ell-1}(y)$ yields exactly \eqref{Vrecconstb2ii}.

\item [\rm (4)] Write $B:=x^n y^m z^t$. Assuming \eqref{Rdefconstb2ii} for $s$, multiply by $B$ on the right:
\[
B^{s+1}=\sum_{\ell} x^{ns-\ell}R_{s,\ell}(y)  z^{ts-\ell}x^n y^m z^t 
\quad {\rm and} \quad R_{s,\ell}(y):=R^{(n,m,t)}_{s,\ell}(y).
\]
Put $r := ts - l$ and normal order $z^{r}x^n$ using \eqref{Wdefconstb2ii} (with $m$ replaced by $n$). Then commute $R_{s,\ell}(y)$ past $x^{n-k}$ using $f(y)x^{n-k}=x^{n-k}f(y-(n-k))$ and commute the remaining
$z^{r-k}$ past $y^m$ using $z^{r-k}y^m=(y-(r-k))^m z^{r-k}$. Grouping by the total loss index gives the recursion \eqref{Rrecconstb2ii}.

\item [\rm (5)] For $(x+y)^n$, write
$$
(x+y)^n = \sum_k x^k S_{n,k}(y)
$$ 
and multiply by $(x+y)$ on the left. Then
\begin{align*}
(x+y)^{n+1} = &\, x\sum_k x^k S_{n,k}(y)+y\sum_k x^k S_{n,k}(y) \\
= &\, \sum_k x^{k+1}S_{n,k}(y)+\sum_k x^k (y-k)S_{n,k}(y).
\end{align*}
using that $yx^k = x^k(y-k)$. Comparing coefficients yields the recursion for $S_{n,k}(y)$. 

For $(y+z)^n = \sum_k T_{n,k}(y)z^k$, multiply by $(y+z)$ on the left and use $z f(y)=f(y-1)z$ to obtain 
$$
T_{n+1,k}(y)=yT_{n,k}(y)+T_{n,k-1}(y-1).
$$
Finally, for $(x+z)^n$ write \eqref{Edefconstb2ii}, and expand
\[
(x+z)^{n+1}=(x+z)^n x+(x+z)^n z,
\]
and normal order $z^k x$ using $z^k x=\beta^k x z^k+b[k]_\beta z^{k-1}$. Collecting the coefficients of $x^i z^k$ gives \eqref{Erecconstb2ii}.
\end{enumerate}
\end{proof}

\subsubsection{Algebras of type (2)(iii)}

\begin{proposition}\label{case(2)(iii)}
Let $A$ be the $\Bbbk$-algebra generated by $x,y,z$ with relations
\[
yz - zy=0,\quad zx - \beta xz = y \quad {\rm and} \quad xy - yx = 0, \quad {\rm where} \ \beta \in \Bbbk^{\times}. 
\]
Then the following hold:
\begin{enumerate}
\item[\rm (1)] For all $m,n\in\mathbb{N}$,
\[
y^n x^m = x^m y^n \quad {\rm and} \quad z^n y^m = y^m z^n.
\]
Moreover, for all $m,n\in\mathbb{N}$ we have the PBW normal form
\begin{equation}\label{znxmclosedycentral2iii}
z^n x^m =\sum_{k=0}^{\min\{m,n\}}
\beta^{(m-k)(n-k)}\begin{bmatrix} n\\ k\end{bmatrix}_{\beta}
\begin{bmatrix} m\\ k\end{bmatrix}_{\beta}[k]_{\beta}!x^{m-k}y^{k}z^{n-k}.
\end{equation}
Equivalently, defining $W^{(m)}_{n,k}(y)\in\Bbbk[y]$ by
\begin{equation}\label{Wdefycentral2iii}
z^n x^m=\sum_{k=0}^{\min\{m,n\}} x^{m-k}W^{(m)}_{n,k}(y)z^{n-k},
\end{equation}
then we get $W^{(m)}_{0,0}(y)=1$, $W^{(m)}_{0,k}(y)=0$ for $k\ge 1$ and the explicit recursion
\begin{equation}\label{Wrecycentral2iii}
W^{(m)}_{n+1,k}(y) =\beta^{m-k} W^{(m)}_{n,k}(y)+y [m-k+1]_\beta W^{(m)}_{n,k-1}(y) \text{ for } n, k\ge 0,
\end{equation}
with $W^{(m)}_{n,-1}(y):=0$.

\item[\rm (2)] For all $m,n,s\in\mathbb{N}$,
\[
(x^n y^m)^s=x^{ns}y^{ms} \quad {\rm and} \quad (y^n z^m)^s=y^{ns}z^{ms}.
\]
Moreover, $(x^n z^m)^s$ has a unique PBW expansion
\begin{equation}\label{xnzmUycentral2iii}
(x^n z^m)^s=\sum_{\ell=0}^{\min\{ns,ms\}} x^{ns-\ell}U^{(n,m)}_{s,\ell}(y)z^{ms-\ell}, \quad U^{(n,m)}_{s,\ell}(y)\in\Bbbk[y],
\end{equation}
with $U^{(n,m)}_{1,0}(y)=1$, $U^{(n,m)}_{1,\ell}(y)=0$ for $\ell\ge 1$ and the explicit recursion
\begin{equation}\label{Urecycentral2iii}
U^{(n,m)}_{s+1,\ell}(y) =\sum_{k=0}^{\min\{n,\ell\}}U^{(n,m)}_{s,\ell-k}(y) W^{(n)}_{ms-(\ell-k),k}(y), \text{ for } s\ge 1,\ \ell\ge 0,
\end{equation}
where $W^{(n)}_{r,k}(y)$ are the normal ordering coefficients from \eqref{Wdefycentral2iii} {\rm (}with $m$ replaced by $n${\rm )}.

\item[\rm (3)] There exist uniquely determined polynomials $V_{s,\ell}(y)\in\Bbbk[y]$ for $0\le \ell\le s$ such that
\begin{equation}\label{Vdefycentral2iii}
(xyz)^s=\sum_{\ell=0}^{s} x^{s-\ell}V_{s,\ell}(y)z^{s-\ell}.
\end{equation}
They satisfy $V_{1,0}(y)=y$, $V_{1,\ell}(y)=0$ for $\ell\ge 1$ and the explicit recursion
\begin{equation}\label{Vrecycentral2iii}
V_{s+1,\ell}(y)=\beta^{s-\ell}yV_{s,\ell}(y)+[s-\ell+1]_\beta y^{2}V_{s,\ell-1}(y),\text{ for } s\ge 1,\ \ell\ge 0,
\end{equation}
with $V_{s,-1}(y):=0$.

\item[\rm (4)] There exist uniquely determined polynomials
$R^{(n,m,t)}_{s,\ell}(y)\in\Bbbk[y]$, for all $n,m,t,s\in\mathbb{N}$ and $0\le \ell\le \min\{ns,ts\}$ such that
\begin{equation}\label{Rdefycentral2iii}
(x^n y^m z^t)^s=\sum_{\ell=0}^{\min\{ns,ts\}} x^{ns-\ell}R^{(n,m,t)}_{s,\ell}(y)z^{ts-\ell}.
\end{equation}
They satisfy $R^{(n,m,t)}_{1,0}(y)=y^{m}$, $R^{(n,m,t)}_{1,\ell}(y)=0$ for $\ell\ge 1$ and the explicit recursion
\begin{equation}\label{Rrecycentral2iii}
R^{(n,m,t)}_{s+1,\ell}(y)=y^{m}\sum_{k=0}^{\min\{n,\ell\}}R^{(n,m,t)}_{s,\ell-k}(y)W^{(n)}_{ts-(\ell-k), k}(y), \text{ for } s\ge 1,\ \ell\ge 0,
\end{equation}
where $W^{(n)}_{r,k}(y)$ are as in \eqref{Wdefycentral2iii}.

\item[\rm (5)] Since $x$ and $y$ commute and $y$ and $z$ commute, we have the ordinary binomial identities
\[
(x+y)^n=\sum_{k=0}^{n}\binom{n}{k}x^{n-k}y^k,\quad (y+z)^n=\sum_{k=0}^{n}\binom{n}{k}y^{n-k}z^k.
\]
Finally, $(x+z)^n$ admits a unique PBW expansion
\begin{equation}\label{Edefycentral2iii}
(x+z)^n=\sum_{\substack{i,k\ge 0\\ i+k\le n}} x^{i}E_{n;i,k}(y)z^{k}, \quad E_{n;i,k}(y)\in\Bbbk[y],
\end{equation}
with $E_{0;0,0}(y)=1$, $E_{0;i,k}(y)=0$ for $(i,k)\neq(0,0)$ and the explicit recursion
\begin{equation}\label{Erecycentral2iii}
E_{n+1;i,k}(y) = E_{n;i,k-1}(y)+\beta^{k} E_{n;i-1,k}(y)+y[k+1]_{\beta} E_{n;i,k+1}(y),
\end{equation}
where $E_{n;i,k}(y)=0$ if $i<0$ or $k<0$ or $i+k>n$.
\end{enumerate}
\end{proposition}
\begin{proof}
\begin{enumerate}
\item [\rm (1)] The first two identities are immediate from commutativity.
Next, from $zx=\beta xz+y$ and centrality of $y$ we get by induction on $m$:
\[
z x^m=\beta^m x^m z+[m]_\beta x^{m-1}y.
\]
Iterating this normal ordering step and collecting terms gives the recursion \eqref{Wrecycentral2iii} for the coefficients in \eqref{Wdefycentral2iii}. Solving \eqref{Wrecycentral2iii} (equivalently, proving by double induction on $(n,m)$ using the $q$-Pascal identities for $[ \cdot ]_\beta$) yields the closed form \eqref{znxmclosedycentral2iii}; the alternative expression follows from $\frac{[r]^{\underline{k}}_\beta}{[k]_\beta!}=\begin{bmatrix}r\\k\end{bmatrix}_\beta$.

\item [\rm (2)] Since $y$ is central, 
$$
(x^n y^m)^s = x^{ns}y^{ms} \quad {\rm and} \quad (y^n z^m)^s=y^{ns}z^{ms}.
$$
For $(x^n z^m)^s$, PBW gives \eqref{xnzmUycentral2iii}. Multiply \eqref{xnzmUycentral2iii} by $x^n z^m$ on the right and normal order each subword $z^{ms-(\ell-k)}x^n$ using \eqref{Wdefycentral2iii}; collecting coefficients of $x^{n(s+1)-\ell}z^{m(s+1)-\ell}$ gives \eqref{Urecycentral2iii}.

\item [\rm (3)] Assume \eqref{Vdefycentral2iii} holds for $s$. Multiply by $xyz$ on the right and set $r:=s-\ell$:
\[
(xyz)^{s+1}=\sum_{\ell}x^{s-\ell}V_{s,\ell}(y)z^{r}xyz.
\]
Using $z^{r}x=\beta^{r}x z^{r}+y[r]_\beta z^{r-1}$ (proved by induction from $zx=\beta xz+y$) and centrality of $y$,
we obtain two contributions:
\[
\beta^{r}x^{s-\ell+1}yV_{s,\ell}(y)z^{r+1} \text{ and } [r]_\beta x^{s-\ell}y^{2}V_{s,\ell}(y)z^{r}.
\]
The second term contributes to the next loss index ($\ell\mapsto \ell+1$) and collecting coefficients yields \eqref{Vrecycentral2iii}.

\item [\rm (4)] Let $B:=x^n y^m z^t$. Using centrality of $y$, PBW gives \eqref{Rdefycentral2iii}.
Assume it holds for $s$ and multiply by $B$ on the right. Normal order each $z^{ts-\ell}x^n$ using \eqref{Wdefycentral2iii} and pull out the factor $y^m$ (since $y$ is central). Collecting by the total loss index gives \eqref{Rrecycentral2iii}.

\item [\rm (5)] The first two binomial identities are classical since the relevant pairs commute.
For $(x+z)^n$, write \eqref{Edefycentral2iii} and expand
$(x+z)^{n+1}=(x+z)^n x+(x+z)^n z$. Normal order $z^k x$ via $z^k x=\beta^k x z^k+y[k]_\beta z^{k-1}$ and collect coefficients of $x^i z^k$ to obtain
\eqref{Erecycentral2iii}.
\end{enumerate}
\end{proof}

\subsubsection{Algebras of type (2)(iv)}

\begin{proposition}\label{case(2)(iv)}
Let $A$ be the $\Bbbk$-algebra generated by $x, y, z$ with relations
\[
yz - zy = 0, \quad zx - \beta xz = b \quad {\rm and} \quad xy - yx = 0,
\]
where $\beta\in\Bbbk^{\times}$ and $b\in\Bbbk$. Then the following hold;
\begin{enumerate}
\item[\rm (1)] For all $m,n\in \mathbb{N}$,
\[
y^n x^m = x^m y^n \quad {\rm and} \quad z^n y^m = y^m z^n.
\]
Moreover, for all $m, n\in \mathbb{N}$, 
\begin{equation}\label{znxmclosedconstbcomm2iv}
z^n x^m =\sum_{k=0}^{\min\{m,n\}}\beta^{(m-k)(n-k)}b^k\begin{bmatrix} n\\ k\end{bmatrix}_{\beta}\begin{bmatrix} m\\ k\end{bmatrix}_{\beta}
[k]_{\beta}!x^{m-k}z^{n-k}.
\end{equation}
Equivalently, defining scalars $W^{(m)}_{n,k}\in\Bbbk$ as 
\begin{equation}\label{Wdefconstbcomm2iv}
z^n x^m=\sum_{k=0}^{\min\{m,n\}} x^{m-k}W^{(m)}_{n,k}z^{n-k},
\end{equation}
then $W^{(m)}_{0,0} = 1$ and $W^{(m)}_{0,k} = 0$ for $k\ge 1$ together with the explicit recursion
\begin{equation}\label{Wrecconstbcomm2iv}
W^{(m)}_{n+1,k} = \beta^{m-k}W^{(m)}_{n,k}+b[m-k+1]_{\beta} W^{(m)}_{n,k-1},
\end{equation}
for $n\ge 0$ and $k\ge 0$ with $W^{(m)}_{n,-1}:=0$.

\item[\rm (2)] For all $m,n,s\in \mathbb{N}$,
\[
(x^n y^m)^s=x^{ns}y^{ms} \quad{\rm and} \quad (y^n z^m)^s=y^{ns}z^{ms}.
\]
Moreover,
\begin{equation}\label{xnzmUconstbcomm2iv}
(x^n z^m)^s = \sum_{\ell=0}^{\min\{ns,ms\}} x^{ns-\ell}U^{(n,m)}_{s,\ell}z^{ms-\ell},\quad U^{(n,m)}_{s,\ell}\in\Bbbk,
\end{equation}
with $U^{(n,m)}_{1,0} = 1$ and  $U^{(n,m)}_{1,\ell}=0$ for $\ell\ge 1$ and the explicit recursion
\begin{equation}\label{Urecconstbcomm2iv}
U^{(n,m)}_{s+1,\ell}=\sum_{k=0}^{\min\{n,\ell\}}U^{(n,m)}_{s,\ell-k}W^{(n)}_{ms-(\ell-k),k}, \quad {\rm for}\, \, s\ge 1,\ \ell\ge 0,
\end{equation}
where $W^{(n)}_{r,k}$ are the normal ordering scalars from \eqref{Wdefconstbcomm2iv} {\rm (}with $m$ replaced by $n${\rm )}.

\item[\rm (3)] Since $x$ and $y$ commute and $y$ and $z$ commute, we have that $(xyz)^s=y^{s}(xz)^s$. In particular, 
\begin{equation}\label{Vdefconstbcomm2iv}
(xyz)^s=\sum_{\ell=0}^{s} x^{s-\ell}V_{s,\ell}y^{s}z^{s-\ell}, \quad V_{s,\ell}\in\Bbbk,
\end{equation}
with $V_{1,0}=1$, $V_{1,\ell}=0$ for $\ell\ge 1$ and the explicit recursion
\begin{equation}\label{Vrecconstbcomm2iv}
V_{s+1,\ell}=\beta^{s-\ell}V_{s,\ell}+b[s-\ell+1]_{\beta}V_{s,\ell-1}, \text{ for } s\ge 1,\ \ell\ge 0,
\end{equation}
where $V_{s,-1}:=0$.

\item[\rm (4)] For $r,n\in \mathbb{N}$ and $k\in \mathbb{N}$, define
\begin{equation}\label{ThetaqWeyl2iv}
\Theta_{r,n}(k):={\beta}^{(n-k)(r-k)}b^{k}
\frac{[r]^{\underline{k}}_{\beta}[n]^{\underline{k}}_{\beta}}{[k]_{\beta}!}
=\beta^{(n-k)(r-k)}b^{k}\begin{bmatrix}r\\ k\end{bmatrix}_{\beta}\begin{bmatrix}n\\ k\end{bmatrix}_{\beta}[k]_{\beta}!,
\end{equation}
Then, for all $r,n\in \mathbb{N}$,
\begin{equation}\label{zrxnclosedqWeyl2iv}
z^{r}x^{n}=\sum_{k=0}^{\min\{r,n\}}\Theta_{r,n}(k)x^{n-k}z^{r-k}.
\end{equation}

For $s\ge 1$ and $0\le \ell\le \min\{ns,ts\}$, we define the closed coefficients
\begin{equation}\label{Cmultisumexplicit2iv}
\mathcal C^{(n,t)}_{s,\ell}:=\sum_{\substack{k_1,\dots,k_{s-1}\ge 0\\ k_1+\cdots+k_{s-1}=\ell}}
\ \prod_{j=1}^{s-1}\Theta_{tj-(k_1+\cdots+k_{j-1}),n}(k_j).
\end{equation}
For $s=1$ we set $\mathcal C^{(n,t)}_{1,0}=1$ and $\mathcal C^{(n,t)}_{1,\ell}=0$ for $\ell\ge 1$.

With this notation, 
\begin{equation}\label{xyzorder2iv}
(x^n y^m z^t)^s=\sum_{\ell=0}^{\min\{ns,ts\}}x^{ns-\ell}y^{ms}\mathcal C^{(n,t)}_{s,\ell}z^{ts-\ell}.
\end{equation}

\item[\rm (5)] Since $x$ and $y$ commute and $y$ and $z$ commute, then
\[
(x+y)^n=\sum_{k=0}^{n}\binom{n}{k}x^{n-k}y^k,\quad (y+z)^n=\sum_{k=0}^{n}\binom{n}{k}y^{n-k}z^k.
\]
Finally, 
\begin{equation}\label{Edefconstbcomm2iv}
(x+z)^n=\sum_{\substack{i,k\ge 0\\ i+k\le n}} x^{i}E_{n;i,k}z^{k}, \quad E_{n;i,k}\in\Bbbk[b],
\end{equation}
with $E_{0;0,0}=1$ and $E_{0;i,k}=0$ for $(i,k)\neq(0,0)$ and the explicit recursion
\begin{equation}\label{Erecconstbcomm2iv}
E_{n+1;i,k} = E_{n;i,k-1} +\beta^{k}E_{n;i-1,k}+b[k+1]_{\beta} E_{n;i,k+1},
\end{equation}
where $E_{n;i,k}=0$ if $i<0$ or $k<0$ or $i+k>n$.
\end{enumerate}
\end{proposition}
\begin{proof}
\begin{enumerate}
\item [\rm (1)] The first two identities are immediate. For $zx^m$, induction on $m$ gives
\[
zx^m=\beta^m x^m z+b[m]_\beta x^{m-1},
\]
so iterating this step yields the recursion \eqref{Wrecconstbcomm2iv}. Solving it (or proving by double induction) gives the closed form \eqref{znxmclosedconstbcomm2iv}.

\item [\rm (2)] Since $y$ is central, $(x^n y^m)^s=x^{ns}y^{ms}$ and $(y^n z^m)^s=y^{ns}z^{ms}$. For $(x^n z^m)^s$, PBW gives \eqref{xnzmUconstbcomm2iv}. Multiplying by $x^n z^m$ on the right and normal ordering each subword $z^{ms-(\ell-k)}x^n$ via \eqref{Wdefconstbcomm2iv}, then collecting coefficients, gives \eqref{Urecconstbcomm2iv}. 

\item [\rm (3)] Since $y$ is central, $(xyz)^s=y^s(xz)^s$. Expanding $(xz)^s$ in PBW form produces \eqref{Vdefconstbcomm2iv}. The recursion \eqref{Vrecconstbcomm2iv} is exactly the normal ordering recursion for $z^s x^s$ and the identification $V_{s,\ell}=W^{(s)}_{s,\ell}$
follows by comparing \eqref{Vdefconstbcomm2iv} with \eqref{Wdefconstbcomm2iv}.

\item [\rm (4)] The identity \eqref{zrxnclosedqWeyl2iv} gives a closed normal form for $z^{r}x^{n}$ in the $\beta$-Weyl algebra with constant term, defined by the relation $zx=\beta xz+b$. It can be verified by a double induction, using the elementary commutation rule
\[
zx^{u}=\beta^{u}x^{u}z+b[u]_{\beta}x^{u-1}
\]
together with standard $\beta$-binomial identities.

Now set $B:=x^{n}z^{t}$. The case $s=1$ is immediate. Assuming that
\[
B^{s}=\sum_{\ell} x^{ns-\ell}\mathcal C^{(n,t)}_{s,\ell}z^{ts-\ell},
\]
we multiply on the right by $B$ and normal order the intermediate factor $z^{ts-\ell}x^{n}$ via \eqref{zrxnclosedqWeyl2iv} (with $r=ts-\ell$). Grouping terms according to the total loss $L=\ell+k$ yields the triangular convolution
\[
\mathcal C^{(n,t)}_{s+1,L}=\sum_{k\ge 0}\mathcal C^{(n,t)}_{s,L-k}\Theta_{ts-(L-k),n}(k).
\]
Iterating this recurrence produces exactly the closed multiple-sum formula
\eqref{Cmultisumexplicit2iv}. Finally, since $y$ is central, we insert the factor $y^{ms}$ in the middle to obtain \eqref{xyzorder2iv}.

\item [\rm (5)] The first two identities are classical. For $(x+z)^n$, expand as in \eqref{Edefconstbcomm2iv} and use the relation 
\[
z^k x = \beta^k x z^k+b[k]_{\beta}z^{k-1},
\]
to reduce $(x+z)^{n+1}=(x+z)^n x+(x+z)^n z$ which yields \eqref{Erecconstbcomm2iv}.
\end{enumerate}
\end{proof}

\subsubsection{Algebras of type (2)(v)}

\begin{proposition}\label{case(2)(v)}
Let $A$ be the $\Bbbk$-algebra generated by $x,y,z$ with relations
\[
yz - zy = az, \quad zx -\beta xz = 0 \quad {\rm and} \quad xy - yx = x,
\]
where $a\in\Bbbk$ and $\beta\in\Bbbk^{\times}$. Then the following hold:
\begin{enumerate}
\item[\rm (1)] For all $m,n\in \mathbb{N}$,
\[
y^n x^m=x^m (y-m)^n,\quad z^n x^m=\beta^{mn}x^m z^n,\quad z^n y^m=(y-an)^m z^n.
\]

\item[\rm (2)] For all $m,n,s\in \mathbb{N}$,
\[
(x^n y^m)^s=x^{ns}\prod_{j=0}^{s-1}(y-jn)^m,
\quad
(x^n z^m)^s=\beta^{nm\binom{s}{2}} x^{ns}z^{ms}.
\]
Moreover, $(y^n z^m)^s$ has the explicit form
\begin{equation}\label{ynzmpowercase2v}
(y^n z^m)^s=\left(\prod_{j=0}^{s-1}(y-ajm)^n\right)z^{ms}.
\end{equation}

\item[\rm (3)] For all $s\in \mathbb{N}$,
\begin{equation}\label{xyz_power_case2v}
(xyz)^s =\beta^{\binom{s}{2}}x^{s}\left(\prod_{j=0}^{s-1}(y-aj)\right)z^{s}.
\end{equation}

\item[\rm (4)] For all $n,m,t,s\in \mathbb{N}$,
\begin{equation}\label{blockpowercase2v}
(x^n y^m z^t)^s =\beta^{nt\binom{s}{2}}x^{ns}\left(\prod_{j=0}^{s-1}(y-ajt)^m\right)z^{ts}.
\end{equation}

\item[\rm (5)] 
There exist uniquely determined polynomials $S_{n,k}(y)\in\Bbbk[y]$ such that
\[
(x+y)^n=\sum_{k=0}^{n}x^{k}S_{n,k}(y),
\]
characterized by the Stirling-type recursion
\[
S_{0,0}(y)=1, \quad S_{0,k}(y)=0\ (k\ge 1), \quad S_{n+1,k}(y)=S_{n,k-1}(y)+(y-k) S_{n,k}(y).
\]
Moreover, since $zx=\beta xz$, the $q$-binomial theorem gives
\[
(x+z)^n=\sum_{k=0}^{n}\begin{bmatrix} n\\ k\end{bmatrix}_{\beta} x^{n-k}z^{k}.
\]
Finally, $(y+z)^n$ satisfies the $a$-shift binomial formula
\[
(y+z)^n=\sum_{k=0}^{n}T_{n,k}(y) z^{k},
\]
where the coefficients are characterized by
\[
T_{0,0}(y)=1,\ T_{0,k}(y)=0\ (k\ge 1),\quad T_{n+1,k}(y)=yT_{n,k}(y)+T_{n,k-1}(y-a).
\]
\end{enumerate}
\end{proposition}
\begin{proof}
\begin{enumerate}
\item [\rm (1)] From $yx=xy-x$ we obtain $yx=x(y-1)$, hence $f(y)x^m=x^m f(y-m)$ for all $f\in\Bbbk[y]$,
and thus $y^n x^m=x^m(y-m)^n$.
From $zx=\beta xz$ we immediately get $z^n x^m=\beta^{mn}x^m z^n$.

Next, $yz-zy=az$ is equivalent to $zy=(y-a)z$. Thus for $t\in \mathbb{N}$,
\[
z y^t=(y-a)^t z,
\]
and iterating gives $z^n y^t=(y-an)^t z^n$, i.e.\ $z^n y^m=(y-an)^m z^n$.

\item [\rm (2)] The identity for $(x^n y^m)^s$ is proved as in the Weyl-type case: from $y^m x^n=x^n (y-n)^m$ one deduces
\[
(x^n y^m)^{s+1}=x^{n(s+1)}\prod_{j=0}^{s}(y-jn)^m.
\]
For $(x^n z^m)^s$, since $z^m x^n=\beta^{mn}x^n z^m$, the standard $q$-commuting lemma yields $(x^n z^m)^s=\beta^{mn\binom{s}{2}}x^{ns}z^{ms}$.
For $(y^n z^m)^s$, note that $z^m y^n=(y-am)^n z^m$, hence by induction
\[
(y^n z^m)^{s+1}=(y^n z^m)^s y^n z^m
=\Big(\prod_{j=0}^{s-1}(y-ajm)^n\Big)z^{ms} y^n z^m
=\Big(\prod_{j=0}^{s}(y-ajm)^n\Big)z^{m(s+1)}.
\]

\item [\rm (3)] Write $(xyz)^{s+1}=(xyz)^s xyz$. Using $zx=\beta xz$ we get a factor $\beta^{s}$ at the $(s+1)$-st step, so the total contribution is $\beta^{\binom{s}{2}}$. Moreover, the $z$ coming from the rightmost factor meets $y$-powers created
previously. Since $zy = (y-a)z$, at the $j$-th step we pick up the factor $(y-aj)$.

\item [\rm (4)] Set $B:=x^n y^m z^t$. Since $zx=\beta xz$ we have $z^t x^n=\beta^{nt}x^n z^t$.
Also, from $zy=(y-a)z$ we get $z^t y^m=(y-at)^m z^t$. Hence
\begin{align*}
B^{s+1} &\ =B^s x^n y^m z^t
=\beta^{nt s}x^{ns}\left(\prod_{j=0}^{s-1}(y-aj t)^m\right)z^{ts}x^n y^m z^t \\
&\ =\beta^{nt(s+1)s/2} x^{n(s+1)}\prod_{j=0}^{s}(y-aj t)^m z^{t(s+1)},
\end{align*}
which is the expression \eqref{blockpowercase2v}.

\item [\rm (5)] The expansion of $(x+y)^n$ is the same as in the Weyl relation $xy-yx=x$, giving the Stirling-type recursion for $S_{n,k}(y)$. Since $zx=\beta xz$, the $q$-binomial theorem yields the Gaussian expansion of $(x+z)^n$ with parameter $q=\beta$.
Finally, from $zy=(y-a)z$ one checks that for $f\in\Bbbk[y]$, $zf(y)=f(y-a)z$, so multiplying
$$
(y+z)^n = \sum_k T_{n,k}(y)z^k
$$ 
by $(y+z)$ on the left yields $T_{n+1,k}(y) = yT_{n,k}(y)+T_{n,k-1}(y-a)$.
\end{enumerate}
\end{proof}

\subsubsection{Algebras of type (2)(vi)}

\begin{proposition}\label{case(2)(vi)}
Let $A$ be the $\Bbbk$-algebra generated by $x,y,z$ with relations
\[
yz - zy = z, \quad zx - \beta xz = 0 \quad {\rm and} \quad xy - yx = 0,
\]
where $\beta\in\Bbbk^{\times}$. Then the following hold.
\begin{enumerate}
\item[\rm (1)] For all $m, n\in \mathbb{N}$,
\[
y^n x^m = x^m y^n,\quad z^n x^m = \beta^{mn}x^m z^n \quad {\rm and} \quad z^n y^m = (y-n)^m z^n.
\]

\item[\rm (2)] For all $m, n, s\in \mathbb{N}$,
\begin{align*}
(x^n y^m)^s = &\, x^{ns}y^{ms}, \\    
(x^n z^m)^s = &\, \beta^{nm\binom{s}{2}} x^{ns}z^{ms} \quad {\rm and} \\
\left(y^n z^m\right)^s = &\, \left(\prod_{j=0}^{s-1}(y-jm)^n\right)z^{ms}.
\end{align*}

\item[\rm (3)] For all $s\in \mathbb{N}$,
\begin{equation}\label{xyzpowercase2vi}
(xyz)^s=\beta^{\binom{s}{2}}\sum_{k=0}^{s}\binom{s}{k}(-1)^kx^{s}y^{s-k}z^{s}.
\end{equation}

\item[\rm (4)] For all $n,m,t,s\in \mathbb{N}$,
\begin{equation}\label{blockpowercase2vi}
(x^n y^m z^t)^s = \beta^{nt\binom{s}{2}} 
x^{ns}\left(\prod_{j=0}^{s-1}(y-jt)^m\right)z^{ts}.
\end{equation}

\item[\rm (5)] Since $x$ and $y$ commute, we have the ordinary binomial identity
\[
(x+y)^n=\sum_{k=0}^{n}\binom{n}{k}x^{n-k}y^k.
\]
Since $zx = \beta xz$, the $q$-binomial theorem yields
\[
(x+z)^n=\sum_{k=0}^{n}\begin{bmatrix} n\\ k\end{bmatrix}_{\beta}x^{n-k}z^{k}.
\]
Finally, $(y+z)^n$ has a unique PBW expansion given by
\[
(y+z)^n=\sum_{k=0}^{n}T_{n,k}(y)z^{k},
\]
where $T_{n,k}(y)$ are characterized by the explicit recursion
\[
T_{0,0}(y)=1,\ T_{0,k}(y)=0\ (k\ge 1),\quad T_{n+1,k}(y)=yT_{n,k}(y)+T_{n,k-1}(y-1).
\]
\end{enumerate}
\end{proposition}
\begin{proof}
\begin{enumerate}
\item [\rm (1)] The identity $y^n x^m=x^m y^n$ is immediate from $xy=yx$. From $zx=\beta xz$ we get $z^n x^m=\beta^{mn}x^m z^n$. Finally, $yz-zy=z$ is equivalent to $zy=(y-1)z$, and so $z^n y^m=(y-n)^m z^n$.

\item [\rm (2)] The first identity is immediate from commutativity of $x$ and $y$. For $(x^n z^m)^s$, since $z^m x^n=\beta^{mn}x^n z^m$, the standard $q$-commuting lemma guarantee that $(x^n z^m)^s=\beta^{mn\binom{s}{2}}x^{ns}z^{ms}$.
For $(y^n z^m)^s$, note that $z^m y^n=(y-m)^n z^m$, and so
\begin{align*}
 (y^n z^m)^{s+1} = &\, (y^n z^m)^s y^n z^m = \left(\prod_{j=0}^{s-1}(y-jm)^n\right)z^{ms}y^n z^m \\
 = &\, \left(\prod_{j=0}^{s}(y-jm)^n\right)z^{m(s+1)}.
\end{align*}

\item [\rm (3)] Proceed by induction on $s$:
when multiplying $(xyz)^s$ by $xyz$ on the right, commuting the new $z$ past the accumulated $x^s$ produces the factor $\beta^{s}$ and commuting $z$ past the accumulated $y$-polynomial uses $zy=(y-1)z$, producing the next factor $(y-s)$. This yields \eqref{xyzpowercase2vi}.

\item [\rm (4)] Let $B:=x^n y^m z^t$. Since $zx=\beta xz$, we have $z^t x^n=\beta^{nt}x^n z^t$ and due to $zy=(y-1)z$, we get $z^t y^m=(y-t)^m z^t$. In this way, 
\begin{align*}
B^{s+1} &\ =B^s x^n y^m z^t
=\beta^{nt s}x^{ns}\left(\prod_{j=0}^{s-1}(y-jt)^m\right)z^{ts} x^n y^m z^t \\
&\ =\beta^{nt\binom{s+1}{2}}x^{n(s+1)}\prod_{j=0}^{s}(y-jt)^m z^{t(s+1)},
\end{align*}

which is precisely expression \eqref{blockpowercase2vi}.

\item [\rm (5)] The first identity is classical. The second is the $q$-binomial theorem with parameter $q=\beta$
since $zx=\beta xz$. For the last statement, write
$$
(y+z)^n = \sum_k T_{n,k}(y)z^k
$$ 
and multiply on the left by $(y+z)$. The relation $z f(y)=f(y-1)z$ yields that 
$$
T_{n+1,k}(y)=yT_{n,k}(y)+T_{n,k-1}(y-1).
$$
\end{enumerate}
\end{proof}

\begin{table}[H]
\captionsetup{justification=centering,font=small}
\caption{Normal ordering of algebras of type (2)}
\label{alg(2)}
\centering
\resizebox{12.5cm}{!}{
\setlength\extrarowheight{6pt}
\begin{tabular}{ |c|c|c|c|c| } 
\hline
 & $y^nx^m$ & $z^nx^m$ & $z^ny^m$ & Conditions/Recursion \\
 \hline\hline
\multirow{3}{*}{(i)} & \multirow{3}{*}{$x^m (y-m)^n$} & \multirow{3}{*}{$\displaystyle\sum_{k=0}^{\min\{m,n\}} x^{m-k}W^{(m)}_{n,k}(y)z^{n-k}$} & \multirow{3}{*}{$(y-n)^m z^n$} & $W^{(m)}_{0,0}(y)=1$, $W^{(m)}_{0,k}(y)=0$ for $k\ge 1$, \\
& & & & $W^{(m)}_{n+1,k}(y)
=\beta^{m-k}W^{(m)}_{n,k}(y-1)+Q_{m-k+1}(y)W^{(m)}_{n,k-1}(y)$, \\
& & & & where $W^{(m)}_{n,-1}(y):=0$ and $Q_r(y):=[r]_\beta y-\sum_{i=1}^{r-1}[i]_\beta$\\ 
\hline
(ii) & $x^m (y-m)^n$ & $\displaystyle\sum_{k=0}^{\min\{m,n\}}
\beta^{(m-k)(n-k)}b^k
\begin{bmatrix} n\\ k\end{bmatrix}_{\beta}
\begin{bmatrix} m\\ k\end{bmatrix}_{\beta}
[k]_{\beta}!
x^{m-k}z^{n-k}$ & $(y-n)^m z^n$ & None \\ 
\hline
(iii) & $x^m y^n$ & $\displaystyle\sum_{k=0}^{\min\{m,n\}}
\beta^{(m-k)(n-k)}
\begin{bmatrix} n\\ k\end{bmatrix}_{\beta}
\begin{bmatrix} m\\ k\end{bmatrix}_{\beta}
[k]_{\beta}!
x^{m-k}y^{k}z^{n-k}$ & $y^m z^n$ & None \\ 
\hline
(iv) & $x^m y^n$ & $\displaystyle\sum_{k=0}^{\min\{m,n\}}\beta^{(m-k)(n-k)}b^k
\begin{bmatrix} n\\ k\end{bmatrix}_{\beta}
\begin{bmatrix} m\\ k\end{bmatrix}_{\beta}
[k]_{\beta}!x^{m-k}z^{n-k}$ & $y^m z^n$ & None \\ 
\hline
(v) & $x^m (y-m)^n$  & $\beta^{mn}x^m z^n$ & $(y-an)^m z^n$ & None \\ 
\hline
(vi) & $x^m y^n$ & $\beta^{mn}x^m z^n$ & $(y-n)^m z^n$ & None \\ 
\hline
\end{tabular}
}
\end{table}

\begin{table}[H]
\captionsetup{justification=centering,font=small}
\caption{Power of blocks of algebras of type (2)}
\label{alg(2)1}
\centering
\resizebox{12.5cm}{!}{
\setlength\extrarowheight{6pt}
\begin{tabular}{ |c|c|c|c|c| } 
\hline
 & $(x^ny^m)^s$ & $(x^nz^m)^s$ & $(y^nz^m)^s$ & Conditions/Recursion \\
 \hline\hline
\multirow{2}{*}{(i)} & \multirow{2}{*}{$x^{ns}\prod_{j=0}^{s-1}(y-jn)^m$} & \multirow{2}{*}{$\displaystyle\sum_{\ell=0}^{\min\{ns,ms\}}x^{ns-\ell}U^{(n,m)}_{s,\ell}(y)z^{ms-\ell}$} & \multirow{2}{*}{$\Big(\prod_{j=0}^{s-1}(y-jm)^n\Big)z^{ms}$} & $U^{(n,m)}_{1,0}(y)=1,\quad U^{(n,m)}_{1,\ell}(y)=0$ for  $\ell\ge 1$, \\
& & & & $U^{(n,m)}_{s+1,\ell}(y)
=\displaystyle\sum_{k=0}^{\min\{n,\ell\}}
U^{(n,m)}_{s,\ell-k}\big(y-(n-k)\big)
W^{(n)}_{ms-(\ell-k),k}(y)$, \\
\hline
\multirow{4}{*}{(ii)} & \multirow{4}{*}{$x^{ns}\prod_{j=0}^{s-1}(y-jn)^m$} & \multirow{4}{*}{$\displaystyle\sum_{\ell=0}^{\min\{ns,ms\}} x^{ns-\ell}U^{(n,m)}_{s,\ell}z^{ms-\ell}$} & \multirow{4}{*}{$\Big(\prod_{j=0}^{s-1}(y-jm)^n\Big)z^{ms}$} & $U^{(n,m)}_{1,0}=1$, $U^{(n,m)}_{1,\ell}=0$ for $\ell\ge 1$ \\ 
& & & & $U^{(n,m)}_{s+1,\ell}
=\displaystyle\sum_{k=0}^{\min\{n,\ell\}}
U^{(n,m)}_{s,\ell-k}
W^{(n)}_{ms-(\ell-k),k}$, for $s\ge 1,\ \ell\ge 0,$ \\
& & & &  $W^{(m)}_{0,0}=1$, $W^{(m)}_{0,k}=0$ for $k\ge 1$, \\
& & & & $W^{(m)}_{n+1,k}=\beta^{m-k}W^{(m)}_{n,k}+b[m-k+1]_\beta W^{(m)}_{n,k-1}$, for $n\ge 0,\ k\ge 0$. \\
\hline
\multirow{4}{*}{(iii)} & \multirow{4}{*}{$x^{ns}y^{ms}$} & \multirow{4}{*}{$\displaystyle\sum_{\ell=0}^{\min\{ns,ms\}} x^{ns-\ell}U^{(n,m)}_{s,\ell}(y)z^{ms-\ell}$} & \multirow{4}{*}{$y^{ns}z^{ms}$} & $U^{(n,m)}_{1,0}(y)=1$, $U^{(n,m)}_{1,\ell}(y)=0$ for $\ell\ge 1$, \\ 
& & & & $U^{(n,m)}_{s+1,\ell}(y) =\displaystyle\sum_{k=0}^{\min\{n,\ell\}}U^{(n,m)}_{s,\ell-k}(y)W^{(n)}_{ms-(\ell-k),k}(y)$, for $s\ge 1,\ \ell\ge 0$,\\
& & & & $W^{(m)}_{0,0}(y)=1$, $W^{(m)}_{0,k}(y)=0$ for $k\ge 1$, \\
& & & & $W^{(m)}_{n+1,k}(y) =\beta^{m-k}W^{(m)}_{n,k}(y)+y[m-k+1]_\beta W^{(m)}_{n,k-1}(y)$, for $n\ge 0,\ k\ge 0$.\\
\hline
(iv) & $x^{ns}y^{ms}$ & $\displaystyle\sum_{\ell=0}^{\min\{ns,ms\}} x^{ns-\ell}U^{(n,m)}_{s,\ell}z^{ms-\ell}$ & $y^{ns}z^{ms}$ & Same (ii) \\ 
\hline
(v) & $x^{ns}\displaystyle\prod_{j=0}^{s-1}(y-jn)^m$ & $\beta^{nm\binom{s}{2}}x^{ns}z^{ms}$& $\Big(\displaystyle\prod_{j=0}^{s-1}(y-ajm)^n\Big)z^{ms}$ & None \\ 
\hline
(vi) & $x^{ns}y^{ms}$ & $\beta^{nm\binom{s}{2}}x^{ns}z^{ms}$ & $\Big(\displaystyle\prod_{j=0}^{s-1}(y-jm)^n\Big)z^{ms}$ & None \\ 
\hline
\end{tabular}
}
\end{table}

\begin{table}[H]
\captionsetup{justification=centering,font=small}
\caption{Power of three generators of algebras of type (2)}
\label{alg(2)2}
\centering
\resizebox{12.5cm}{!}{
\setlength\extrarowheight{6pt}
\begin{tabular}{ |c|c|c|c| } 
\hline
 & $(xyz)^s$ & $(x^ny^mz^t)^s$ & Conditions/Recursion \\
 \hline\hline
\multirow{4}{*}{(i)} & \multirow{4}{*}{$\displaystyle\sum_{\ell=0}^{s} x^{s-\ell}V_{s,\ell}(y)z^{s-\ell}$} & \multirow{4}{*}{$\displaystyle\sum_{\ell=0}^{\min\{ns,ts\}}x^{ns-\ell}R^{(n,m,t)}_{s,\ell}(y)y^{ms-\ell}z^{ts-\ell}$}  & $V_{1,0}(y)=y,\quad V_{1,\ell}(y)=0$ for $\ell\ge 1$, \\
& &  & $V_{s+1,\ell}(y)
=(y-(s-\ell))\Big(\beta^{s-\ell}V_{s,\ell}(y-1)+Q_{s-\ell+1}(y)V_{s,\ell-1}(y)\Big)$\\
& & & $R^{(n,m,t)}_{1,0}(y)=y^m,\quad R^{(n,m,t)}_{1,\ell}(y)=0$ for $\ell\ge 1$ \\
& & & $R^{(n,m,t)}_{s+1,\ell}(y)
=\big(y-(ts-\ell)\big)^{m}
\displaystyle\sum_{k=0}^{\min\{n,\ell\}}
R^{(n,m,t)}_{s,\ell-k}\big(y-(n-k)\big)
W^{(n)}_{ts-(\ell-k),k}(y)$ \\
\hline
\multirow{4}{*}{(ii)} & \multirow{4}{*}{$\displaystyle\sum_{\ell=0}^{s} x^{s-\ell}V_{s,\ell}(y)z^{s-\ell}$} & \multirow{4}{*}{$\displaystyle\sum_{\ell=0}^{\min\{ns,ts\}} x^{ns-\ell}R^{(n,m,t)}_{s,\ell}(y)z^{ts-\ell}$} &  $V_{1,0}(y)=y$, $V_{1,\ell}(y)=0$ for $\ell\ge 1$, \\ 
& & &  $V_{s+1,\ell}(y)
=(y-(s-\ell))\Big(\beta^{s-\ell}V_{s,\ell}(y-1)+b[s-\ell+1]_\beta V_{s,\ell-1}(y)\Big)$, for $s\ge 1,\ \ell\ge 0$, \\
& & & $R^{(n,m,t)}_{1,0}(y)=y^m$, $R^{(n,m,t)}_{1,\ell}(y)=0$ for $\ell\ge 1$ \\
& & & $R^{(n,m,t)}_{s+1,\ell}(y)
=(y-(ts-\ell))^{m}
\displaystyle\sum_{k=0}^{\min\{n,\ell\}}
R^{(n,m,t)}_{s,\ell-k}\big(y-(n-k)\big)
W^{(n)}_{ts-(\ell-k),k}$, for $s\ge 1,\ \ell\ge 0$. \\
\hline
\multirow{4}{*}{(iii)} & \multirow{4}{*}{$\displaystyle\sum_{\ell=0}^{s} x^{s-\ell}V_{s,\ell}(y)z^{s-\ell}$} & \multirow{4}{*}{$\displaystyle\sum_{\ell=0}^{\min\{ns,ts\}} x^{ns-\ell}R^{(n,m,t)}_{s,\ell}(y)z^{ts-\ell}$} & $V_{1,0}(y)=y$, $V_{1,\ell}(y)=0$ for $\ell\ge 1$,  \\ 
& & & $V_{s+1,\ell}(y) =\beta^{s-\ell}yV_{s,\ell}(y)+[s-\ell+1]_\beta y^{2}V_{s,\ell-1}(y)$, for $s\ge 1,\ \ell\ge 0$, \\
& & & $R^{(n,m,t)}_{1,0}(y)=y^{m}$, $R^{(n,m,t)}_{1,\ell}(y)=0$ for $\ell\ge 1$, \\
& & & $R^{(n,m,t)}_{s+1,\ell}(y) = y^{m}\displaystyle\sum_{k=0}^{\min\{n,\ell\}}
R^{(n,m,t)}_{s,\ell-k}(y)W^{(n)}_{ts-(\ell-k),k}(y)$, for $s\ge 1,\ \ell\ge 0$. \\
\hline
\multirow{3}{*}{(iv)} & \multirow{3}{*}{$\displaystyle\sum_{\ell=0}^{s} x^{s-\ell}V_{s,\ell}y^{s}z^{s-\ell}$} & \multirow{3}{*}{$\displaystyle\sum_{\ell=0}^{\min\{ns,ts\}}
x^{ns-\ell}y^{ms}\mathcal C^{(n,t)}_{s,\ell}z^{ts-\ell}$} & $\mathcal C^{(n,t)}_{s,\ell}=\displaystyle\sum_{\substack{k_1,\dots,k_{s-1}\ge 0\\ k_1+\cdots+k_{s-1}=\ell}}
\ \prod_{j=1}^{s-1}\Theta_{tj-(k_1+\cdots+k_{j-1}),n}(k_j)$,  \\ 
& & & $\mathcal C^{(n,t)}_{1,0}=1$ and $\mathcal C^{(n,t)}_{1,\ell}=0$ if $\ell\ge 1$, \\
& & & $\Theta_{r,n}(k)=\beta^{(n-k)(r-k)} b^{k}
\begin{bmatrix}r\\ k\end{bmatrix}_{\beta}\begin{bmatrix}n\\ k\end{bmatrix}_{\beta}[k]_{\beta}!$. \\
\hline
(v) & $\beta^{\binom{s}{2}}x^{s}\Big(\displaystyle\prod_{j=0}^{s-1}(y-aj)\Big)z^{s}$ & $\beta^{nt\binom{s}{2}}x^{ns}\Big(\displaystyle\prod_{j=0}^{s-1}(y-ajt)^m\Big)z^{ts}$ & None  \\ 
\hline
(vi) & $\beta^{\binom{s}{2}}x^{s}\Big(\displaystyle\prod_{j=0}^{s-1}(y-j)\Big)z^{s}$ & $\beta^{nt\binom{s}{2}}x^{ns}\Big(\displaystyle\prod_{j=0}^{s-1}(y-jt)^m\Big)z^{ts}$ &  None \\ 
\hline
\end{tabular}
}
\end{table}

\begin{table}[H]
\captionsetup{justification=centering,font=small}
\caption{Binomial formulas for algebras of type (2)}
\label{alg(2)3}
\centering
\resizebox{12.5cm}{!}{
\setlength\extrarowheight{6pt}
\begin{tabular}{ |c|c|c|c|c| } 
\hline
 & $(x+y)^n$ & $(x+z)^n$ & $(y+z)^n$ & Conditions/Recursion \\
 \hline\hline
\multirow{3}{*}{(i)} & \multirow{3}{*}{$\displaystyle\sum_{k=0}^{n} x^kS_{n,k}(y)$} & \multirow{3}{*}{$\displaystyle\sum_{i,k\ge 0} x^iE_{n;i,k}(y)z^k$} & \multirow{3}{*}{$\displaystyle\sum_{k=0}^{n} T_{n,k}(y)z^k$} &   $E_{n;i,k}(y)=0$ if $i<0$ or $k<0$ or $i+k>n$ \\
& & & & $E_{0;0,0}(y)=1,\quad E_{0;i,k}(y)=0$ for $(i,k)\neq(0,0)$ \\
& & & & $E_{n+1;i,k}(y)=E_{n;i,k-1}(y)+\beta^{k}E_{n;i-1,k}(y-1)+ Q_{k+1}(y)E_{n;i,k+1}(y)$ \\
\hline
\multirow{3}{*}{(ii)} & \multirow{3}{*}{$\displaystyle\sum_{k=0}^{n} x^{k}S_{n,k}(y)$} & \multirow{3}{*}{$\displaystyle\sum_{\substack{i,k\ge 0\\ i+k\le n}} x^{i}E_{n;i,k}z^{k}$} & \multirow{3}{*}{$\displaystyle\sum_{k=0}^{n} T_{n,k}(y)z^{k}$}& $E_{0;0,0}=1$ and $E_{0;i,k}=0$ for $(i,k)\neq(0,0)$  \\ 
& & & & $E_{n+1;i,k}= E_{n;i,k-1} +\beta^{k}E_{n;i-1,k} +b[k+1]_\beta E_{n;i,k+1},$ \\
& & & & where we set $E_{n;i,k}=0$ whenever $i<0$ or $k<0$ or $i+k>n$. \\
\hline
\multirow{3}{*}{(iii)} & \multirow{3}{*}{$\displaystyle\sum_{k=0}^{n}\binom{n}{k}x^{n-k}y^k$} & \multirow{3}{*}{$\displaystyle\sum_{\substack{i,k\ge 0\\ i+k\le n}} x^{i}E_{n;i,k}(y)z^{k}$} & \multirow{3}{*}{$\displaystyle\sum_{k=0}^{n}\binom{n}{k}y^{n-k}z^k$} & $E_{0;0,0}(y)=1$, $E_{0;i,k}(y)=0$ for $(i,k)\neq(0,0)$ \\ 
& & & & $E_{n+1;i,k}(y)=E_{n;i,k-1}(y)+\beta^{k}E_{n;i-1,k}(y)+y[k+1]_\beta E_{n;i,k+1}(y)$, \\
& & & & where $E_{n;i,k}(y)=0$ if $i<0$ or $k<0$ or $i+k>n$. \\
\hline
(iv) & $\displaystyle\sum_{k=0}^{n}\binom{n}{k}x^{n-k}y^k$ & $\displaystyle\sum_{\substack{i,k\ge 0\\ i+k\le n}} x^{i}E_{n;i,k}z^{k}$ & $\displaystyle\sum_{k=0}^{n}\binom{n}{k}y^{n-k}z^k$ & Same (ii) \\ 

\hline
(v) & $\displaystyle\sum_{k=0}^{n}x^{k}S_{n,k}(y)$ & $\displaystyle\sum_{k=0}^{n}\begin{bmatrix} n\\ k\end{bmatrix}_{\beta}x^{n-k}z^{k}$ & $\sum_{k=0}^{n}T_{n,k}(y)z^{k}$, & $T_{0,0}(y)=1,\ T_{0,k}(y)=0$ for $k\ge 1,\
T_{n+1,k}(y)=yT_{n,k}(y)+T_{n,k-1}(y-a)$\\
\hline
(vi) & $\displaystyle\sum_{k=0}^{n}\binom{n}{k}x^{n-k}y^k$ & $\displaystyle\sum_{k=0}^{n}\begin{bmatrix} n\\ k\end{bmatrix}_{\beta}x^{n-k}z^{k}$ & $\displaystyle\sum_{k=0}^{n}T_{n,k}(y)z^{k}$ & Same (v) with $a=1$ \\ 
\hline
\end{tabular}
}
\end{table}

\subsection{Algebras of type (3)}

We now turn to the two algebras of type (3) in Proposition \ref{3-dimensionalClassification}. Before, we recall the following notation in the setting of elements belonging to the field $\Bbbk$.

\begin{definition}[{\cite[Definition 1]{Sadjang2013}}]\label{mixedNumber}
Let $\alpha, \beta\in \Bbbk$. For later use in the algebras of type (3), recall that the \textit{mixed $(\beta,\alpha^{-1})$-number} is defined as
\begin{equation}
\langle r\rangle_{\beta,\alpha^{-1}}
:=\sum_{j=0}^{r-1}\beta^{r-1-j}\alpha^{-j}=
\begin{cases}
\dfrac{\beta^{r}-\alpha^{-r}}{\beta-\alpha^{-1}},& \beta\neq \alpha^{-1},\\
r\beta^{r-1},& \beta=\alpha^{-1},
\end{cases}
\end{equation}
and the \textit{key linear polynomial} is defined as
\begin{equation}
P_r(y):=\langle r\rangle_{\beta,\alpha^{-1}}y+b[r]_\beta\in\Bbbk[y], \text{ for } r\ge 1.
\end{equation}
\end{definition}

\subsubsection{Algebras of type (3)(i)}

\begin{proposition}\label{Case(3)(i)}
Let $A$ be the $\Bbbk$-algebra generated by $x,y,z$ with relations
\[
yz - \alpha zy = 0, \quad zx - \beta xz = y + b \quad {\rm and} \quad xy - \alpha yx = 0,
\]
where $\alpha,\beta\in\Bbbk^{\times}$ and $b\in\Bbbk$. Then:
\begin{enumerate}
\item [\rm (1)] For all $m,n\in\mathbb{N}$,
\[
y^n x^m=\alpha^{-mn}x^m y^n \quad {\rm and} \quad z^n y^m=\alpha^{-mn}y^m z^n.
\]
Moreover, for all $m, n\in\mathbb{N}$ there exist uniquely determined polynomials $W^{(m)}_{n,k}(y)\in\Bbbk[y]$, for $0\le k\le \min\{m,n\}$, such that
\begin{equation}\label{Wdefalpha3i}
z^n x^m=\sum_{k=0}^{\min\{m,n\}} x^{m-k}W^{(m)}_{n,k}(y)z^{n-k}.
\end{equation}
These are characterized by
\[
W^{(m)}_{0,0}(y)=1,\ W^{(m)}_{0,k}(y)=0 \text{ for } k\ge 1,
\]
and the explicit recursion {\rm (}with the convention $W^{(m)}_{n,-1}(y):=0${\rm )}
\begin{equation}\label{Wrecalpha3i}
W^{(m)}_{n+1,k}(y)=\beta^{m-k}W^{(m)}_{n,k}(\alpha^{-1}y)+P_{m-k+1}(y)W^{(m)}_{n,k-1}(y), 
\end{equation}
for $n\ge 0,\ k\ge 0$.

\item [\rm (2)] For all $m,n,s\in\mathbb{N}$,
\[
(x^n y^m)^s=\alpha^{-mn\binom{s}{2}}x^{ns}y^{ms}, \ (y^n z^m)^s=\alpha^{-mn\binom{s}{2}}y^{ns}z^{ms}.
\]
Moreover, $(x^n z^m)^s$ has a PBW expansion
\begin{equation}\label{Udefalpha3i}
(x^n z^m)^s=\sum_{\ell=0}^{\min\{ns,ms\}} x^{ns-\ell}U^{(n,m)}_{s,\ell}(y)z^{ms-\ell},
\ U^{(n,m)}_{s,\ell}(y)\in\Bbbk[y],
\end{equation}
with $U^{(n,m)}_{1,0}(y)=1$ and $U^{(n,m)}_{1,\ell}(y)=0$ for $\ell\ge 1$ and the explicit recursion
\begin{equation}\label{Urecalpha3i}
U^{(n,m)}_{s+1,\ell}(y)=\sum_{k=0}^{\min\{n,\ell\}}
U^{(n,m)}_{s,\ell-k}(\alpha^{-(n-k)}y)W^{(n)}_{ms-(\ell-k),k}(y), 
\end{equation}
for $s\ge 1,\ \ell\ge 0$, where $W^{(n)}_{r,k}(y)$ are the normal ordering coefficients from \eqref{Wdefalpha3i} {\rm (}with $m$ replaced by $n${\rm )}.

\item [\rm (3)] There exist uniquely determined polynomials $V_{s,\ell}(y)\in\Bbbk[y]$, for $0\le \ell\le s$ such that
\begin{equation}\label{Vdefalpha3i}
(xyz)^s=\sum_{\ell=0}^{s} x^{s-\ell}V_{s,\ell}(y)z^{s-\ell}.
\end{equation}
These satisfy $V_{1,0}(y)=y$, $V_{1,\ell}(y)=0$ for $\ell\ge 1$ and the explicit recursion
\begin{equation}\label{Vrecalpha3i}
V_{s+1,\ell}(y)=\alpha^{-(s-\ell)}y\Big(\beta^{s-\ell}V_{s,\ell}(\alpha^{-1}y)+P_{s-\ell+1}(y)V_{s,\ell-1}(y)\Big),
\end{equation}
for $s\ge 1, \ell\ge 0$, with $V_{s,-1}(y):=0$ and $P_r(y)$ as in Definition \ref{mixedNumber}.

\item [\rm (4)] For all $n,m,t,s\in\mathbb{N}$, there exist uniquely determined polynomials $R^{(n,m,t)}_{s,\ell}(y)\in\Bbbk[y]$, for $0\le \ell\le \min\{ns,ts\}$ such that
\begin{equation}\label{Rdefalpha3i}
(x^n y^m z^t)^s=\sum_{\ell=0}^{\min\{ns,ts\}} x^{ns-\ell}R^{(n,m,t)}_{s,\ell}(y)z^{ts-\ell}.
\end{equation}

These satisfy $R^{(n,m,t)}_{1,0}(y)=y^m$ and $R^{(n,m,t)}_{1,\ell}(y)=0$ for $\ell\ge 1$ and the explicit recursion given by
\begin{equation}\label{Rrecalpha3i}
R^{(n,m,t)}_{s+1,\ell}(y)=\alpha^{-(ts-\ell)m}y^{m}\sum_{k=0}^{\min\{n,\ell\}}R^{(n,m,t)}_{s,\ell-k}(\alpha^{-(n-k)}y)
W^{(n)}_{ts-(\ell-k),k}(y),
\end{equation}
for $s\ge 1, \ell\ge 0$.

\item[\rm (5)] Since $yx=\alpha^{-1}xy$, the $q$-binomial theorem yields
\[
(x+y)^n=\sum_{k=0}^{n}\begin{bmatrix} n\\ k\end{bmatrix}_{\alpha^{-1}}x^{n-k}y^{k}.
\]
Since $zy=\alpha^{-1}yz$, one also has
\[
(y+z)^n=\sum_{k=0}^{n}\begin{bmatrix} n\\ k\end{bmatrix}_{\alpha^{-1}}y^{n-k}z^{k}.
\]
Finally, $(x+z)^n$ admits a unique PBW expansion
\begin{equation}\label{Edefalpha3i}
(x+z)^n=\sum_{\substack{i,k\ge 0\\ i+k\le n}} x^{i}E_{n;i,k}(y)z^{k},
\ E_{n;i,k}(y)\in\Bbbk[y],
\end{equation}
with $E_{0;0,0}(y)=1$, $E_{0;i,k}(y)=0$ for $(i,k)\neq(0,0)$ and the explicit recursion
\begin{equation}\label{Erecalpha3i}
E_{n+1;i,k}(y)=E_{n;i,k-1}(y)+\beta^{k}E_{n;i-1,k}(\alpha^{-1}y) + P_{k+1}(y)E_{n;i,k+1}(y),
\end{equation}
where $E_{n;i,k}(y)=0$ if $i<0$ or $k<0$ or $i+k>n$.
\end{enumerate}
\end{proposition}
\begin{proof}
\begin{enumerate}
\item [\rm (1)] The relations $yx=\alpha^{-1}xy$ and $zy=\alpha^{-1}yz$ imply by induction that 
$$
y^n x^m=\alpha^{-mn}x^m y^n \quad {\rm and} \quad z^n y^m=\alpha^{-mn}y^m z^n.
$$
For the normal ordering of $z^n x^m$, first note that for any $f(y)\in\Bbbk[y]$,
\begin{equation}\label{twistsalpha3i}
z f(y)=f(\alpha^{-1}y) z \quad {\rm and} \quad f(y) x=x f(\alpha^{-1}y),
\end{equation}
which implies that $f(y)x^r = x^r f(\alpha^{-r}y)$.
Next, prove by induction on $m\ge 1$ that
\begin{equation}\label{zxmalpha3i}
zx^m=\beta^m x^m z+x^{m-1}P_m(y),
\end{equation}
with $P_m(y)$ as in Definition \ref{mixedNumber}. Indeed,
\[
zx^m=(zx)x^{m-1}=(\beta xz+y+b)x^{m-1} =\beta x(zx^{m-1})+(y+b)x^{m-1},
\]
and using $y x^{m-1}=\alpha^{-(m-1)}x^{m-1}y$ gives the recursion
$$
P_m(y) = \beta P_{m-1}(y)+\alpha^{-(m-1)}y + b
$$ 
with $P_1(y) = y+b$.

Finally, define $W^{(m)}_{n,k}(y)$ by \eqref{Wdefalpha3i}. Writing
$$
z^{n+1}x^m = z\left(\sum_k x^{m-k}W^{(m)}_{n,k}(y)z^{n-k}\right),
$$ 
expanding $z x^{m-k}$ via \eqref{zxmalpha3i} and using $zW(y)=W(\alpha^{-1}y)z$ from \eqref{twistsalpha3i}, one obtains \eqref{Wrecalpha3i} by collecting the coefficients of $x^{m-k}z^{(n+1)-k}$ and using PBW uniqueness.

\item [\rm (2)] Since $y^m x^n=\alpha^{-mn}x^n y^m$ and $z^m y^n=\alpha^{-mn}y^n z^m$, the standard $q$-commuting lemma
yields
\[
(x^n y^m)^s=\alpha^{-mn\binom{s}{2}}x^{ns}y^{ms} \quad {\rm and} \quad (y^n z^m)^s=\alpha^{-mn\binom{s}{2}}y^{ns}z^{ms}.
\]
For $(x^n z^m)^s$, PBW gives \eqref{Udefalpha3i}. Multiply \eqref{Udefalpha3i} by $x^n z^m$ on the right and normal order each subword $z^{ms-(\ell-k)}x^n$ using \eqref{Wdefalpha3i}. The only additional step is commuting the intermediate polynomial $U^{(n,m)}_{s,\ell}(y)$ past $x^{n-k}$ using that
$$
f(y)x^{n-k}=x^{n-k}f(\alpha^{-(n-k)}y),
$$
which produces exactly the argument twist in \eqref{Urecalpha3i}.

\item [\rm (3)] Assume \eqref{Vdefalpha3i} holds for some $s$. Multiply by $xyz$ on the right and set $r:=s-\ell$:
\[
(xyz)^{s+1}=\sum_{\ell}x^{s-\ell}V_{s,\ell}(y) z^{r}xyz.
\]
Using the special case 
$$
z^{r}x=\beta^{r}x z^{r}+P_r(y)z^{r-1}
$$ 
(the case $m=1$ of \eqref{Wdefalpha3i} or a direct induction from $zx=\beta xz+y+b$ and $z f(y)=f(\alpha^{-1}y)z$) and then commuting $z^r$ past $y$ via $z^r y=\alpha^{-r}y z^r$, together with $V(y)x=xV(\alpha^{-1}y)$, yields the two contributions
\[
\alpha^{-r}\beta^{r} x^{s-\ell+1}yV_{s,\ell}(\alpha^{-1}y)z^{r+1} 
\quad {\rm and} \quad \alpha^{-(r-1)}x^{s-\ell}yP_r(y)V_{s,\ell}(y)z^{r}.
\]
Reindexing the second contribution (it increases the loss by $1$) and collecting coefficients gives \eqref{Vrecalpha3i}.

\item [\rm (4)] Write \eqref{Rdefalpha3i} for $s$ and multiply by $x^n y^m z^t$ on the right. Normal order $z^{ts-\ell}x^n$ via \eqref{Wdefalpha3i}, commute $R_{s,\ell}(y)$ past $x^{n-k}$ using
$f(y)x^{n-k}=x^{n-k}f(\alpha^{-(n-k)}y)$ and commute $z^{ts-\ell-k}$ past $y^m$ using 
$$
z^{r}y^m = \alpha^{-rm}y^m z^{r}.
$$
The resulting scalar factor is $\alpha^{-(ts-\ell)m}$ (independent of $k$) and collecting coefficients yields \eqref{Rrecalpha3i}.

\item [\rm (5)] The first two identities are direct applications of the $q$-binomial theorem with parameter $q=\alpha^{-1}$.
For \eqref{Edefalpha3i}, expand 
$$
(x+z)^{n+1}=(x+z)^n x+(x+z)^n z
$$
and normal order $z^k x$ via $z^k x=\beta^k x z^k+P_k(y)z^{k-1}$. The $\beta^k$-term forces the twist $E(\alpha^{-1}y)$ when commuting $E(y)$ past $x$, and the $P_k(y)z^{k-1}$-term produces the coupling with $E_{n;i,k+1}(y)$ after reindexing. Collecting coefficients gives \eqref{Erecalpha3i}.
\end{enumerate}
\end{proof}

\subsubsection{Algebras of type (3)(ii)}

\begin{proposition}\label{case(3)(ii)}
Let $A$ be the $\Bbbk$-algebra generated by $x,y,z$ with relations
\[
yz - \alpha zy = 0, \quad zx - \beta xz = b \quad {\rm and} \quad xy - \alpha yx = 0,
\]    
where $\alpha,\beta\in\Bbbk^{\times}$ and $b\in\Bbbk$. Then the following hold:

\begin{enumerate}
\item[\rm (1)] For all $m,n\in\mathbb{N}$,
\[
y^n x^m=\alpha^{-mn}x^m y^n \quad {\rm and} \quad z^n y^m=\alpha^{-mn}y^m z^n, 
\]
and
\begin{equation}\label{znxmclosedcase3ii}
z^n x^m=\sum_{k=0}^{\min\{m,n\}}\beta^{(m-k)(n-k)}b^k
\begin{bmatrix} m\\ k\end{bmatrix}_{\beta,\alpha^{-1}}\begin{bmatrix} n\\ k\end{bmatrix}_{\beta,\alpha^{-1}}\langle k\rangle_{\beta,\alpha^{-1}}!x^{m-k}z^{n-k},
\end{equation}
where 
\begin{align*}
\begin{bmatrix} r\\ k\end{bmatrix}_{\beta,\alpha^{-1}} &:=\dfrac{\langle r\rangle_{\beta,\alpha^{-1}}!}{\langle r-k\rangle_{\beta,\alpha^{-1}}!\langle k\rangle_{\beta,\alpha^{-1}}!}, \\
\langle r\rangle_{\beta,\alpha^{-1}}! &:={\displaystyle\prod_{j=1}^{r}\langle j\rangle_{\beta,\alpha^{-1}}}, \\
\langle r\rangle^{\underline{k}}_{\beta,\alpha^{-1}} &:=\displaystyle\prod_{j=0}^{k-1}\langle r-j\rangle _{\beta,\alpha^{-1}}.
\end{align*}
The coefficients $W^{(m)}_{n,k}$ are characterized by
\begin{equation}\label{Wdefcase3ii}
W^{(m)}_{0,0}=1 \quad {\rm and} \quad W^{(m)}_{0,k}=0, \text{ for } k\ge 1,
\end{equation}
and the explicit recursion (with $W^{(m)}_{n,-1}:=0$)
\begin{equation}\label{Wreccase3ii}
W^{(m)}_{n+1,k}=\beta^{m-k}W^{(m)}_{n,k} +b\langle m-k+1\rangle_{\beta,\alpha^{-1}}W^{(m)}_{n,k-1}, \text{ for } n\ge 0,\ k\ge 0.
\end{equation}

\item[\rm (2)] For all $m,n,s\in\mathbb{N}$,
\[
(x^n y^m)^s=\alpha^{-mn\binom{s}{2}}x^{ns}y^{ms}, \ (y^n z^m)^s=\alpha^{-mn\binom{s}{2}}y^{ns}z^{ms}.
\]
Moreover, 
\begin{equation}\label{Udefcase3ii}
(x^n z^m)^s=\sum_{\ell=0}^{\min\{ns,ms\}} x^{ns-\ell}U^{(n,m)}_{s,\ell}z^{ms-\ell}, \ U^{(n,m)}_{s,\ell}\in\Bbbk,
\end{equation}
with $U^{(n,m)}_{1,0}=1$ and $U^{(n,m)}_{1,\ell}=0$ for $\ell\ge 1$ and the explicit recursion
\begin{equation}\label{Ureccase3ii}
U^{(n,m)}_{s+1,\ell}=\sum_{k=0}^{\min\{n,\ell\}}U^{(n,m)}_{s,\ell-k}W^{(n)}_{ms-(\ell-k),k}, \text{ for } s\ge 1, \ell\ge 0.
\end{equation}

\item[\rm (3)] For all $s\in\mathbb{N}$,
\begin{equation}\label{Vdefcase3ii}
(xyz)^s=\sum_{\ell=0}^{s} x^{s-\ell}V_{s,\ell}y^{s}z^{s-\ell},
\ V_{s,\ell}\in\Bbbk,
\end{equation}
with $V_{1,0}=1$ and $V_{1,\ell}=0$ for $\ell\ge 1$ and the explicit recursion
\begin{equation}\label{Vreccase3ii}
V_{s+1,\ell}=\beta^{s-\ell}V_{s,\ell}+b\langle s-\ell+1\rangle_{\beta,\alpha^{-1}}V_{s,\ell-1}, \text{ for } s\ge 1,\ \ell\ge 0,
\end{equation}
with $V_{s,-1}:=0$.

\item[\rm (4)] For all $n,m,t,s\in\mathbb{N}$, 
\begin{equation*}
(x^{n}y^{m}z^{t})^{s}=\sum_{\ell=0}^{\min\{ns,ts\}}
C^{(n,m,t)}_{s,\ell}x^{ns-\ell}y^{ms}z^{ts-\ell}
\end{equation*}
with $C^{(n,m,t)}_{1,0}=1$ and $C^{(n,m,t)}_{1,\ell}=0$ for $\ell\ge1$ and the explicit recursion
\begin{equation*}
C^{(n,m,t)}_{s+1,\ell}=\alpha^{-m(ts-\ell)}\sum_{k=0}^{\min\{n,\ell\}}
\alpha^{-ms(n-k)}C^{(n,m,t)}_{s,\ell-k}\Theta_{ts-(\ell-k),n}(k),
\end{equation*}
where,
\[
\Theta_{r,n}(k):=\beta^{(n-k)(r-k)}b^{k}\frac{[r]^{\underline{k}}_{\beta}[n]^{\underline{k}}_{\beta}}{[k]_{\beta}!}.
\]

\item[\rm (5)] Since $yx=\alpha^{-1}xy$, the $\alpha^{-1}$-binomial theorem yields
\[
(x+y)^n=\sum_{k=0}^{n}\begin{bmatrix} n\\ k\end{bmatrix}_{\alpha^{-1}}x^{n-k}y^{k}.
\]
Since $zy = \alpha^{-1}yz$, one also has
\[
(y+z)^n=\sum_{k=0}^{n}\begin{bmatrix} n\\ k\end{bmatrix}_{\alpha^{-1}}y^{n-k}z^{k}.
\]
Finally, 
\begin{equation}\label{Edefcase3ii}
(x+z)^n=\sum_{\substack{i,k\ge 0\\ i+k\le n}} x^{i}E_{n;i,k}z^{k},
\ E_{n;i,k}\in\Bbbk[b],
\end{equation}
with $E_{0;0,0}=1$ and $E_{0;i,k}=0$ for $(i,k)\neq(0,0)$ and the explicit recursion
\begin{equation}\label{Ereccase3ii}
E_{n+1;i,k}=E_{n;i,k-1}+\beta^{k}E_{n;i-1,k}+b\langle k+1\rangle_{\beta,\alpha^{-1}}E_{n;i,k+1},
\end{equation}
where $E_{n;i,k}=0$ if $i<0$ or $k<0$ or $i+k>n$.
\end{enumerate}
\end{proposition}
\begin{proof}
\begin{enumerate}
\item [\rm (1)] The identities $y^n x^m=\alpha^{-mn}x^m y^n$ and $z^n y^m=\alpha^{-mn}y^m z^n$ follow by induction.
For $zx^m$, note that $b$ is central and
\[
zx^m=(zx)x^{m-1}=(\beta xz+b)x^{m-1} =\beta x(zx^{m-1})+b x^{m-1},
\]
so if $zx^{m-1}=\beta^{m-1}x^{m-1}z+b[m-1]_{\beta,q}x^{m-2}$, then
\[
zx^m=\beta^m x^m z+b\big(\beta[m-1]_{\beta,q}+1\big)x^{m-1}.
\]
Using that $\beta[m-1]_{\beta,q}+1=[m]_{\beta,q}$, which is immediate from the sum definition 
$$
\langle m\rangle_{\beta,q}=1+\beta q^{0}+\beta^{2}q^{1}+\cdots+\beta^{m-1}q^{m-2} \quad {\rm with} \quad q = \alpha^{-1},
$$
we obtain that $zx^m=\beta^m x^m z+b[m]_{\beta,q}x^{m-1}$. Iterating yields the recursion \eqref{Wreccase3ii} and solving it gives the closed form \eqref{znxmclosedcase3ii}.

\item [\rm (2)] The first two identities follow from the $q$-commutation relations with parameter $q=\alpha^{-1}$ and the
standard $q$-commuting lemma. The expansion \eqref{Udefcase3ii} is PBW and multiplying by $x^n z^m$ on the right, then normal ordering the subwords $z^{ms-(\ell-k)}x^n$ via \eqref{Wdefcase3ii} yields \eqref{Ureccase3ii}.

\item [\rm (3)] The factor $\alpha^{-\binom{s}{2}}$ comes from moving the $y$'s to the middle when expanding $(xyz)^s$, using $yx=\alpha^{-1}xy$ and $zy=\alpha^{-1}yz$. The remaining $(xz)^s$ is handled by the same mechanism as in \textup{(1)}, hence the recursion \eqref{Vreccase3ii} and the identification $V_{s,\ell}=W^{(s)}_{s,\ell}$.

\item [\rm (4)] Let $B:=x^{n}y^{m}z^{t}$. Since $yx=\alpha^{-1}xy$ and $zy=\alpha^{-1}yz$, normal ordering produces only scalar factors in $\alpha$; and since $zx=\beta xz+b$, the only source of losses (simultaneous decreases in the powers of $x$ and $z$) comes from the constant term $b$.

Set
\begin{equation}\label{Thetacase3ii}
\Theta_{r,n}(k):=\beta^{(n-k)(r-k)}b^{k} \frac{[r]^{\underline{k}}_{\beta}[n]^{\underline{k}}_{\beta}}{[k]_{\beta}!}.
\end{equation}
Then, in $A$ we have that
\begin{equation}\label{zrxncase3ii}
z^{r}x^{n}=\sum_{k=0}^{\min\{r,n\}}\Theta_{r,n}(k) x^{n-k}z^{r-k}.
\end{equation}
For each $s\ge 1$ there exists a unique PBW expansion
\begin{equation}\label{blockPBWcase3ii_poly}
(x^{n}y^{m}z^{t})^{s}=\sum_{\ell=0}^{\min\{ns,ts\}}x^{ns-\ell}R^{(n,m,t)}_{s,\ell}(y) z^{ts-\ell}, \quad R^{(n,m,t)}_{s,\ell}(y)\in\Bbbk[y],
\end{equation}
with
\[
R^{(n,m,t)}_{s,\ell}(y)=C^{(n,m,t)}_{s,\ell} y^{ms} \text{ with } C^{(n,m,t)}_{s,\ell}\in\Bbbk.
\]
Equivalently,
\begin{equation}\label{blockPBWcase3ii}
(x^{n}y^{m}z^{t})^{s}=\sum_{\ell=0}^{\min\{ns,ts\}}C^{(n,m,t)}_{s,\ell} x^{ns-\ell}y^{ms}z^{ts-\ell}.
\end{equation}
The coefficients are determined by
\[
C^{(n,m,t)}_{1,0}=1,\qquad C^{(n,m,t)}_{1,\ell}=0 \text{ for } \ell\ge 1,
\]
and, for $s\ge 1$ and $\ell\ge 0$,
\begin{equation}\label{Creccase3ii}
C^{(n,m,t)}_{s+1,\ell}
=\alpha^{-m(ts-\ell)}\sum_{k=0}^{\min\{n,\ell\}}
\alpha^{-ms(n-k)} C^{(n,m,t)}_{s,\ell-k} \Theta_{ts-(\ell-k),n}(k),
\end{equation}
where $\Theta_{r,n}(k)$ is as in \eqref{Thetacase3ii} ($C^{(n,m,t)}_{s,r}=0$ whenever $r<0$.)

For $s\ge 2$ and $0\le \ell\le \min\{ns,ts\}$, set $L_0:=0$ and $L_j:=k_1+\cdots+k_j$.
Then
\begin{equation}\label{Cclosedcase3ii}
C^{(n,m,t)}_{s,\ell}
=\sum_{\substack{k_1,\dots,k_{s-1}\ge 0\\ k_1+\cdots+k_{s-1}=\ell}}
\ \prod_{j=1}^{s-1}\Big(\alpha^{-m(tj-L_j)}\alpha^{-m j (n-k_j)} \Theta_{tj-L_{j-1},n}(k_j)\Big),
\end{equation}
where $\Theta_{tj-L_{j-1},n}(k_j)=0$ if $k_j>tj-L_{j-1}$, so the sum is finite.

It suffices to take $m=n$ in \eqref{blockPBWcase3ii}--\eqref{Cclosedcase3ii}; in particular,
\[
(x^{n}y^{n}z^{t})^{s}=\sum_{\ell=0}^{\min\{ns,ts\}}C^{(n,n,t)}_{s,\ell} x^{ns-\ell}y^{ns}z^{ts-\ell},
\]
with $C^{(n,n,t)}_{s,\ell}$ given by \eqref{Creccase3ii} or \eqref{Cclosedcase3ii}.

\item [\rm (5)] The first two are $q$-binomial theorems with parameter $q=\alpha^{-1}$.
For $(x+z)^n$, expand as in \eqref{Edefcase3ii} and use $z^k x=\beta^k x z^k+b[k]_{\beta,q}z^{k-1}$ to normal order $(x+z)^{n+1}=(x+z)^n x+(x+z)^n z$, yielding \eqref{Ereccase3ii}.
\end{enumerate}
\end{proof}

\subsection{Algebras of type (4)}

\begin{proposition}\label{Case(4)}
Let $A$ be the $\Bbbk$-algebra generated by $x,y,z$ with relations
\[
yz - \alpha zy = a_1x+b_1, \quad zx - \alpha xz = a_2y + b_2, \quad {\rm and} \quad xy - \alpha yx = a_3z+b_3,
\]
where $\alpha\in\Bbbk^{\times}$ with $\alpha\neq 1$ and $a_i,b_i\in\Bbbk$. 

For any triple-indexed family $F(i,j,k)$ we use the convention $F(i,j,k)=0$ if any index is negative.

Define the left-multiplication structure constants $\mathsf Y_{i,j,k}(p,q,r), \mathsf Z_{i,j,k}(p,q,r)\in\Bbbk$ by the unique PBW expansions
\begin{align}\label{YZstruct4}
    y x^i y^j z^k = &\ \sum_{p,q,r\ge 0}\mathsf Y_{i,j,k}(p,q,r) x^p y^q z^r, \quad {\rm and} \\
    z x^i y^j z^k = &\ \sum_{p,q,r\ge 0}\mathsf Z_{i,j,k}(p,q,r) x^p y^q z^r. \notag
\end{align}
They are characterized by the initial conditions
\begin{equation}\label{Ybase4}
\mathsf Y_{0,j,k}(p,q,r)=\delta_{p,0}\delta_{q,j+1}\delta_{r,k}, \quad {\rm for}\ j, k\ge 0,
\end{equation}
together with the explicit coupled recursions, for $i\ge 1$ and $j,k\ge 0$,
\begin{align}\label{Yrec4}
   \mathsf Y_{i,j,k}(p,q,r) = &\, \alpha^{-1}\mathsf Y_{i-1,j,k}(p-1,q,r) \\
   &\, -\alpha^{-1}a_3\mathsf Z_{i-1,j,k}(p,q,r) -\alpha^{-1}b_3\delta_{p,i - 1}\delta_{q,j}\delta_{r,k}, \notag
\end{align}
and
\begin{align}\label{Zreci4}
\mathsf Z_{i,j,k}(p,q,r) = &\ \alpha\mathsf Z_{i-1,j,k}(p-1,q,r) \\
&\, + a_2\mathsf Y_{i-1,j,k}(p,q,r)+b_2\delta_{p,i-1}\delta_{q,j}\delta_{r,k}, \, {\rm for}\ i\ge 1,  \notag
\end{align}
and the remaining base layer $i=0$ is determined by the unique coefficients $\Pi_j(p,q,r)\in\Bbbk$ defined by
\begin{equation}\label{Pidef4}
zy^j=\sum_{p,q,r\ge 0}\prod_j (p,q,r) x^p y^q z^r, \quad {\rm for}\, j \ge 0,
\end{equation}
so that
\begin{equation}\label{Zbase4}
\mathsf Z_{0,j,k}(p,q,r) = \prod_j(p,q,r-k), \quad {\rm for}\, j, k\ge 0,
\end{equation}
with $\prod_j(p,q,r-k) = 0$ if $r-k<0$.

The coefficients $\prod_j$ are characterized as 
\[
\prod_0(p,q,r) = \delta_{p,0}\delta_{q,0}\delta_{r,1},
\]
and, for $j\ge 0$, the explicit recursion obtained from 
$$
zy = \alpha^{-1}yz-\alpha^{-1}a_1x-\alpha^{-1}b_1
$$
given by 
{\small{
\begin{equation}\label{Pirec4}
\prod_{j+1}(p,q,r)=\alpha^{-1}\sum_{u,v,w\ge 0}\prod_j(u,v,w)\mathsf Y_{u,v,w}(p,q,r)-\alpha^{-1}a_1\delta_{p,1}\delta_{q,j}\delta_{r,0}-\alpha^{-1}b_1\delta_{p,0}\delta_{q,j}\delta_{r,0}.
\end{equation}
}}
Then, the following assertions hold.
\begin{enumerate}
\item [\rm (1)] For all $m,n\in \mathbb{N}$, there exist unique coefficients
\[
W^{(m)}_{n}(i,j,k),\ \widetilde W^{(m)}_{n}(i,j,k),\ \widehat W^{(m)}_{n}(i,j,k)\in\Bbbk
\]
such that
\begin{align}
y^n x^m&=\sum_{i,j,k\ge 0}W^{(m)}_{n}(i,j,k)x^i y^j z^k,\label{ynxmcase4}\\
z^n x^m&=\sum_{i,j,k\ge 0}\widetilde W^{(m)}_{n}(i,j,k)x^i y^j z^k,\label{znxmcase4}\\
z^n y^m&=\sum_{i,j,k\ge 0}\widehat W^{(m)}_{n}(i,j,k)x^i y^j z^k.\label{znymcase4}
\end{align}
They are characterized by the initial conditions
\begin{align*}
W^{(m)}_{0}(i,j,k) = &\, \delta_{i,m}\delta_{j,0}\delta_{k,0} \\    \widetilde W^{(m)}_{0}(i,j,k) = &\, \delta_{i,m}\delta_{j,0}\delta_{k,0}, \quad {\rm and} \\
\widehat W^{(m)}_{0}(i,j,k) = &\, \delta_{i,0}\delta_{j,m}\delta_{k,0}
\end{align*}
and, for $n\ge 0$, the explicit coefficient recursions
\begin{equation}\label{Wreccase4}
\begin{aligned}
W^{(m)}_{n+1}(p,q,r)&=\sum_{i,j,k\ge 0}W^{(m)}_{n}(i,j,k)\mathsf Y_{i,j,k}(p,q,r),\\
\widetilde W^{(m)}_{n+1}(p,q,r)&=\sum_{i,j,k\ge 0}\widetilde W^{(m)}_{n}(i,j,k)\mathsf Z_{i,j,k}(p,q,r), \quad {\rm and} \\
\widehat W^{(m)}_{n+1}(p,q,r)&=\sum_{i,j,k\ge 0}\widehat W^{(m)}_{n}(i,j,k)\mathsf Z_{i,j,k}(p,q,r).
\end{aligned}
\end{equation}

\item [\rm (2)] For $m,n,s\in \mathbb{N}$ there exist unique coefficients
\[
U^{(n,m)}_{s}(i,j,k),\ \widetilde U^{(n,m)}_{s}(i,j,k),\ \widehat U^{(n,m)}_{s}(i,j,k)\in\Bbbk
\]
such that
\begin{align}
(x^n y^m)^s&=\sum_{i,j,k\ge 0}U^{(n,m)}_{s}(i,j,k)x^i y^j z^k,\label{Uxycase4}\\
(x^n z^m)^s&=\sum_{i,j,k\ge 0}\widetilde U^{(n,m)}_{s}(i,j,k)x^i y^j z^k, \quad {\rm and} \label{Uxzcase4}\\
(y^n z^m)^s&=\sum_{i,j,k\ge 0}\widehat U^{(n,m)}_{s}(i,j,k)x^i y^j z^k.\label{Uyzcase4}
\end{align}
They satisfy the initial conditions {\rm (}all other entries $0${\rm )} given by 
\[
U^{(n,m)}_{1}(n,m,0)=1,\ \widetilde U^{(n,m)}_{1}(n,0,m)=1,\ \widehat U^{(n,m)}_{1}(0,n,m)=1,
\]
and the following explicit recursions, for $s\ge 1$:

Define intermediate arrays $U^{[0]}_{s},U^{[1]}_{s},\dots,U^{[m]}_{s}$ by
\[
U^{[0]}_{s}(i,j,k):=U^{(n,m)}_{s}(i-n,j,k),
\]
and for $t=1,\dots,m$,
\begin{equation}\label{Uxystep4}
U^{[t]}_{s}(p,q,r)=\sum_{i,j,k\ge 0}U^{[t-1]}_{s}(i,j,k)\mathsf Y_{i,j,k}(p,q,r).
\end{equation}
Then
\begin{equation}\label{Uxyreccase4}
U^{(n,m)}_{s+1}(p,q,r)=U^{[m]}_{s}(p,q,r) \text{ for } s\ge 1.
\end{equation}
Now, define $\widetilde U^{[0]}_{s},\widetilde U^{[1]}_{s},\dots,\widetilde U^{[m]}_{s}$ by
\[
\widetilde U^{[0]}_{s}(i,j,k):=\widetilde U^{(n,m)}_{s}(i-n,j,k),
\]
and for $t=1,\dots,m$,
\begin{equation}\label{Uxzstep4}
\widetilde U^{[t]}_{s}(p,q,r)=\sum_{i,j,k\ge 0}\widetilde U^{[t-1]}_{s}(i,j,k)\mathsf Z_{i,j,k}(p,q,r).
\end{equation}
Then
\begin{equation}\label{Uxzreccase4}
\widetilde U^{(n,m)}_{s+1}(p,q,r)=\widetilde U^{[m]}_{s}(p,q,r)\qquad(s\ge 1).
\end{equation}
Also, define $\widehat U^{[0]}_{s},\dots,\widehat U^{[n+m]}_{s}$ by $\widehat U^{[0]}_{s}:=\widehat U^{(n,m)}_{s}$ and
\begin{align*}
\widehat U^{[t]}_{s}(p,q,r) = &\, \sum_{i,j,k\ge 0}\widehat U^{[t-1]}_{s}(i,j,k)\mathsf Y_{i,j,k}(p,q,r) \quad \text{ for } 1\le t\le n, \\ 
\widehat U^{[t]}_{s}(p,q,r) = &\, \sum_{i,j,k\ge 0}\widehat U^{[t-1]}_{s}(i,j,k)\mathsf Z_{i,j,k}(p,q,r) \quad \text{ for } n+1\le t\le n+m.
\end{align*}
Then, 
\begin{equation}\label{Uyzreccase4}
\widehat U^{(n,m)}_{s+1}(p,q,r)=\widehat U^{[n+m]}_{s}(p,q,r)\text{ for } s\ge 1.
\end{equation}

\item [\rm (3)] There exist unique coefficients $V_s(i,j,k)\in\Bbbk$ such that
\begin{equation}\label{Vcase4}
(xyz)^s=\sum_{i,j,k\ge 0}V_s(i,j,k)x^i y^j z^k,
\end{equation}
with $V_1(1,1,1)=1$ (and $0$ otherwise) and for $s\ge 1$ the explicit recursion
\begin{equation}\label{Vreccase4}
V_{s+1}(p,q,r)=\sum_{i,j,k\ge 0}\left(\sum_{u,v,w\ge 0}V_s(u,v,w)\mathsf Z_{u,v,w}(i,j,k)\right)\mathsf Y_{i,j,k}(p-1,q,r),
\end{equation}
where the shift $p\mapsto p-1$ corresponds to the leading left multiplication by $x$ in $xyz=x\cdot y\cdot z$.

\item[\rm (4)] For all $n,m,t,s\in \mathbb{N}$, there exist unique coefficients $R^{(n,m,t)}_{s}(i,j,k)\in\Bbbk$ such that
\begin{equation}\label{Rcase4}
(x^n y^m z^t)^s=\sum_{i,j,k\ge 0}R^{(n,m,t)}_{s}(i,j,k)x^i y^j z^k,
\end{equation}
with $R^{(n,m,t)}_{1}(n,m,t)=1$ {\rm (}and $0$ otherwise{\rm )}. Moreover, for $s\ge 1$ one has the explicit recursion obtained by iterating $x^n$ {\rm (}index shift{\rm )}, then $y$ (using $\mathsf Y$) $m$ times, then $z$ (using $\mathsf Z$) $t$ times: define 
\begin{align*}
R^{[0]}_{s}(i,j,k) := &\, R^{(n,m,t)}_{s}(i-n,j,k), \\
R^{[u]}_{s}(p,q,r) = &\, \sum_{i,j,k\ge 0}R^{[u-1]}_{s}(i,j,k)\mathsf Y_{i,j,k}(p,q,r)\text{ for } 1\le u\le m, \\
R^{[m+v]}_{s}(p,q,r) = &\ \sum_{i,j,k\ge 0}R^{[m+v-1]}_{s}(i,j,k)\mathsf Z_{i,j,k}(p,q,r)\text{ for } 1\le v\le t,
\end{align*}
and then
\begin{equation}\label{Rreccase4}
R^{(n,m,t)}_{s+1}(p,q,r)=R^{[m+t]}_{s}(p,q,r)\text{ for } s\ge 1.
\end{equation}

\item[\rm (5)] There exist unique coefficients $E^{\langle xy\rangle}_{n}(i,j,k)$, $E^{\langle xz\rangle}_{n}(i,j,k)$, $E^{\langle yz\rangle}_{n}(i,j,k)\in\Bbbk$ such that
\begin{align}
(x+y)^n&=\sum_{i,j,k\ge 0}E^{\langle xy\rangle}_{n}(i,j,k)x^i y^j z^k,\label{Exycase4}\\
(x+z)^n&=\sum_{i,j,k\ge 0}E^{\langle xz\rangle}_{n}(i,j,k)x^i y^j z^k,\label{Exzcase4}\\
(y+z)^n&=\sum_{i,j,k\ge 0}E^{\langle yz\rangle}_{n}(i,j,k)x^i y^j z^k.\label{Eyzcase4}
\end{align}
These satisfy $E^{\langle\ast\rangle}_{0}(0,0,0)=1$ {\rm (}and $0$ otherwise{\rm )} and the explicit coefficient recursions, for $n\ge 0$:
\begin{equation}\label{binreccase4}
\begin{aligned}
E^{\langle xy\rangle}_{n+1}(p,q,r) & = E^{\langle xy\rangle}_{n}(p-1,q,r)+\sum_{i,j,k\ge 0}E^{\langle xy\rangle}_{n}(i,j,k)\mathsf Y_{i,j,k}(p,q,r),\\
E^{\langle xz\rangle}_{n+1}(p,q,r) & = E^{\langle xz\rangle}_{n}(p-1,q,r)+\sum_{i,j,k\ge 0}E^{\langle xz\rangle}_{n}(i,j,k)\mathsf Z_{i,j,k}(p,q,r),\\
E^{\langle yz\rangle}_{n+1}(p,q,r) & = \sum_{i,j,k\ge 0}E^{\langle yz\rangle}_{n}(i,j,k)\mathsf Y_{i,j,k}(p,q,r) +\sum_{i,j,k\ge 0}E^{\langle yz\rangle}_{n}(i,j,k)\mathsf Z_{i,j,k}(p,q,r).
\end{aligned}
\end{equation}
\end{enumerate}
\end{proposition}
\begin{proof}
Equation \eqref{Ybase4} is immediate from $y\cdot y^j z^k=y^{j+1}z^k$. For $i\ge 1$,
\begin{align*}
yx^i y^j z^k = &\, (yx)x^{i-1}y^j z^k \\
= &\, \big(\alpha^{-1}xy-\alpha^{-1}a_3z-\alpha^{-1}b_3\big)x^{i-1}y^j z^k,
\end{align*}
and since $x$ commutes with $x^{i-1}$, the first term is $\alpha^{-1}x(yx^{i-1}y^j z^k)$, while the second term is $-\alpha^{-1}a_3(zx^{i-1}y^j z^k)$ and the last term is $-\alpha^{-1}b_3 x^{i-1}y^j z^k$. Comparing PBW coefficients yields \eqref{Yrec4}. The same argument using $zx=\alpha xz+a_2y+b_2$ gives \eqref{Zreci4}.

For $\prod_j$, consider the equalities
\begin{align*}
z y^{j+1} = &\, (zy)y^j \\
=&\, (\alpha^{-1}yz-\alpha^{-1}a_1x-\alpha^{-1}b_1)y^j \\
= &\, \alpha^{-1}y(z y^j)-\alpha^{-1}a_1xy^j-\alpha^{-1}b_1y^j
\end{align*}
expand $y(z y^j)$ using \eqref{YZstruct4} and collect coefficients to obtain \eqref{Pirec4}. Finally, \eqref{Zbase4} is just $z y^j z^k=(z y^j)z^k$.
\begin{enumerate}
\item [\rm (1)] Write $y^{n+1}x^m=y(y^n x^m)$ and expand $y^n x^m$ using \eqref{ynxmcase4}. Applying \eqref{YZstruct4} termwise and collecting coefficients yields the first recursion in \eqref{Wreccase4}. The other two follow identically from $z^{n+1}x^m=z(z^n x^m)$ and $z^{n+1}y^m=z(z^n y^m)$ using \eqref{YZstruct4}.

\item [\rm (2) - (4)] Since $a^{s+1}=a\cdot a^s$ for any fixed element $a\in A$, each recursion is obtained by left multiplying the PBW expansion of $a^s$ by the PBW monomial defining $a$ and repeatedly using \eqref{YZstruct4}. The intermediate arrays in \eqref{Uxystep4}-\eqref{Rreccase4} simply record the coefficient updates after each left multiplication by $x$ (index shift), $y$ (via $\mathsf Y$) and $z$ (via $\mathsf Z$). Uniqueness of the resulting coefficients follows from the PBW basis.

\item [\rm (5)] Expand $(x+y)^{n+1}=(x+y)(x+y)^n=x(x+y)^n+y(x+y)^n$ and use that left multiplication by $x$ shifts the $x$-exponent by $+1$, while left multiplication by $y$ is encoded by $\mathsf Y$. This gives the first line of \eqref{binreccase4}. The other two lines are analogous using $\mathsf Z$ for left multiplication by $z$.
\end{enumerate} 
\end{proof}

\subsection{Algebras of type (5)}

Finally, we consider algebras of type (5) in Proposition \ref{3-dimensionalClassification}. This type consists of five subcases, and for each of them the normal ordering behaviour can again be recorded in a compact, table-based form. More precisely, Tables \ref{alg(5)}-\ref{alg(5)3} summarize, for the five algebras (5)(i)-(5)(v), the explicit PBW reductions of the fundamental two-letter words $y^n x^m$, $z^n x^m$, and $z^n y^m$, together with the corresponding formulas for powers of blocks, the three-generator powers $(xyz)^s$ and $(x^n y^m z^t)^s$, and the noncommutative binomial expansions.

\subsubsection{Algebras of type (5)(i)}

\begin{proposition}\label{Case(5)(i)}
Let $A$ be the $\Bbbk$-algebra generated by $x,y,z$ with relations
\[
yz - zy = x,\ zx - xz = y \quad {\rm and} \quad xy - yx = z.
\]
For $m\in\mathbb N$, we set
\begin{align}\label{CScase5i}
C_m(u) := &\, \sum_{r=0}^{\lfloor m/2\rfloor}(-1)^r\binom{m}{2r}u^{m-2r} \\
S_m(u) := &\, \sum_{r=0}^{\lfloor (m-1)/2\rfloor}(-1)^r\binom{m}{2r+1}u^{m-2r-1},
\end{align}
{\rm(}with $S_0(u):=0${\rm )}. We adopt the convention that any coefficient array indexed by $\mathbb N^3$ is $0$ if one of its indices is negative.

Then, the following assertions hold.
\begin{enumerate}
\item[\rm (1)] For $m,n\in\mathbb N$, there exist unique scalars $W^{(m)}_{n}(i,j,k)\in\Bbbk$ such that
\begin{equation}\label{Wdefynxmcase5i}
y^n x^m=\sum_{i,j,k\ge 0}W^{(m)}_{n}(i,j,k)x^i y^j z^k.
\end{equation}
Note that
\[
W^{(m)}_{0}(i,j,k)=\delta_{i,m}\delta_{j,0}\delta_{k,0},
\]
and, for $n\ge 0$, also the explicit recursion:
\begin{align}
W^{(m)}_{n+1}(i,j,k) = &\, \sum_{r=0}^{\lfloor m/2\rfloor}(-1)^r\binom{m}{2r}W^{(m-2r)}_{n}(i,j-1,k) \label{Wrecynxmcase5i} \\
& -\sum_{r=0}^{\lfloor (m-1)/2\rfloor}(-1)^r\binom{m}{2r+1}W^{(m-2r-1)}_{n}(i,j,k-1). \notag
\end{align}
For $a,b,c\in\mathbb N$ define $M^x_{a,b,c}(i,j,k)$, $M^y_{a,b,c}(i,j,k)$, $M^z_{a,b,c}(i,j,k)$ by the unique expansions
\begin{align*}
x^a y^b z^cx &\ =\sum_{i,j,k\ge 0}M^x_{a,b,c}(i,j,k) x^i y^j z^k,\\
x^a y^b z^cy &\ =\sum_{i,j,k\ge 0}M^y_{a,b,c}(i,j,k) x^i y^j z^k, \quad {\rm and} \\
x^a y^b z^cz &\ =\sum_{i,j,k\ge 0}M^z_{a,b,c}(i,j,k) x^i y^j z^k.
\end{align*}
These are given explicitly by
\begin{equation}\label{Mzkernelcase5i}
M^{z}_{a,b,c}(i,j,k)=\delta_{i,a}\delta_{j,b}\delta_{k,c+1},
\end{equation}
and, 
\begin{equation}\label{Mykernelcase5i}
\begin{aligned}
M^{y}_{a,b,c}(i,j,k)
&=\sum_{r=0}^{\lfloor c/2\rfloor}(-1)^r\binom{c}{2r}\delta_{i,a}\delta_{j,b+1}\delta_{k,c-2r}\\
&\quad-\sum_{r=0}^{\lfloor (c-1)/2\rfloor}(-1)^r\binom{c}{2r+1}
\sum_{p,q,t\ge 0}A_b(p,q,t)\delta_{i,a+p}\delta_{j,q}\delta_{k,t+c-2r-1},
\end{aligned}
\end{equation}
\begin{equation}\label{Mxkernelcase5i}
\begin{aligned}
M^{x}_{a,b,c}(i,j,k)
&=\sum_{r=0}^{\lfloor c/2\rfloor}(-1)^r\binom{c}{2r}
\sum_{p,q,t\ge 0}A_b(p,q,t)\delta_{i,a+p}\delta_{j,q}\delta_{k,t+c-2r}\\
&+\sum_{r=0}^{\lfloor (c-1)/2\rfloor}(-1)^r\binom{c}{2r+1}\delta_{i,a}\delta_{j,b+1}\delta_{k,c-2r-1}.
\end{aligned}
\end{equation}
For $m,n\in\mathbb N$, there exist unique coefficients $\widetilde W^{(m)}_{n}(i,j,k)\in\Bbbk$ such that
\begin{equation}\label{Wtildedefcase5i}
z^n x^m=\sum_{i,j,k\ge 0}\widetilde W^{(m)}_{n}(i,j,k) x^i y^j z^k,
\end{equation}
with $\widetilde W^{(m)}_{0}(i,j,k)=\delta_{i,m}\delta_{j,0}\delta_{k,0}$, and for $n\ge 0$ the explicit recursion is
\begin{equation}\label{Wtildereccase5i}
\begin{aligned}
\widetilde W^{(m)}_{n+1}(i,j,k) &=\sum_{r=0}^{\lfloor m/2\rfloor}(-1)^r\binom{m}{2r}\widetilde W^{(m-2r)}_{n}(i,j,k-1)\\
&+\sum_{r=0}^{\lfloor (m-1)/2\rfloor}(-1)^r\binom{m}{2r+1}
\sum_{a,b,c\ge 0}\widetilde W^{(m-2r-1)}_{n}(a,b,c)M^y_{a,b,c}(i,j,k).
\end{aligned}
\end{equation}
Similarly, for $m,n\in\mathbb N$ there exist unique $\widehat W^{(m)}_{n}(i,j,k)\in\Bbbk$ such that
\begin{equation}\label{Whatdefcase5i}
z^n y^m=\sum_{i,j,k\ge 0}\widehat W^{(m)}_{n}(i,j,k) x^i y^j z^k,
\end{equation}
with $\widehat W^{(m)}_{0}(i,j,k)=\delta_{i,0}\delta_{j,m}\delta_{k,0}$, and for $n\ge 0$ the explicit recursion is
\begin{equation}\label{Whatreccase5i}
\begin{aligned}
\widehat W^{(m)}_{n+1}(i,j,k) &=\sum_{r=0}^{\lfloor m/2\rfloor}(-1)^r\binom{m}{2r}\widehat W^{(m-2r)}_{n}(i,j,k-1)\\
&-\sum_{r=0}^{\lfloor (m-1)/2\rfloor}(-1)^r\binom{m}{2r+1}\sum_{a,b,c\ge 0}\widehat W^{(m-2r-1)}_{n}(a,b,c)M^x_{a,b,c}(i,j,k).
\end{aligned}
\end{equation}

\item [\rm (2)] For $n,m,s\in\mathbb N$ there exist unique coefficient arrays
$U^{(n,m)}_{s}(i,j,k)$, $\widetilde U^{(n,m)}_{s}(i,j,k)$, $\widehat U^{(n,m)}_{s}(i,j,k)\in\Bbbk$ such that
\begin{align}
(x^n y^m)^s&=\sum_{i,j,k\ge 0}U^{(n,m)}_{s}(i,j,k)x^iy^jz^k,\label{Uxydefcase5i}\\
(x^n z^m)^s&=\sum_{i,j,k\ge 0}\widetilde U^{(n,m)}_{s}(i,j,k)x^iy^jz^k,\label{Uxzdefcase5i}\\
(y^n z^m)^s&=\sum_{i,j,k\ge 0}\widehat U^{(n,m)}_{s}(i,j,k)x^iy^jz^k,\label{Uyzdefcase5i}
\end{align}
with the initial data
\begin{align*}
U^{(n,m)}_{1}(i,j,k) &\ =\delta_{i,n}\delta_{j,m}\delta_{k,0},\\
\widetilde U^{(n,m)}_{1}(i,j,k) &\ =\delta_{i,n}\delta_{j,0}\delta_{k,m},\\
\widehat U^{(n,m)}_{1}(i,j,k)&\ =\delta_{i,0}\delta_{j,n}\delta_{k,m}.
\end{align*}
These satisfy the following explicit coefficient recursions for $s\ge 1$:

Define intermediate arrays $U^{[0]}_{s},U^{[1]}_{s},\dots,U^{[n+m]}_{s}$ by $U^{[0]}_{s}:=U^{(n,m)}_{s}$ and
\begin{align*}
    U^{[r]}_{s}(i,j,k) = &\, \sum_{a,b,c\ge 0}U^{[r-1]}_{s}(a,b,c)M^x_{a,b,c}(i,j,k)\text{ for } 1\le r\le n, \\
    U^{[n+t]}_{s}(i,j,k) = &\, \sum_{a,b,c\ge 0}U^{[n+t-1]}_{s}(a,b,c)M^y_{a,b,c}(i,j,k)\text{ for } 1\le t\le m.
\end{align*}
Then
\begin{equation}\label{Uxyreccase5i}
U^{(n,m)}_{s+1}(i,j,k)=U^{[n+m]}_{s}(i,j,k).
\end{equation}
Define $\widetilde U^{[0]}_{s},\dots,\widetilde U^{[n]}_{s}$ by $\widetilde U^{[0]}_{s}:=\widetilde U^{(n,m)}_{s}$ and
\[
\widetilde U^{[r]}_{s}(i,j,k)=\sum_{a,b,c\ge 0}\widetilde U^{[r-1]}_{s}(a,b,c)M^x_{a,b,c}(i,j,k)\text{ for } 1\le r\le n.
\]
Then
\begin{equation}\label{Uxzreccase5i}
\widetilde U^{(n,m)}_{s+1}(i,j,k)=\widetilde U^{[n]}_{s}(i,j,k-m),
\end{equation}
with the convention $\widetilde U^{[n]}_{s}(i,j,k-m)=0$ if $k-m<0$.

Define $\widehat U^{[0]}_{s},\dots,\widehat U^{[n]}_{s}$ by $\widehat U^{[0]}_{s}:=\widehat U^{(n,m)}_{s}$ and
\[
\widehat U^{[r]}_{s}(i,j,k)=\sum_{a,b,c\ge 0}\widehat U^{[r-1]}_{s}(a,b,c)M^y_{a,b,c}(i,j,k)\text{ for } 1\le r\le n.
\]
Then
\begin{equation}\label{Uyzreccase5i}
\widehat U^{(n,m)}_{s+1}(i,j,k)=\widehat U^{[n]}_{s}(i,j,k-m),
\end{equation}
with $\widehat U^{[n]}_{s}(i,j,k-m)=0$ if $k-m<0$.

\item[\rm (3)] There exist unique coefficients $V_s(i,j,k)\in\Bbbk$ such that
\begin{equation}\label{Vdefcase5i}
(xyz)^s=\sum_{i,j,k\ge 0}V_s(i,j,k) x^i y^j z^k,\ V_1(i,j,k)=\delta_{i,1}\delta_{j,1}\delta_{k,1}.
\end{equation}

For $s\ge 1$, define $V^{[0]}_{s}:=V_s$ and
\[
V^{[1]}_{s}(i,j,k)=\sum_{a,b,c\ge 0}V^{[0]}_{s}(a,b,c) M^x_{a,b,c}(i,j,k),\
V^{[2]}_{s}(i,j,k)=\sum_{a,b,c\ge 0}V^{[1]}_{s}(a,b,c) M^y_{a,b,c}(i,j,k).
\]

Then the explicit coefficient recursion is
\begin{equation}\label{Vreccase5i}
V_{s+1}(i,j,k)=V^{[2]}_{s}(i,j,k-1).
\end{equation}

\item [\rm (4)] For all $n,m,t,s\in\mathbb N$, there exist unique coefficients $R^{(n,m,t)}_{s}(i,j,k)\in\Bbbk$ such that
\begin{equation}\label{Rdefcase5i}
(x^n y^m z^t)^s=\sum_{i,j,k\ge 0}R^{(n,m,t)}_{s}(i,j,k) x^i y^j z^k, \ R^{(n,m,t)}_{1}(i,j,k)=\delta_{i,n}\delta_{j,m}\delta_{k,t}.
\end{equation}
For $s\ge 1$, define $R^{[0]}_{s}:=R^{(n,m,t)}_{s}$, and set
\[
R^{[r]}_{s}(i,j,k)=\sum_{a,b,c\ge 0}R^{[r-1]}_{s}(a,b,c) M^x_{a,b,c}(i,j,k) \quad {\rm for} \, \,  1\le r\le n,
\]
\[
R^{[n+t']}_{s}(i,j,k)=\sum_{a,b,c\ge 0}R^{[n+t'-1]}_{s}(a,b,c) M^y_{a,b,c}(i,j,k) \quad {\rm for} \, \, 1\le t'\le m.
\]
The explicit coefficient recursion is
\begin{equation}\label{Rreccase5i}
R^{(n,m,t)}_{s+1}(i,j,k)=R^{[n+m]}_{s}(i,j,k-t),
\end{equation}
with $R^{[n+m]}_{s}(i,j,k-t)=0$ if $k-t<0$.

\item [\rm (5)] There exist unique coefficient arrays $E^{\langle xy\rangle}_{n}(i,j,k)$, $E^{\langle xz\rangle}_{n}(i,j,k)$, $E^{\langle yz\rangle}_{n}(i,j,k)$ belonging to $\Bbbk$ such that
\begin{align}
(x+y)^n&=\sum_{i,j,k\ge 0}E^{\langle xy\rangle}_{n}(i,j,k) x^i y^j z^k,\label{Exydefcase5i}\\
(x+z)^n&=\sum_{i,j,k\ge 0}E^{\langle xz\rangle}_{n}(i,j,k) x^i y^j z^k,\label{Exzdefcase5i} \quad {\rm and} \\
(y+z)^n&=\sum_{i,j,k\ge 0}E^{\langle yz\rangle}_{n}(i,j,k) x^i y^j z^k,\label{Eyzdefcase5i}
\end{align}
with $E^{\langle\ast\rangle}_{0}(i,j,k)=\delta_{i,0}\delta_{j,0}\delta_{k,0}$. For $n\ge 0$ they satisfy the explicit recursions
\begin{equation}\label{Ereccase5i}
\begin{aligned}
E^{\langle xy\rangle}_{n+1}(i,j,k)
&=\sum_{a,b,c\ge 0}E^{\langle xy\rangle}_{n}(a,b,c)\left(M^x_{a,b,c}(i,j,k)+M^y_{a,b,c}(i,j,k)\right),\\
E^{\langle xz\rangle}_{n+1}(i,j,k)
&=\sum_{a,b,c\ge 0}E^{\langle xz\rangle}_{n}(a,b,c)\left(M^x_{a,b,c}(i,j,k)+M^z_{a,b,c}(i,j,k)\right), \quad {\rm and} \\
E^{\langle yz\rangle}_{n+1}(i,j,k)
&=\sum_{a,b,c\ge 0}E^{\langle yz\rangle}_{n}(a,b,c)\left(M^y_{a,b,c}(i,j,k)+M^z_{a,b,c}(i,j,k)\right).
\end{aligned}
\end{equation}
\end{enumerate}
\end{proposition}
\begin{proof}
In the enveloping algebra $U(\mathfrak g)$, it is well-known \cite[Lemma 2.3]{Passman1992} that
\[
b a^m=\sum_{k=0}^{m}\binom{m}{k}a^{m-k}(\mathrm{ad}_a)^k(b),
\qquad \mathrm{ad}_a(b)=[b,a].
\]
Apply it with $(a,b)=(x,y)$ and $(a,b)=(x,z)$. Since
$$
\mathrm{ad}_x(y)=[y,x]=-z \quad {\rm and} \quad \mathrm{ad}_x(z)=[z,x]=y,
$$
it follows that 
$\mathrm{ad}_x^{2}(y)=-y$ and $\mathrm{ad}_x^{2}(z)=-z$, so the terms in the binomial sum split into $C_m$ and $S_m$. The same identity with $(a,b)=(y,z)$.
\begin{enumerate}
    \item [\rm (1)] Write $y^{n+1}x^m=y^n(yx^m)$ and expanded via \eqref{CScase5i}.
Since 
\[
y^n x^{m-2r}y=\sum_{i,j,k}W^{(m-2r)}_n(i,j,k) x^i y^{j+1}z^k
\]
and
\[
y^n x^{m-2r-1}z=\sum_{i,j,k}W^{(m-2r-1)}_n(i,j,k) x^i y^{j}z^{k+1},
\]
collecting coefficients gives \eqref{Wrecynxmcase5i}.

The formula \eqref{Mzkernelcase5i} is immediate. For \eqref{Mykernelcase5i}, write $x^a y^b z^c y=x^a y^b(z^c y)$ and expand $z^c y$ via \eqref{CScase5i}.
The term $x^a y^b(y z^{c-2r})$ is already in PBW order. The term $x^a y^b(x z^{c-2r-1})$ contains the subword $y^b x$; collecting coefficients yields \eqref{Mykernelcase5i}. The proof of \eqref{Mxkernelcase5i} is analogous.

Use $z^{n+1}x^m=z^n(zx^m)$. The $C_m(x)$-part appends a $z$ on the right, while the $S_m(x)$-part appends a $y$ on the right, handled coefficientwise by $M^y$; this gives \eqref{Wtildereccase5i}. Similarly, from $z^{n+1}y^m=z^n(zy^m)$, the $S_m(y)$-part appends an $x$ on the right, handled coefficientwise by $M^x$, yielding\eqref{Whatreccase5i}.

\item [\rm (2)--(5)] In each case, uniqueness of the PBW expansion defines the corresponding coefficient arrays.
Multiplying by the prescribed block on the right updates coefficients by repeated use of the explicit kernels
$M^x,M^y,M^z$ from \eqref{Mzkernelcase5i}-\eqref{Mxkernelcase5i}, which is exactly what the displayed recursions encode.
\end{enumerate}
\end{proof}

\subsubsection{Algebras of type (5)(ii)}

\begin{proposition}\label{Case(5)(ii)}
Let $A$ be the $\Bbbk$-algebra generated by $x,y,z$ with relations
\[
yz - zy = 0, \quad zx - xz=0 \quad {\rm and} \quad xy - yx = z.
\]
Write 
$$
m^{\underline{k}} := m(m-1)\cdots(m-k+1)
$$ 
for the falling factorial {\rm (}with $m^{\underline{0}}=1${\rm )} and
$\left\{\begin{matrix}s\\r\end{matrix}\right\}$ for the Stirling numbers of the second kind.

Then the following hold:
\begin{enumerate}
\item [\rm (1)] For all $m,n\in\mathbb N$,
\[
z^n x^m=x^m z^n \quad {\rm and} \quad z^n y^m=y^m z^n.
\]
Moreover, for all $m,n\in\mathbb N$ there exist uniquely determined scalars
$W^{(m)}_{n,k}\in\Bbbk$, for $0\le k\le \min\{m,n\}$, such that
\begin{equation}\label{Wdefcase5ii}
y^n x^m=\sum_{k=0}^{\min\{m,n\}} x^{m-k}W^{(m)}_{n,k} y^{n-k}z^{k}.
\end{equation}
These elements are characterized by
\[
W^{(m)}_{0,0}=1,\ W^{(m)}_{0,k}=0\ \text{ for } k\ge 1,
\]
and the explicit recursion (with $W^{(m)}_{n,-1}:=0$)
\begin{equation}\label{Wreccase5ii}
W^{(m)}_{n+1,k}=W^{(m)}_{n,k}-(m-k+1) W^{(m)}_{n,k-1},
\qquad n\ge 0,\ k\ge 0.
\end{equation}
In fact, one has the closed form
\begin{equation}\label{Wclosed_case5ii}
W^{(m)}_{n,k}=(-1)^k\binom{n}{k} m^{\underline{k}}.
\end{equation}

\item [\rm (2)] For all $m,n,s\in\mathbb N$,
\[
(x^n z^m)^s=x^{ns}z^{ms},\ (y^n z^m)^s=y^{ns}z^{ms}.
\]
Moreover, $(x^n y^m)^s$ admits a unique PBW expansion
\begin{equation}\label{Udefcase5ii}
(x^n y^m)^s=\sum_{\ell=0}^{\min\{ns,ms\}} x^{ns-\ell} U^{(n,m)}_{s,\ell} y^{ms-\ell}z^{\ell},
\ U^{(n,m)}_{s,\ell}\in\Bbbk,
\end{equation}
with $U^{(n,m)}_{1,0}=1$ and $U^{(n,m)}_{1,\ell}=0$ for $\ell\ge 1$, and the explicit recursion
\begin{equation}\label{Ureccase5ii}
U^{(n,m)}_{s+1,\ell}=\sum_{k=0}^{\min\{n,\ell\}}U^{(n,m)}_{s,\ell-k} W^{(n)}_{ms-(\ell-k), k},
\text{ for } s\ge 1,\ \ell\ge 0,
\end{equation}
where $W^{(n)}_{r,k}$ are the normal ordering coefficients from \eqref{Wdefcase5ii} (with $m$ replaced by $n$).

\item [\rm (3)] There exist uniquely determined scalars $V_{s,\ell}\in\Bbbk$, for $0\le \ell\le s$, such that
\begin{equation}\label{Vdefcase5ii}
(xyz)^s=\sum_{\ell=0}^{s} x^{ s-\ell} y^{ s-\ell} V_{s,\ell} z^{ s+\ell}.
\end{equation}
These satisfy $V_{1,0}=1$, $V_{1,\ell}=0$ for $\ell\ge 1$, and the explicit recursion {\rm (}with $V_{s,-1}:=0${\rm )}
\begin{equation}\label{Vreccase5ii}
V_{s+1,\ell}=V_{s,\ell}-(s+1-\ell) V_{s,\ell-1},
\qquad s\ge 1,\ \ell\ge 0.
\end{equation}
In fact,
\begin{equation}\label{Vclosedcase5ii}
V_{s,\ell}=(-1)^{\ell}\left\{\!\!\begin{matrix} s\\ s-\ell\end{matrix}\!\!\right\}.
\end{equation}

\item [\rm (4)] For all $n,m,t,s\in\mathbb N$, there exist uniquely determined scalars $R^{(n,m,t)}_{s,\ell}\in\Bbbk$,
for $0\le \ell\le \min\{ns,ms\}$, such that
\begin{equation}\label{Rdefcase5ii}
(x^n y^m z^t)^s=\sum_{\ell=0}^{\min\{ns,ms\}} x^{ns-\ell} R^{(n,m,t)}_{s,\ell} y^{ms-\ell}z^{ts+\ell}.
\end{equation}

These elements satisfy $R^{(n,m,t)}_{1,0}=1$ and $R^{(n,m,t)}_{1,\ell}=0$ for $\ell\ge 1$, and the explicit recursion
\begin{equation}\label{Rreccase5ii}
R^{(n,m,t)}_{s+1,\ell}=\sum_{k=0}^{\min\{n,\ell\}}R^{(n,m,t)}_{s,\ell-k} W^{(n)}_{ms-(\ell-k), k},
\text{ for } s\ge 1,\ \ell\ge 0,
\end{equation}
i.e., $R^{(n,m,t)}_{s,\ell}=U^{(n,m)}_{s,\ell}$ for all $s,\ell$.

\item [\rm (5)] Since $z$ is central and commutes with $x$ and $y$, one has the ordinary binomial formulas
\[
(x+z)^n=\sum_{k=0}^{n}\binom{n}{k}x^{n-k}z^{k},\ (y+z)^n=\sum_{k=0}^{n}\binom{n}{k}y^{n-k}z^{k}.
\]
Finally, $(x+y)^n$ admits a unique PBW expansion
\begin{equation}\label{Edefcase5ii}
(x+y)^n=\sum_{\substack{i,k\ge 0\\ i+2k\le n}} x^{i} E_{n;i,k} y^{ n-i-2k} z^{k},
\ E_{n;i,k}\in\Bbbk,
\end{equation}
with $E_{0;0,0}=1$ and $E_{0;i,k}=0$ for $(i,k)\neq(0,0)$, and the explicit recursion
\begin{equation}\label{Ereccase5ii}
E_{n+1;i,k}=E_{n;i-1,k}+E_{n;i,k}-(n+2-i-2k) E_{n;i,k-1},
\end{equation}
where $E_{n;i,k}=0$ if $i<0$ or $k<0$ or $i+2k>n$.
Moreover, one has the closed form
\begin{equation}\label{Eclosedcase5ii}
E_{n;i,k}=(-1)^k \frac{n!}{i! (n-i-2k)! k! 2^{k}}, \text{ for } i+2k\le n.
\end{equation}
\end{enumerate}
\end{proposition}
\begin{proof}
\begin{enumerate}
\item [\rm (1)] The identities $z^n x^m=x^m z^n$ and $z^n y^m=y^m z^n$ follow by induction from $zx=xz$ and $zy=yz$. For $y^n x^m$, PBW gives the expansion \eqref{Wdefcase5ii}. To obtain the recursion, note first that
\begin{equation} \label{yxmcase5ii}
yx^m = (yx)x^{m-1} = (xy-z)x^{m-1} =  x(yx^{m-1})-x^{m-1}z, 
\end{equation}
and since $z$ is central, an induction on $m$ yields
\begin{equation}\label{yxmexplicitcase5ii}
yx^m=x^m y-m x^{m-1}z.
\end{equation}
Write $y^{n+1}x^m=y^n(yx^m)$ and substitute \eqref{yxmexplicitcase5ii}; collecting the coefficients of $x^{m-k}y^{(n+1)-k}z^k$ gives \eqref{Wreccase5ii}.

\item [\rm (2)] The formulas for $(x^n z^m)^s$ and $(y^n z^m)^s$ are immediate since $z$ commutes with $x$ and $y$. For \eqref{Udefcase5ii}, PBW gives the existence and uniqueness of the coefficients $U^{(n,m)}_{s,\ell}$.
Assume \eqref{Udefcase5ii} holds for some $s\ge 1$ and multiply by $x^n y^m$ on the right:
\[
(x^n y^m)^{s+1}=\sum_{\ell}x^{ns-\ell}U^{(n,m)}_{s,\ell} (y^{ms-\ell}x^n) y^m z^\ell.
\]
Normal order the subword $y^{ms-\ell}x^n$ using \eqref{Wdefcase5ii} with $m$ replaced by $n$:
\[
y^{ms-\ell}x^n=\sum_{k=0}^{\min\{n,ms-\ell\}}x^{n-k}W^{(n)}_{ms-\ell,k} y^{ms-\ell-k}z^k.
\]
Since $z$ is central, this yields exactly \eqref{Ureccase5ii} after reindexing $\ell\mapsto \ell-k$.

\item [\rm (3)] Since $z$ is central, $\left(xyz\right)^s = (xy)^s z^s$. We claim that
\begin{equation}\label{xypowerstirlingcase5ii}
(xy)^s=\sum_{\ell=0}^{s}x^{s-\ell}y^{s-\ell} V_{s,\ell} z^{\ell}.
\end{equation}
For $s=1$ this holds with $V_{1,0}=1$. Assume \eqref{xypowerstirlingcase5ii} for $s$ and multiply by $xy$ on the right:
\[
(xy)^{s+1}=\sum_{\ell}x^{s-\ell}y^{s-\ell}V_{s,\ell} z^\ell xy.
\]
Using that $z$ is central and the identity 
\begin{equation}\label{yrxcase5ii}
y^{r}x=x y^{r}-r y^{r-1}z \text{ for } r\ge 1,
\end{equation}
(proved by induction from $yx=xy-z$) with $r=s-\ell$, we obtain 
\begin{align*}
x^{s-\ell}\big(y^{s-\ell}x\big)y z^\ell
 &\ =x^{s-\ell}\big(xy^{s-\ell}-(s-\ell)y^{s-\ell-1}z\big)y z^\ell \\
&\ =x^{s+1-\ell}y^{s+1-\ell}z^\ell-(s-\ell)x^{s-\ell}y^{s-\ell}z^{\ell+1}.
\end{align*}
Collecting coefficients of $x^{(s+1)-\ell}y^{(s+1)-\ell}z^\ell$ yields \eqref{Vreccase5ii}.
Multiplying \eqref{xypowerstirlingcase5ii} by $z^s$ gives \eqref{Vdefcase5ii}.
Finally, \eqref{Vclosedcase5ii} follows from \eqref{Vreccase5ii} together with the standard Stirling recursion
$$
\left\{\begin{matrix}s+1\\r\end{matrix}\right\}
=r\left\{\begin{matrix}s\\r\end{matrix}\right\}
+\left\{\begin{matrix}s\\r-1\end{matrix}\right\}, 
$$ 
after the change $r = s-\ell$.

\item [\rm (4)] Since $z$ is central,
\[
(x^n y^m z^t)^s=(x^n y^m)^s z^{ts}.
\]
Substituting \eqref{Udefcase5ii} gives \eqref{Rdefcase5ii} with $R^{(n,m,t)}_{s,\ell}=U^{(n,m)}_{s,\ell}$, and then \eqref{Rreccase5ii} is exactly \eqref{Ureccase5ii}.

\item [\rm (5)] The expansions for $(x+z)^n$ and $(y+z)^n$ are the ordinary binomial theorem since $z$ commutes with $x$ and $y$. For \eqref{Edefcase5ii}, the PBW basis gives the existence/uniqueness of $E_{n;i,k}$. Multiply
\[
(x+y)^{n+1}=(x+y)^n x+(x+y)^n y
\]
and use \eqref{yrxcase5ii} to normal order the subword $y^{n-i-2k}x$ that occurs in the $x$-term:
\[
y^{n-i-2k}x=x y^{n-i-2k}-(n-i-2k) y^{n-i-2k-1}z.
\]
Collecting coefficients of $x^i y^{(n+1)-i-2k}z^k$ yields \eqref{Ereccase5ii}.
\end{enumerate}
\end{proof}

\subsubsection{Algebras of type (5)(iii)}

\begin{proposition}\label{Case(5)(iii)}
Let $A$ be the $\Bbbk$-algebra generated by $x,y,z$ with relations
\[
yz-zy=0,\ zx-xz=0 \quad {\rm and} \quad xy-yx=b,
\]
where $b\in\Bbbk$. Then the following hold.

\begin{enumerate}
\item [\rm (1)] For all $m,n\in\mathbb N$,
\[
z^n x^m = x^m z^n \quad {\rm and} \quad z^n y^m = y^m z^n.
\]
Moreover, for all $m,n\in\mathbb N$ there exist uniquely determined scalars
$W^{(m)}_{n,k}\in\Bbbk$, for $0\le k\le \min\{m,n\}$, such that
\begin{equation}\label{Wdefcase5iii}
y^n x^m=\sum_{k=0}^{\min\{m,n\}} x^{m-k}W^{(m)}_{n,k} y^{n-k}.
\end{equation}
These are characterized by
\[
W^{(m)}_{0,0}=1,\ W^{(m)}_{0,k}=0\ \text{ for } k\ge 1,
\]
and the explicit recursion {\rm (}with $W^{(m)}_{n,-1}:=0${\rm )}, 
\begin{equation}\label{Wreccase5iii}
W^{(m)}_{n+1,k}=W^{(m)}_{n,k}-b (m-k+1) W^{(m)}_{n,k-1},
\qquad n\ge 0,\ k\ge 0.
\end{equation}
In fact, we have the closed form
\begin{equation}\label{Wclosedcase5iii}
W^{(m)}_{n,k}=(-1)^k\binom{n}{k} b^{k} m^{\underline{k}}.
\end{equation}

\item [\rm (2)] For all $m,n,s\in\mathbb N$,
\[
(x^n z^m)^s=x^{ns}z^{ms},\ (y^n z^m)^s=y^{ns}z^{ms}.
\]
Moreover, $(x^n y^m)^s$ admits the PBW expansion
\begin{equation}\label{Udefcase5iii}
(x^n y^m)^s=\sum_{\ell=0}^{\min\{ns,ms\}} x^{ns-\ell} U^{(n,m)}_{s,\ell} y^{ms-\ell}, \ U^{(n,m)}_{s,\ell}\in\Bbbk,
\end{equation}
with $U^{(n,m)}_{1,0}=1$ and $U^{(n,m)}_{1,\ell}=0$ for $\ell\ge 1$, and the explicit recursion
\begin{equation}\label{Ureccase5iii}
U^{(n,m)}_{s+1,\ell}=\sum_{k=0}^{\min\{n,\ell\}}
U^{(n,m)}_{s,\ell-k} W^{(n)}_{ms-(\ell-k), k}, \text{ for }  s\ge 1, \ell\ge 0,
\end{equation}
where $W^{(n)}_{r,k}$ are the normal ordering coefficients from \eqref{Wdefcase5iii} {\rm (}with $m$ replaced by $n${\rm )}. Equivalently, using \eqref{Wclosedcase5iii},
\begin{equation}\label{Urecclosedcase5iii}
U^{(n,m)}_{s+1,\ell}=\sum_{k=0}^{\min\{n,\ell\}}(-1)^k\binom{ms-(\ell-k)}{k} b^{k} n^{\underline{k}} U^{(n,m)}_{s,\ell-k}.
\end{equation}

\item [\rm (3)] Since $b$ is central and $z$ commutes with $x,y$, one has
\[
(xyz)^s=(xy)^s z^s.
\]
There exist uniquely determined scalars $V_{s,\ell}\in\Bbbk$, for $0\le \ell\le s$, such that
\begin{equation}\label{Vdefcase5iii}
(xyz)^s=\sum_{\ell=0}^{s} x^{ s-\ell} y^{ s-\ell} V_{s,\ell} z^{s}.
\end{equation}
These satisfy $V_{1,0}=1$, $V_{1,\ell}=0$ for $\ell\ge 1$, and the explicit recursion (with $V_{s,-1}:=0$)
\begin{equation}\label{Vreccase5iii}
V_{s+1,\ell}=V_{s,\ell}-b (s+1-\ell) V_{s,\ell-1}, \text{ for } s\ge 1,\ \ell\ge 0.
\end{equation}
Equivalently, 
\begin{equation}\label{Vclosedcase5iii}
V_{s,\ell}=(-1)^{\ell}\binom{s}{\ell} b^{\ell} (s-\ell)!.
\end{equation}

\item [\rm (4)] For all $n,m,t,s\in\mathbb N$, there exist uniquely determined scalars $R^{(n,m,t)}_{s,\ell}\in\Bbbk$,
for $0\le \ell\le \min\{ns,ms\}$, such that
\begin{equation}\label{Rdefcase5iii}
(x^n y^m z^t)^s=\sum_{\ell=0}^{\min\{ns,ms\}} x^{ns-\ell} R^{(n,m,t)}_{s,\ell} y^{ms-\ell}z^{ts}.
\end{equation}
These satisfy $R^{(n,m,t)}_{1,0}=1$ and $R^{(n,m,t)}_{1,\ell}=0$ for $\ell\ge 1$, and the explicit recursion
\begin{equation}\label{Rreccase5iii}
R^{(n,m,t)}_{s+1,\ell}=\sum_{k=0}^{\min\{n,\ell\}}R^{(n,m,t)}_{s,\ell-k} W^{(n)}_{ms-(\ell-k), k}, \text{ for } s\ge 1,\ \ell\ge 0,
\end{equation}
i.e., $R^{(n,m,t)}_{s,\ell}=U^{(n,m)}_{s,\ell}$ for all $s, l$.

\item [\rm (5)] Since $z$ is central and commutes with $x$ and $y$, then  
\[
(x+z)^n=\sum_{k=0}^{n}\binom{n}{k}x^{n-k}z^{k},\ (y+z)^n=\sum_{k=0}^{n}\binom{n}{k}y^{n-k}z^{k}.
\]
Finally, $(x+y)^n$ admits a unique PBW expansion given by 
\begin{equation}\label{Edefcase5iii}
(x+y)^n=\sum_{i=0}^{n} x^{i}E_{n;i} y^{ n-i}, \ E_{n;i}\in\Bbbk,
\end{equation}
with $E_{0;0}=1$ and $E_{0;i}=0$ for $i\neq 0$, and the explicit recursion
\begin{equation}\label{Ereccase5iii}
E_{n+1;i}=E_{n;i-1}+E_{n;i}-b (n+1-i) E_{n;i},
\end{equation}
where $E_{n;i}=0$ if $i<0$ or $i>n$. Equivalently, we have the closed form
\begin{equation}\label{Eclosedcase5iii}
E_{n;i}=\binom{n}{i} ( -b)^{ n-i} (n-i)!.
\end{equation}
\end{enumerate}
\end{proposition}
\begin{proof}
\begin{enumerate}
\item [\rm (1)] The commutations with $z$ are immediate. For $y^n x^m$, first prove by induction on $m$ that
\begin{equation}\label{yxmcase5iii}
yx^m=x^m y-bm x^{m-1}.
\end{equation}
Indeed, $yx = (xy-b) = xy - b$, and if \eqref{yxmcase5iii} holds for $m-1$, then
\begin{align*}
yx^m = &\, (yx)x^{m-1}=(xy-b)x^{m-1} \\
= &\, x(yx^{m-1})-b x^{m-1} \\
= &\, x(x^{m-1}y-b(m-1)x^{m-2})-b x^{m-1},
\end{align*}
which gives \eqref{yxmcase5iii}. Now write $y^{n+1}x^m=y^n(yx^m)$ and substitute \eqref{yxmcase5iii}.
Collecting coefficients of $x^{m-k}y^{(n+1)-k}$ yields \eqref{Wreccase5iii}. The closed form \eqref{Wclosedcase5iii} follows by induction on $n$.

\item [\rm (2)] The statements for $(x^n z^m)^s$ and $(y^n z^m)^s$ are immediate since $z$ commutes with $x,y$. For \eqref{Udefcase5iii}, PBW gives existence of $U^{(n,m)}_{s,\ell}$.
Assume \eqref{Udefcase5iii} for $s$ and multiply by $x^n y^m$ on the right:
\[
(x^n y^m)^{s+1}=\sum_{\ell}x^{ns-\ell}U^{(n,m)}_{s,\ell} (y^{ms-\ell}x^n) y^m.
\]
Normal order the subword $y^{ms-\ell}x^n$ using \eqref{Wdefcase5iii} with $m$ replaced by $n$, and collect coefficients, which yields \eqref{Ureccase5iii} (and then \eqref{Urecclosedcase5iii}).

\item [\rm (3)] Since $zx = xz$ and $zy = yz$, we have that $(xyz)^s=(xy)^s z^s$.
Write 
$$
(xy)^s=\sum_{\ell=0}^s x^{s-\ell}y^{s-\ell}V_{s,\ell}
$$ 
and proceed as in the proof of (5)(ii), but using $y^{r}x=x y^{r}-br y^{r-1}$ (the analogue of \eqref{yxmcase5iii} with $m=1$).
This gives \eqref{Vreccase5iii}. The closed form \eqref{Vclosedcase5iii} follows by the identity $e^{t(x+y)}=e^{tx}e^{ty}e^{-bt^2/2}$ since $[x,y] = b$ is central.

\item [\rm (4)] Since $z$ is central,
\[
\left(x^n y^m z^t\right)^s = \left(x^n y^m\right)^s z^{ts},
\]
and we get \eqref{Rdefcase5iii} with $R^{(n,m,t)}_{s,\ell}=U^{(n,m)}_{s,\ell}$, and so the recursion matches \eqref{Ureccase5iii}.

\item [\rm (5)] The binomial formulas with $z$ are classical. For $(x+y)^n$, define $E_{n;i}$ by \eqref{Edefcase5iii}.
Then
\[
(x+y)^{n+1}=(x+y)^n x+(x+y)^n y,
\]
and normal order the $x$-term using 
$$
y^{n-i}x=x y^{n-i}-b(n-i)y^{n-i-1}.
$$
Collecting coefficients gives \eqref{Ereccase5iii}, and \eqref{Eclosedcase5iii} follows by induction on $n$.
\end{enumerate}
\end{proof}

\subsubsection{Algebras of type (5)(vi)}

\begin{proposition}\label{Case(5)(iv)}
Let $A$ be the $\Bbbk$-algebra generated by $x,y,z$ with relations
\[
yz - zy = -y,\quad zx - xz = x+y \quad {\rm and} \quad xy - yx = 0.
\]
Then the following hold.

\begin{enumerate}
\item [\rm (1)] For all $m,n\in\mathbb N$, $y^n x^m=x^m y^n$. Moreover, for all $m,n\in\mathbb N$,
\begin{equation}\label{znymcase5ivclosed}
z^n y^m = y^m\left(z+m\right)^n = \sum_{k=0}^{n}\binom{n}{k}m^k y^m z^{ n-k}.
\end{equation}
Finally, for all $m,n\in\mathbb N$ there exist uniquely elements $W^{(m)}_{n;k,j}\in\Bbbk$, for $0\le k\le n$ and $0\le j\le \min\{m,k\}$, such that
\begin{equation}\label{znxmcase5ivWdef}
z^n x^m = \sum_{k=0}^{n}\ \sum_{j=0}^{\min\{m,k\}} x^{m-j}y^{j} W^{(m)}_{n;k,j} z^{ n-k}.
\end{equation}
These elements are characterized as
\[
W^{(m)}_{0;0,0} = 1 \quad {\rm and} \quad W^{(m)}_{0;k,j} = 0 \quad {\rm for}\, \, (k,j)\neq(0,0),
\]
and the explicit recursion {\rm (}with the conventions $W^{(m)}_{n;-1,j}=W^{(m)}_{n;k,-1}=0${\rm )}
\begin{align}
W^{(m)}_{n+1;k,j} = &\, W^{(m)}_{n;k,j}+m W^{(m)}_{n;k-1,j} \\
&\, +(m-j+1) W^{(m)}_{n;k-1,j-1}, \quad {\rm for} \, \, n, k, j \ge 0. \label{Wreccase5iv}
\end{align}

\item [\rm (2)] For all $m,n,s\in\mathbb N$, $\left(x^n y^m\right)^s = x^{ns}y^{ms}$. There exist unique elements $\widehat U^{(n,m)}_{s,\ell}\in\Bbbk$ {\rm (}$0\le \ell\le ms${\rm )} such that
\begin{equation}\label{Uhatdefcase5iv}
(y^n z^m)^s= y^{ns}\sum_{\ell=0}^{ms}\widehat U^{(n,m)}_{s,\ell} z^{ ms-\ell}.
\end{equation}
These satisfy $\widehat U^{(n,m)}_{1,0}=1$ and $\widehat U^{(n,m)}_{1,\ell}=0$ for $\ell\ge 1$, and the explicit recursion
\begin{equation}\label{Uhatreccase5iv}
\widehat U^{(n,m)}_{s+1,\ell} =\sum_{r=0}^{\min\{m,\ell\}}
\binom{m}{r}(ns)^{r} \widehat U^{(n,m)}_{s,\ell-r}, \text{ for } s\ge 1,\ \ell\ge 0.
\end{equation}
Equivalently, we have the closed product form
\begin{equation}\label{Uhatproductcase5iv}
\left(y^n z^m\right)^s = y^{ns}\prod_{r=0}^{s-1}(z+rn)^{m}.
\end{equation}
There exist unique scalars $\widetilde U^{(n,m)}_{s;\ell,j}\in\Bbbk$,
$0\le \ell\le ms$ and $0\le j\le \min\{ns,\ell\}$, such that
\begin{equation}\label{Utildedefcase5iv}
(x^n z^m)^s=\sum_{\ell=0}^{ms}\ \sum_{j=0}^{\min\{ns,\ell\}}
x^{ns-j}y^{j} \widetilde U^{(n,m)}_{s;\ell,j} z^{ ms-\ell}.
\end{equation}
These satisfy $\widetilde U^{(n,m)}_{1;0,0}=1$ and $\widetilde U^{(n,m)}_{1;\ell,j}=0$ for $(\ell,j)\neq(0,0)$, and the explicit recursion
\begin{equation}\label{Utildereccase5iv}
\widetilde U^{(n,m)}_{s+1;\ell,j} = \sum_{k=0}^{\min\{n,\ell\}}
\ \sum_{q=0}^{\min\{k,j\}} \widetilde U^{(n,m)}_{s;\ell-k, j-q} 
W^{(n)}_{ ms-(\ell-k);\ k,q}, \text{ for } s\ge 1,\ \ell\ge 0,\ j\ge 0,
\end{equation}
where $W^{(n)}_{r;k,q}$ are the normal ordering coefficients from \eqref{znxmcase5ivWdef} {\rm (}with $m$ replaced by $n${\rm )}.

\item [\rm (3)] There exist unique coefficients $V_s(i,j,k)\in\Bbbk$ such that
\begin{equation}\label{Vdefcase5iv}
(xyz)^s=\sum_{i,j,k\ge 0}V_s(i,j,k) x^i y^j z^k, \ V_1(i,j,k)=\delta_{i,1}\delta_{j,1}\delta_{k,1}.
\end{equation}
Define auxiliary arrays $V^{[x]}_s$ and $V^{[xy]}_s$ by the explicit coefficient updates
\begin{align}
V^{[x]}_s(i,j,k) = &\, \sum_{c\ge k}\binom{c}{k} V_s(i-1,j,c) + \sum_{c\ge k+1}c\binom{c-1}{k} V_s(i,j-1,c) \quad {\rm and}  \label{Vmultxcase5iv} \\
V^{[xy]}_s(i,j,k) = &\, \sum_{c\ge k}\binom{c}{k} V^{[x]}_s(i,j-1,c), \label{Vmultycase5iv}
\end{align}
and then
\begin{equation}\label{Vreccase5iv}
V_{s+1}(i,j,k)=V^{[xy]}_s(i,j,k-1), \quad {\rm for}\, \, s\ge 1,
\end{equation}
with the convention $V^{[xy]}_s(i,j,-1)=0$.

\item [\rm (4)] For all $n,m,t,s\in\mathbb N$ there exist unique scalars $R^{(n,m,t)}_{s;\ell,j}\in\Bbbk$,
$0\le \ell\le ts$ and $0\le j\le \min\{ns,\ell\}$, such that
\begin{equation}\label{Rdefcase5iv}
(x^n y^m z^t)^s=\sum_{\ell=0}^{ts}\ \sum_{j=0}^{\min\{ns,\ell\}}x^{ns-j}y^{ms+j} R^{(n,m,t)}_{s;\ell,j} z^{ ts-\ell}.
\end{equation}
These satisfy $R^{(n,m,t)}_{1;0,0}=1$ and $R^{(n,m,t)}_{1;\ell,j}=0$ for $(\ell,j)\neq(0,0)$, and the explicit recursion
\begin{equation}\label{Rreccase5iv}
R^{(n,m,t)}_{s+1;\ell,j}=\sum_{k=0}^{\min\{n,\ell\}} \sum_{q=0}^{\min\{k,j\}}
R^{(n,m,t)}_{s;\ell-k, j-q} \mathcal W^{(n,m)}_{ ts-(\ell-k);\ k,q},
\end{equation}
for $s\ge 1$, where $\mathcal W^{(n,m)}_{r;k,q}$ are the normal ordering coefficients defined by
\begin{equation}\label{Wnmdefcase5iv}
z^{r}x^{n}y^{m}=\sum_{k=0}^{r}\ \sum_{q=0}^{\min\{n,k\}}x^{n-q}y^{m+q} \mathcal W^{(n,m)}_{r;\ k,q} z^{ r-k},\quad \mathcal W^{(n,m)}_{0;0,0}=1,
\end{equation}
and satisfy the explicit recursion
\begin{equation}\label{Wnmreccase5iv}
\mathcal W^{(n,m)}_{r+1;\ k,q}=\mathcal W^{(n,m)}_{r;\ k,q}+(n+m) \mathcal W^{(n,m)}_{r;\ k-1,q}+(n-q+1) \mathcal W^{(n,m)}_{r;\ k-1,q-1}.
\end{equation}

\item [\rm (5)] Since $xy = yx$, the classical binomial theorem gives
\[
(x+y)^n=\sum_{k=0}^{n}\binom{n}{k}x^{n-k}y^{k}.
\]
There exist unique scalars $E^{\langle yz\rangle}_{n;i,k}\in\Bbbk$ {\rm (}$i,k\ge 0$, $i+k\le n${\rm )} such that
\begin{equation}\label{Eyzdefcase5iv}
\left(y + z\right)^n = \sum_{\substack{i,k\ge 0\\ i+k\le n}} y^{i} E^{\langle yz\rangle}_{n;i,k} z^{k},
\end{equation}
with $E^{\langle yz\rangle}_{0;0,0}=1$ and $E^{\langle yz\rangle}_{0;i,k}=0$ for $(i,k)\neq(0,0)$, and the explicit recursion
\begin{equation}\label{Eyzreccase5iv}
E^{\langle yz\rangle}_{n+1;i,k}=E^{\langle yz\rangle}_{n;i,k-1}+\sum_{c\ge k}\binom{c}{k} E^{\langle yz\rangle}_{n;i-1,c},
\end{equation}
where $E^{\langle yz\rangle}_{n;i,k}=0$ if $i<0$ or $k<0$ or $i+k>n$. There exist unique elements $E^{\langle xz\rangle}_{n;i,j,k}\in\Bbbk$ such that
\begin{equation}\label{Exzdefcase5iv}
\left(x + z\right)^n = \sum_{\substack{i,j,k\ge 0\\ i+j+k\le n}} x^{i}y^{j} E^{\langle xz\rangle}_{n;i,j,k} z^{k},
\end{equation}
with $E^{\langle xz\rangle}_{0;0,0,0}=1$ and $E^{\langle xz\rangle}_{0;i,j,k}=0$ for $(i,j,k)\neq(0,0,0)$, and the explicit recursion given by
\begin{equation}\label{Exzreccase5iv}
E^{\langle xz\rangle}_{n+1;i,j,k}=E^{\langle xz\rangle}_{n;i,j,k-1}
+\sum_{c\ge k}\binom{c}{k} E^{\langle xz\rangle}_{n;i-1,j,c}
+\sum_{c\ge k+1}c\binom{c-1}{k} E^{\langle xz\rangle}_{n;i,j-1,c},
\end{equation}
with $E^{\langle xz\rangle}_{n;i,j,k}=0$ if any index is negative or $i+j+k>n$.
\end{enumerate}
\end{proposition}
\begin{proof}
\begin{enumerate}
\item [\rm (1)] Since $xy = yx$, then $y^n x^m = x^m y^n$. From $zy = yz+y$ a reasoning by induction on $m$ shows that $z y^m=y^m z+m y^m$, and so 
$z^n y^m=y^m(z+m)^n$, giving \eqref{znymcase5ivclosed}.

For $z^n x^m$, define $W^{(m)}_{n;k,j}$ by \eqref{znxmcase5ivWdef}. Using that
$$
zf = fz + [z,f], \quad [z, x] = x+y \quad {\rm and} \quad [z, y ] = y,
$$
a direct induction on $m$ yields
\[
[z,x^{m-j}y^j]=m x^{m-j}y^j+(m-j) x^{m-j-1}y^{j+1}.
\]
Hence, 
\begin{align*}
z\left(x^{m-j}y^j z^{n-k}\right) = &\, x^{m-j}y^j z^{n-k+1}+m x^{m-j}y^j z^{n-k} \\
&\, +(m-j) x^{m-j-1}y^{j+1}z^{n-k}.
\end{align*}
Plugging this into $z^{n+1}x^m=z(z^n x^m)$ and collecting the coefficients of $x^{m-j}y^j z^{(n+1)-k}$, we get \eqref{Wreccase5iv}.

\item [\rm (2)] Using that $z^m y^n = y^n(z+n)^m$, then 
\begin{align*}
(y^n z^m)^{s+1} = &\, (y^n z^m)^s y^n z^m \\  
= &\, y^{ns}\Big(\prod_{r=0}^{s-1}(z+rn)^m\Big)y^n z^m \\
= &\, y^{n(s+1)}\Big(\prod_{r=0}^{s}(z+rn)^m\Big),
\end{align*}
which is \eqref{Uhatproductcase5iv}. Expanding $(z+ns)^m = \sum_{r=0}^{m}\binom{m}{r}(ns)^r z^{m-r}$ and collecting coefficients in \eqref{Uhatdefcase5iv} gives \eqref{Uhatreccase5iv}.

Assume \eqref{Utildedefcase5iv} for some $s\ge 1$ and multiply by $x^n z^m$ on the right. Normal order each subword $z^{ms-(\ell-k)}x^n$ using \eqref{znxmcase5ivWdef} with $m$ replaced by $n$, ans use that $x,y$ commute to add exponents. Collecting the coefficient of $x^{n(s+1)-j}y^{j}z^{m(s+1)-\ell}$ yields \eqref{Utildereccase5iv}.

\item [\rm (3)] Write \eqref{Vdefcase5iv} for $s$ and multiply by $xyz$ on the right.
Right multiplication by $x$ is governed by the identities
\begin{align*}
z^{c}x = &\, \sum_{u=0}^{c}\binom{c}{u}xz^{c-u}+ c\sum_{u=0}^{c-1}\binom{c-1}{u}yz^{c-1-u} \quad {\rm and} \\
z^{c}y = &\, \sum_{u=0}^{c}\binom{c}{u}yz^{c-u}, 
\end{align*}
whence \eqref{Vmultxcase5iv} and \eqref{Vmultycase5iv} follow by collecting coefficients. Multiplying by $z$ shifts the $z$-exponent by $+1$, giving \eqref{Vreccase5iv}.

\item [\rm (4)] The definition \eqref{Wnmdefcase5iv} is the PBW uniqueness. The recursion \eqref{Wnmreccase5iv} follows exactly as in \textup{(1)} by applying $z$ to $x^{n-q}y^{m+q}$.
Then multiply \eqref{Rdefcase5iv} by $x^n y^m z^t$ on the right and normal order $z^{ts-(\ell-k)}x^n y^m$ via \eqref{Wnmdefcase5iv}; collecting coefficients yields \eqref{Rreccase5iv}.

\item [\rm (5)] The binomial formula for $(x+y)^n$ is classical since $x,y$ commute. For $(y + z)^n$, using 
$$
(y+z)^{n+1} = (y+z)^n y+(y+z)^n z
$$ 
and the property
$$
z^{c}y = \sum_{u=0}^{c}\binom{c}{u}yz^{c-u}
$$ 
to collect coefficients, it follows \eqref{Eyzreccase5iv}.
For $(x+z)^n$, we proceed similarly but use the straightening formula for $z^{c}x$ displayed above, producing the two sums in \eqref{Exzreccase5iv}.
\end{enumerate}
\end{proof}

\subsubsection{Algebras of type (5)(v)}

\begin{proposition}\label{Case(5)(v)}
Let $A$ be the $\Bbbk$-algebra generated by $x,y,z$ with relations
\[
yz - zy= az,\ zx - xz = z \quad {\rm and} \quad xy - yx = 0, \quad {\rm where} \ a\in\Bbbk^{\times}. 
\]
For $s\in\mathbb N$, set the rising factorial
\[
(x)^{\overline{s}}:=\prod_{r=0}^{s-1}(x+r)\in\Bbbk[x],
\]
and the $a$-step falling factorial
\[
\left(y\right)^{\underline{s}}_{a}:=\prod_{r=0}^{s-1}(y-ra)\in\Bbbk[y], \quad {\rm with} \, \, (x)^{\overline{0}}=(y)^{\underline{0}}_{a}:=1.
\]
Then the following assertions hold:
\begin{enumerate}
\item[\rm (1)] For all $m,n\in\mathbb N$, $y^n x^m=x^m y^n$. Moreover,
\begin{equation}\label{znxmcase5v}
z^n x^m=(x+n)^m z^n=\sum_{k=0}^{m}\binom{m}{k} n^{k} x^{m-k} z^n,
\end{equation}
and
\begin{equation}\label{znymcase5v}
z^n y^m = \left(y - na\right)^m z^n = \sum_{k=0}^{m}\binom{m}{k} (-na)^{k} y^{m-k} z^n.
\end{equation}
Equivalently, the coefficients
\[
W^{(m)}_{n,k}:=\binom{m}{k}n^{k} \quad {\rm and} \quad \widehat W^{(m)}_{n,k}:=\binom{m}{k}(-na)^{k}
\]
are the unique scalars such that
$$
z^n x^m = \sum_{k=0}^{m}x^{m-k}W^{(m)}_{n,k}z^n \quad {\rm and} \quad z^n y^m = \sum_{k=0}^{m}y^{m-k}\widehat W^{(m)}_{n,k}z^n.
$$

\item[\rm (2)] For all $m,n,s\in\mathbb N$, $\left(x^n y^m\right)^s = x^{ns}y^{ms}$. Moreover, we get the product forms
\begin{equation}\label{xnzmpowercase5v}
\left(x^n z^m\right)^s = \left(\prod_{r=0}^{s-1}(x+rm)^n\right) z^{ms} \quad {\rm and} \quad (y^n z^m)^s = \left(\prod_{r=0}^{s-1}(y-rma)^n\right) z^{ms}.
\end{equation}
Equivalently, there exist uniquely determined scalars
$U^{(n,m)}_{s,\ell},\ \widehat U^{(n,m)}_{s,\ell}\in\Bbbk$ such that
\begin{align}
\left(x^n z^m\right)^s & = \sum_{\ell=0}^{ns}x^{ns-\ell} U^{(n,m)}_{s,\ell} z^{ms},\label{Udefcase5v}\\
\left(y^n z^m\right)^s & = \sum_{\ell=0}^{ns}y^{ns-\ell} \widehat U^{(n,m)}_{s,\ell} z^{ms},\label{Uhatdefcase5v}
\end{align}
with $U^{(n,m)}_{1,0}=\widehat U^{(n,m)}_{1,0}=1$ and $U^{(n,m)}_{1,\ell}=\widehat U^{(n,m)}_{1,\ell}=0$ for $\ell\ge 1$, and the explicit recursions
\begin{align}
U^{(n,m)}_{s+1,\ell}
&=\sum_{k=0}^{\min\{n,\ell\}}\binom{n}{k}(sm)^{k} U^{(n,m)}_{s,\ell-k}, \text{ for } s\ge 1,\ \ell\ge 0,\label{Ureccase5v}\\
\widehat U^{(n,m)}_{s+1,\ell}
&=\sum_{k=0}^{\min\{n,\ell\}}\binom{n}{k}(-sma)^{k} \widehat U^{(n,m)}_{s,\ell-k}, \text{ for } s\ge 1,\ \ell\ge 0.\label{Uhatreccase5v}
\end{align}

\item[\rm (3)] One has the closed form
\begin{equation}\label{xyzpowercase5v}
\left(xyz\right)^s = (x)^{\overline{s}}(y)^{\underline{s}}_{a} z^{s}.
\end{equation}
Equivalently, there exist uniquely determined scalars $V_{s;i,j}\in\Bbbk$, for $0\le i,j\le s$ such that
\begin{equation}\label{Vdefcase5v}
\left(xyz\right)^s=\sum_{i=0}^{s}\ \sum_{j=0}^{s} x^{i}y^{j} V_{s;i,j} z^{s},
\end{equation}
with $V_{1;1,1}=1$ and $V_{1;i,j}=0$ for $(i,j)\neq(1,1)$, and the explicit recursion
\begin{equation}\label{Vreccase5v}
V_{s+1;i,j}=V_{s;i-1,j-1}+s V_{s;i-1,j}-sa V_{s;i,j-1}-s^{2}a V_{s;i,j}, \text{ for } s\ge 1,
\end{equation}
where $V_{s;i,j}=0$ if $i<0$ or $j<0$ or $i>s$ or $j>s$.

\item[\rm (4)] For all $n,m,t,s\in\mathbb N$, we have the closed product form given by
\begin{equation}\label{generalblockpowercase5v}
\left(x^n y^m z^t\right)^s = \left(\prod_{r=0}^{s-1}(x+rt)^{n}(y-rta)^{m}\right) z^{ts}.
\end{equation}

That is, there exist uniquely determined elements $R^{(n,m,t)}_{s;p,q}\in\Bbbk$ such that
\begin{equation}\label{Rdefcase5v}
\left(x^n y^m z^t\right)^s = \sum_{p\ge 0} \sum_{q\ge 0} x^{p}y^{q} R^{(n,m,t)}_{s;p,q} z^{ts},
\end{equation}
with $R^{(n,m,t)}_{1;n,m}=1$ and all other $R^{(n,m,t)}_{1;p,q}=0$, and the explicit recursion
\begin{equation}\label{Rreccase5v}
R^{(n,m,t)}_{s+1;p,q}=\sum_{u=0}^{n} \sum_{v=0}^{m}\binom{n}{u}\binom{m}{v}\left(st\right)^{u} \left(-sta\right)^{v} R^{(n,m,t)}_{s; p-(n-u),\ q-(m-v)}, \text{ for } s\ge 1,
\end{equation}
with the convention $R^{(n,m,t)}_{s;p,q}=0$ if $p<0$ or $q<0$.

\item[\rm (5)] Since $xy=yx$, the ordinary binomial theorem gives
\[
(x+y)^n=\sum_{k=0}^{n}\binom{n}{k}x^{n-k}y^{k}.
\]
Moreover, there exist uniquely determined scalars $E_{n;i,k}, \widehat E_{n;i,k}\in\Bbbk$ such that
\begin{equation}\label{Exzdefcase5v}
\left(x+z\right)^n=\sum_{\substack{i,k\ge 0\\ i+k\le n}} x^{i}E_{n;i,k} z^{k} \quad {\rm and} \quad \left(y+z\right)^n=\sum_{\substack{i,k\ge 0\\ i+k\le n}} y^{i}\widehat E_{n;i,k} z^{k},
\end{equation}
with $E_{0;0,0}=\widehat E_{0;0,0}=1$ and $E_{0;i,k}=\widehat E_{0;i,k}=0$ for $(i,k)\neq(0,0)$, and the explicit recursions
\begin{align}
E_{n+1;i,k} = &\ E_{n;i-1,k}+k E_{n;i,k}+E_{n;i,k-1} \quad {\rm and} \label{Exzreccase5v}\\
\widehat E_{n+1;i,k} = &\ \widehat E_{n;i-1,k}-ka \widehat E_{n;i,k}+\widehat E_{n;i,k-1},\label{Eyzreccase5v}
\end{align}
where the coefficients are $0$ whenever an index is negative or $i+k>n$.
\end{enumerate}
\end{proposition}
\begin{proof} 
By induction on $r\ge 0$, this yields the fundamental shift identities
\begin{equation}\label{shiftlemmacase5v}
z^{r}x=(x+r)z^{r} \quad {\rm and} \quad z^{r}y=(y-ra)z^{r},
\end{equation}
whence for every polynomial $f(x,y)\in\Bbbk[x,y]$, we obtain that 
\begin{equation}\label{shiftpolycase5v}
z^{r}f(x,y)=f(x+r,y-ra) z^{r}.
\end{equation}
\begin{enumerate}
\item [\rm (1)] The identity $y^n x^m=x^m y^n$ is immediate since $xy=yx$. Using \eqref{shiftpolycase5v} with $f(x,y)=x^m$ and $f(x,y)=y^m$, then
$$
z^n x^m = \left(x + n\right)^m z^n \quad {\rm and} \quad z^n y^m = \left(y - na\right)^m z^n.
$$
Expanding the corresponding binomials, we get expressions \eqref{znxmcase5v} - \eqref{znymcase5v}.

\item [\rm (2)] For the expressions $(x^n z^m)^{s+1}=(x^n z^m)^s x^n z^m$, we commute $z^{ms}$ past $x^n$ using \eqref{shiftpolycase5v}: it yields that $z^{ms}x^n=(x+ms)^n z^{ms}$, whence 
\[
\left(x^n z^m\right)^{s+1} = \Big(\prod_{r=0}^{s-1}\left(x + rm\right)^n\Big)\left(x+sm\right)^n z^{m(s+1)},
\]
which gives the first identity in \eqref{xnzmpowercase5v}; the second is analogous using $z^{ms}y^n=(y-sma)^n z^{ms}$. 

Now, expanding 
$$
\left(x + sm\right)^n = \sum_{k=0}^{n}\binom{n}{k}(sm)^k x^{n-k}
$$ 
and comparing coefficients in \eqref{Udefcase5v}, we obtain \eqref{Ureccase5v}; in a similar way for  \eqref{Uhatreccase5v}.
\item [\rm (3)] Using \eqref{shiftlemmacase5v}, it follows that 
\begin{align*}
        (xyz)^{s+1} = &\, (xyz)^s xyz \\
        = &\, \Big((x)^{\overline{s}}(y)^{\underline{s}}_{a}z^{s}\Big)xyz \\
        = &\, (x)^{\overline{s}}(y)^{\underline{s}}_{a}(x+s)(y-sa)z^{s+1},
\end{align*}
and reasoning by induction it follows \eqref{xyzpowercase5v}. For the coefficient recursion, write 
$$
P_s(x,y):=(x)^{\overline{s}} (y)^{\underline{s}}_{a} = \sum_{i,j}V_{s;i,j}x^i y^j.
$$
Since $(x+s)(y-sa) = xy+s y-sa x-s^2a$, multiplying $P_s$ by this factor and factoring by grouping the coefficient of $x^i y^j$ we get \eqref{Vreccase5v}.

\item [\rm (4)] Set $B:=x^n y^m z^t$. Using \eqref{shiftpolycase5v}, $z^{ts}x^n y^m=(x+st)^n(y-sta)^m z^{ts}$. In this way, 
\begin{align*}
B^{s+1}&\ =B^{s}B=\Big(\prod_{r=0}^{s-1}(x+rt)^n(y-rta)^m\Big)z^{ts} x^n y^m z^t \\
&\ =\Big(\prod_{r=0}^{s}(x+rt)^n(y-rta)^m\Big)z^{t(s+1)},
\end{align*}
which is precisely \eqref{generalblockpowercase5v}. Now, expanding 
$$
(x+st)^n=\sum_{u=0}^{n}\binom{n}{u}(st)^{u}x^{n-u} \quad {\rm and} \quad (y-sta)^m=\sum_{v=0}^{m}\binom{m}{v}(-sta)^{v}y^{m-v},
$$ 
we compare the coefficients in \eqref{Rdefcase5v} to obtain \eqref{Rreccase5v}.

\item [\rm (5)] The first identity is the classical binomial theorem since $x$ and $y$ commute. For $(x+z)^n$, define $E_{n;i,k}$ as in \eqref{Exzdefcase5v}. Using that $z^{k}x=(x+k)z^{k}$ from \eqref{shiftlemmacase5v}, we obtain that
\begin{align*}
 (x+z)^{n+1} = &\, (x+z)^n x+(x+z)^n z \\
 = &\, \sum_{i,k}x^{i}\big(x+k\big)E_{n;i,k}z^{k}+\sum_{i,k}x^{i}E_{n;i,k}z^{k+1}.
\end{align*}
Factoring by grouping, the coefficient of $x^i z^k$ gives \eqref{Exzreccase5v}. In a similar way, the equality $z^{k}y = (y-ka)z^{k}$ yields \eqref{Eyzreccase5v}.
\end{enumerate}
\end{proof}

\begin{table}[H]
\captionsetup{justification=centering,font=small}
\caption{Normal ordering of algebras of type (5)}
\label{alg(5)}
\centering
\resizebox{12.5cm}{!}{
\setlength\extrarowheight{6pt}
\begin{tabular}{ |c|c|c|c|c| } 
\hline
 & $y^nx^m$ & $z^nx^m$ & $z^ny^m$ & Conditions/Recursion \\
 \hline\hline
\multirow{6}{*}{(i)} & \multirow{6}{*}{$\displaystyle\sum_{i,j,k\ge 0}W^{(m)}_{n}(i,j,k)x^i y^j z^k$} & \multirow{6}{*}{$\displaystyle\sum_{i,j,k\ge 0}\widetilde W^{(m)}_{n}(i,j,k) x^i y^j z^k$} & \multirow{6}{*}{$\displaystyle\sum_{i,j,k\ge 0}\widehat W^{(m)}_{n}(i,j,k) x^i y^j z^k$} & $W^{(m)}_{0}(i,j,k)=\delta_{i,m}\delta_{j,0}\delta_{k,0}$ \\ 
& & & & $W^{(m)}_{n+1}(i,j,k) = \sum_{r=0}^{\lfloor m/2\rfloor}(-1)^r\binom{m}{2r}W^{(m-2r)}_{n}(i,j-1,k) -\sum_{r=0}^{\lfloor (m-1)/2\rfloor}(-1)^r\binom{m}{2r+1}W^{(m-2r-1)}_{n}(i,j,k-1).$ \\
& & & & $\widetilde W^{(m)}_{0}(i,j,k)=\delta_{i,m}\delta_{j,0}\delta_{k,0}$. \\
& & & & $\widetilde W^{(m)}_{n+1}(i,j,k) =\sum_{r=0}^{\lfloor m/2\rfloor}(-1)^r\binom{m}{2r}\widetilde W^{(m-2r)}_{n}(i,j,k-1)+\sum_{r=0}^{\lfloor (m-1)/2\rfloor}(-1)^r\binom{m}{2r+1}
\sum_{a,b,c\ge 0}\widetilde W^{(m-2r-1)}_{n}(a,b,c)M^y_{a,b,c}(i,j,k)$ \\
& & & & $\widehat W^{(m)}_{0}(i,j,k)=\delta_{i,0}\delta_{j,m}\delta_{k,0}$. \\
& & & & $\widehat W^{(m)}_{n+1}(i,j,k) =\sum_{r=0}^{\lfloor m/2\rfloor}(-1)^r\binom{m}{2r}\widehat W^{(m-2r)}_{n}(i,j,k-1)-\sum_{r=0}^{\lfloor (m-1)/2\rfloor}(-1)^r\binom{m}{2r+1}\sum_{a,b,c\ge 0}\widehat W^{(m-2r-1)}_{n}(a,b,c)M^x_{a,b,c}(i,j,k)$ \\
\hline
(ii) & $\displaystyle\sum_{k=0}^{\min\{m,n\}} x^{m-k}W^{(m)}_{n,k} y^{n-k}z^{k}$ & $x^m z^n$ & $y^m z^n$ & $W^{(m)}_{n,k}=(-1)^k\binom{n}{k} m^{\underline{k}}$  \\ 
\hline
(iii) & $\displaystyle\sum_{k=0}^{\min\{m,n\}} x^{m-k}W^{(m)}_{n,k} y^{n-k}$ & $x^m z^n$ & $y^m z^n$ & Same (ii) \\ 
\hline
\multirow{2}{*}{(iv)} & \multirow{2}{*}{$x^m y^n$} & \multirow{2}{*}{$\displaystyle\sum_{k=0}^{n}\ \sum_{j=0}^{\min\{m,k\}} x^{m-j}y^{j} W^{(m)}_{n;k,j} z^{ n-k}$} & \multirow{2}{*}{$\displaystyle\sum_{k=0}^{n}\binom{n}{k}m^k y^m z^{ n-k}$} & $W^{(m)}_{0;0,0}=1,\ W^{(m)}_{0;k,j}=0\ \text{for }(k,j)\neq(0,0)$ \\ 
& & & & $W^{(m)}_{n+1;k,j}=W^{(m)}_{n;k,j}+m W^{(m)}_{n;k-1,j}+(m-j+1) W^{(m)}_{n;k-1,j-1},\text{ for } n\ge 0,\ k\ge 0,\ j\ge 0$ \\
\hline
(v) & $x^m y^n$ & $\displaystyle\sum_{k=0}^{m}\binom{m}{k} n^{k} x^{m-k} z^n$ & $\displaystyle\sum_{k=0}^{m}\binom{m}{k} (-na)^{k} y^{m-k} z^n$ & None \\ 
\hline
\end{tabular}
}
\end{table}

\begin{table}[H]
\captionsetup{justification=centering,font=small}
\caption{Power of blocks of algebras of type (5)}
\label{alg(5)1}
\centering
\resizebox{12.5cm}{!}{
\setlength\extrarowheight{6pt}
\begin{tabular}{ |c|c|c|c|c| } 
\hline
 & $(x^ny^m)^s$ & $(x^nz^m)^s$ & $(y^nz^m)^s$ & Conditions/Recursion \\
 \hline\hline
\multirow{7}{*}{(i)} & \multirow{7}{*}{$\displaystyle\sum_{i,j,k\ge 0}U^{(n,m)}_{s}(i,j,k)x^iy^jz^k$} & \multirow{7}{*}{$\displaystyle\sum_{i,j,k\ge 0}\widetilde U^{(n,m)}_{s}(i,j,k)x^iy^jz^k$} & \multirow{7}{*}{$\displaystyle\sum_{i,j,k\ge 0}\widehat U^{(n,m)}_{s}(i,j,k)x^iy^jz^k$} & $U^{(n,m)}_{1}(i,j,k)  =\delta_{i,n}\delta_{j,m}\delta_{k,0},\
\widetilde U^{(n,m)}_{1}(i,j,k)  =\delta_{i,n}\delta_{j,0}\delta_{k,m},\ \widehat U^{(n,m)}_{1}(i,j,k) =\delta_{i,0}\delta_{j,n}\delta_{k,m}.$\\
& & & & $U^{[0]}_{s}:=U^{(n,m)}_{s}$ \\
& & & & $U^{[r]}_{s}(i,j,k)=\sum_{a,b,c\ge 0}U^{[r-1]}_{s}(a,b,c)M^x_{a,b,c}(i,j,k)\text{ for } 1\le r\le n$ \\
& & & & $U^{[n+t]}_{s}(i,j,k)=\sum_{a,b,c\ge 0}U^{[n+t-1]}_{s}(a,b,c)M^y_{a,b,c}(i,j,k)\text{ for } 1\le t\le m$ \\
& & & & $U^{(n,m)}_{s+1}(i,j,k)=U^{[n+m]}_{s}(i,j,k), \ \widetilde U^{(n,m)}_{s+1}(i,j,k)=\widetilde U^{[n]}_{s}(i,j,k-m), \ \widehat U^{(n,m)}_{s+1}(i,j,k)=\widehat U^{[n]}_{s}(i,j,k-m)$ \\
& & & & $\widetilde U^{[r]}_{s}(i,j,k)=\sum_{a,b,c\ge 0}\widetilde U^{[r-1]}_{s}(a,b,c)M^x_{a,b,c}(i,j,k)\text{ for } 1\le r\le n$ \\
& & & & $\widehat U^{[r]}_{s}(i,j,k)=\sum_{a,b,c\ge 0}\widehat U^{[r-1]}_{s}(a,b,c)M^y_{a,b,c}(i,j,k)\text{ for } 1\le r\le n$ \\
\hline
\multirow{2}{*}{(ii)} & \multirow{2}{*}{$\displaystyle\sum_{\ell=0}^{\min\{ns,ms\}} x^{ns-\ell} U^{(n,m)}_{s,\ell} y^{ms-\ell}z^{\ell}$} & \multirow{2}{*}{$x^{ns}z^{ms}$} & \multirow{2}{*}{$y^{ns}z^{ms}$} & $U^{(n,m)}_{1,0}=1$ and $U^{(n,m)}_{1,\ell}=0$ for $\ell\ge 1$ \\
& & & & $U^{(n,m)}_{s+1,\ell}=\sum_{k=0}^{\min\{n,\ell\}}U^{(n,m)}_{s,\ell-k} W^{(n)}_{ms-(\ell-k), k},
\text{ for } s\ge 1,\ \ell\ge 0$ \\
\hline
(iii) & $\displaystyle\sum_{\ell=0}^{\min\{ns,ms\}} x^{ns-\ell} U^{(n,m)}_{s,\ell} y^{ms-\ell}$ & $x^{ns}z^{ms}$ & $y^{ns}z^{ms}$ & Same (ii) \\
\hline
(iv) & $x^{ns}y^{ms}$ & $\displaystyle\sum_{\ell=0}^{ms}\ \sum_{j=0}^{\min\{ns,\ell\}}x^{ns-j}y^{j} \widetilde U^{(n,m)}_{s;\ell,j} z^{ ms-\ell}$ & $y^{ns}\displaystyle\prod_{r=0}^{s-1}(z+rn)^{m}$ &  $\widetilde U^{(n,m)}_{1;0,0}=1$ and $\widetilde U^{(n,m)}_{1;\ell,j}=0$ for $(\ell,j)\neq(0,0)$ \\ 
& & & & $\widetilde U^{(n,m)}_{s+1;\ell,j} = \sum_{k=0}^{\min\{n,\ell\}} \sum_{q=0}^{\min\{k,j\}} \widetilde U^{(n,m)}_{s;\ell-k, j-q} W^{(n)}_{ ms-(\ell-k);\ k,q}, \text{ for } s\ge 1,\ \ell\ge 0,\ j\ge 0$ \\
\hline
(v) & $x^{ns}y^{ms}$ & $\Big(\displaystyle\prod_{r=0}^{s-1}(x+rm)^n\Big) z^{ms}$ & $\Big(\displaystyle\prod_{r=0}^{s-1}(y-rma)^n\Big) z^{ms}$ & None \\ 
\hline
\end{tabular}
}
\end{table}

\begin{table}[H]
\captionsetup{justification=centering,font=small}
\caption{Power of three generators of algebras of type (5)}
\label{alg(5)2}
\centering
\resizebox{12.5cm}{!}{
\setlength\extrarowheight{6pt}
\begin{tabular}{ |c|c|c|c| } 
\hline
 & $(xyz)^s$ & $(x^ny^mz^t)^s$ & Conditions/Recursion \\
 \hline\hline
\multirow{6}{*}{(i)} & \multirow{6}{*}{$\displaystyle\sum_{i,j,k\ge 0}V_s(i,j,k) x^i y^j z^k$} & \multirow{6}{*}{$\displaystyle\sum_{i,j,k\ge 0}R^{(n,m,t)}_{s}(i,j,k) x^i y^j z^k$} & $V_1(i,j,k)=\delta_{i,1}\delta_{j,1}\delta_{k,1}, \ V^{[0]}_{s}:=V_s$ \\
& & & $V^{[1]}_{s}(i,j,k)=\sum_{a,b,c\ge 0}V^{[0]}_{s}(a,b,c) M^x_{a,b,c}(i,j,k),\
V^{[2]}_{s}(i,j,k)=\sum_{a,b,c\ge 0}V^{[1]}_{s}(a,b,c) M^y_{a,b,c}(i,j,k)$ \\
& & & $V_{s+1}(i,j,k)=V^{[2]}_{s}(i,j,k-1)$ \\
& & & $R^{(n,m,t)}_{1}(i,j,k)=\delta_{i,n}\delta_{j,m}\delta_{k,t}, \ R^{[0]}_{s}:=R^{(n,m,t)}_{s}$ \\
& & & $R^{[r]}_{s}(i,j,k)=\sum_{a,b,c\ge 0}R^{[r-1]}_{s}(a,b,c) M^x_{a,b,c}(i,j,k), \ R^{[n+t']}_{s}(i,j,k)=\sum_{a,b,c\ge 0}R^{[n+t'-1]}_{s}(a,b,c) M^y_{a,b,c}(i,j,k)$ \\
& & & $R^{(n,m,t)}_{s+1}(i,j,k)=R^{[n+m]}_{s}(i,j,k-t)$ \\
\hline
\multirow{3}{*}{(ii)} & \multirow{3}{*}{$\displaystyle\sum_{\ell=0}^{s} x^{ s-\ell} y^{ s-\ell} V_{s,\ell} z^{ s+\ell}$} & \multirow{3}{*}{$\displaystyle\sum_{\ell=0}^{\min\{ns,ms\}} x^{ns-\ell} R^{(n,m,t)}_{s,\ell} y^{ms-\ell}z^{ts+\ell}$} & $V_{s,\ell}=(-1)^{\ell}\left\{\!\!\begin{matrix} s\\ s-\ell\end{matrix}\!\!\right\}$ \\
& & & $R^{(n,m,t)}_{1,0}=1$ and $R^{(n,m,t)}_{1,\ell}=0$ for $\ell\ge 1$ \\
& & & $R^{(n,m,t)}_{s+1,\ell}=\sum_{k=0}^{\min\{n,\ell\}}R^{(n,m,t)}_{s,\ell-k} W^{(n)}_{ms-(\ell-k), k},
\text{ for } s\ge 1,\ \ell\ge 0,$ \\
\hline
(iii) & $\displaystyle\sum_{\ell=0}^{s} x^{ s-\ell} y^{ s-\ell} V_{s,\ell} z^{s}$ & $\displaystyle\sum_{\ell=0}^{\min\{ns,ms\}} x^{ns-\ell} R^{(n,m,t)}_{s,\ell} y^{ms-\ell}z^{ts}$ &  Same (ii) \\
\hline
\multirow{4}{*}{(iv)} & \multirow{4}{*}{$\displaystyle\sum_{i,j,k\ge 0}V_s(i,j,k) x^i y^j z^k$} & \multirow{4}{*}{$\displaystyle\sum_{\ell=0}^{ts}\ \sum_{j=0}^{\min\{ns,\ell\}}x^{ns-j}y^{ms+j} R^{(n,m,t)}_{s;\ell,j} z^{ ts-\ell}$} & $V_1(i,j,k)=\delta_{i,1}\delta_{j,1}\delta_{k,1}, \ V^{[x]}_s(i,j,k)=\sum_{c\ge k}\binom{c}{k} V_s(i-1,j,c) + \sum_{c\ge k+1}c\binom{c-1}{k} V_s(i,j-1,c)$ \\
& & & $V^{[xy]}_s(i,j,k)=\sum_{c\ge k}\binom{c}{k} V^{[x]}_s(i,j-1,c), \ V_{s+1}(i,j,k)=V^{[xy]}_s(i,j,k-1), \text{ for } s\ge 1$ \\
& & & $R^{(n,m,t)}_{s+1;\ell,j}=\sum_{k=0}^{\min\{n,\ell\}} \sum_{q=0}^{\min\{k,j\}}
R^{(n,m,t)}_{s;\ell-k, j-q} \mathcal W^{(n,m)}_{ ts-(\ell-k);\ k,q}, \ \mathcal W^{(n,m)}_{0;0,0}=1$\\
& & & $\mathcal W^{(n,m)}_{r+1;\ k,q}=\mathcal W^{(n,m)}_{r;\ k,q}+(n+m) \mathcal W^{(n,m)}_{r;\ k-1,q}+(n-q+1) \mathcal W^{(n,m)}_{r;\ k-1,q-1}$ \\
\hline
\multirow{4}{*}{(v)} & \multirow{4}{*}{$\displaystyle\sum_{i=0}^{s}\ \sum_{j=0}^{s} x^{i}y^{j} V_{s;i,j} z^{s}$} & \multirow{4}{*}{$\displaystyle\sum_{p\ge 0} \sum_{q\ge 0} x^{p}y^{q} R^{(n,m,t)}_{s;p,q} z^{ts}$} & $V_{1;1,1}=1$ and $V_{1;i,j}=0$ for $(i,j)\neq(1,1)$ \\ 
& & & $V_{s+1;i,j}=V_{s;i-1,j-1}+s V_{s;i-1,j}-sa V_{s;i,j-1}-s^{2}a V_{s;i,j}, \text{ for } s\ge 1$ \\
& & & $R^{(n,m,t)}_{1;n,m}=1$ and all other $R^{(n,m,t)}_{1;p,q}=0$ \\
& & & $R^{(n,m,t)}_{s+1;p,q}=\sum_{u=0}^{n} \sum_{v=0}^{m}\binom{n}{u}\binom{m}{v}(st)^{u}(-sta)^{v} R^{(n,m,t)}_{s; p-(n-u),\ q-(m-v)}, \text{ for } s\ge 1$ \\
\hline
\end{tabular}
}
\end{table}

\begin{table}[H]
\captionsetup{justification=centering,font=small}
\caption{Binomial formulas for algebras of type (5)}
\label{alg(5)3}
\centering
\resizebox{12.5cm}{!}{
\setlength\extrarowheight{6pt}
\begin{tabular}{ |c|c|c|c|c| } 
\hline
 & $(x+y)^n$ & $(x+z)^n$ & $(y+z)^n$ & Conditions/Recursion \\
 \hline\hline
\multirow{4}{*}{(i)} & \multirow{4}{*}{$\displaystyle\sum_{i,j,k\ge 0}E^{\langle xy\rangle}_{n}(i,j,k) x^i y^j z^k$} & \multirow{4}{*}{$\displaystyle\sum_{i,j,k\ge 0}E^{\langle xz\rangle}_{n}(i,j,k) x^i y^j z^k$} & \multirow{4}{*}{$\displaystyle\sum_{i,j,k\ge 0}E^{\langle yz\rangle}_{n}(i,j,k) x^i y^j z^k$} & $E^{\langle\ast\rangle}_{0}(i,j,k)=\delta_{i,0}\delta_{j,0}\delta_{k,0}$ \\
& & & & $E^{\langle xy\rangle}_{n+1}(i,j,k)=\sum_{a,b,c\ge 0}E^{\langle xy\rangle}_{n}(a,b,c)\Big(M^x_{a,b,c}(i,j,k)+M^y_{a,b,c}(i,j,k)\Big)$ \\
& & & & $E^{\langle xz\rangle}_{n+1}(i,j,k)=\sum_{a,b,c\ge 0}E^{\langle xz\rangle}_{n}(a,b,c)\Big(M^x_{a,b,c}(i,j,k)+M^z_{a,b,c}(i,j,k)\Big)$ \\
& & & & $E^{\langle yz\rangle}_{n+1}(i,j,k) =\sum_{a,b,c\ge 0}E^{\langle yz\rangle}_{n}(a,b,c)\Big(M^y_{a,b,c}(i,j,k)+M^z_{a,b,c}(i,j,k)\Big)$ \\
\hline
(ii) & $\displaystyle\sum_{\substack{i,k\ge 0\\ i+2k\le n}} x^{i} E_{n;i,k} y^{ n-i-2k} z^{k}$ & $\displaystyle\sum_{k=0}^{n}\binom{n}{k}x^{n-k}z^{k}$ & $\displaystyle\sum_{k=0}^{n}\binom{n}{k}y^{n-k}z^{k}$ & $E_{n;i,k}=(-1)^k \frac{n!}{i! (n-i-2k)! k! 2^{k}}, \text{ for } i+2k\le n$ \\
\hline
(iii) & $\displaystyle\sum_{i=0}^{n} x^{i}E_{n;i} y^{ n-i}$ & $\displaystyle\sum_{k=0}^{n}\binom{n}{k}x^{n-k}z^{k}$ & $\displaystyle\sum_{k=0}^{n}\binom{n}{k}y^{n-k}z^{k}$ & $E_{n;i}=\binom{n}{i} ( -b)^{ n-i} (n-i)!$ \\
\hline
\multirow{4}{*}{(iv)} & \multirow{4}{*}{$\displaystyle\sum_{k=0}^{n}\binom{n}{k}x^{n-k}y^{k}$} & \multirow{4}{*}{$\displaystyle\sum_{\substack{i,j,k\ge 0\\ i+j+k\le n}} x^{i}y^{j} E^{\langle xz\rangle}_{n;i,j,k} z^{k}$} & \multirow{4}{*}{$\displaystyle\sum_{\substack{i,k\ge 0\\ i+k\le n}} y^{i} E^{\langle yz\rangle}_{n;i,k} z^{k}$} & $E^{\langle yz\rangle}_{0;0,0}=1$ and $E^{\langle yz\rangle}_{0;i,k}=0$ for $(i,k)\neq(0,0)$ \\ 
& & & & $E^{\langle yz\rangle}_{n+1;i,k}=E^{\langle yz\rangle}_{n;i,k-1}+\sum_{c\ge k}\binom{c}{k} E^{\langle yz\rangle}_{n;i-1,c}$ \\
& & & & $E^{\langle xz\rangle}_{0;0,0,0}=1$ and $E^{\langle xz\rangle}_{0;i,j,k}=0$ for $(i,j,k)\neq(0,0,0)$ \\
& & & & $E^{\langle xz\rangle}_{n+1;i,j,k}=E^{\langle xz\rangle}_{n;i,j,k-1}
+\sum_{c\ge k}\binom{c}{k} E^{\langle xz\rangle}_{n;i-1,j,c}
+\sum_{c\ge k+1}c\binom{c-1}{k} E^{\langle xz\rangle}_{n;i,j-1,c}$ \\
\hline
\multirow{2}{*}{(v)} & \multirow{2}{*}{$\displaystyle\sum_{k=0}^{n}\binom{n}{k}x^{n-k}y^{k}$} & \multirow{2}{*}{$\displaystyle\sum_{\substack{i,k\ge 0\\ i+k\le n}} x^{i}E_{n;i,k} z^{k}$} & \multirow{2}{*}{$\displaystyle\sum_{\substack{i,k\ge 0\\ i+k\le n}} y^{i}\widehat E_{n;i,k} z^{k}$} & $E_{0;0,0}=\widehat E_{0;0,0}=1$ and $E_{0;i,k}=\widehat E_{0;i,k}=0$ for $(i,k)\neq(0,0)$ \\ 
& & & & $E_{n+1;i,k}=E_{n;i-1,k}+k E_{n;i,k}+E_{n;i,k-1},\ \widehat E_{n+1;i,k}=\widehat E_{n;i-1,k}-ka \widehat E_{n;i,k}+\widehat E_{n;i,k-1}$ \\
\hline
\end{tabular}
}
\end{table}

\section{Computational setup}\label{Computationalsetup}

Having in mind Definition \ref{3dimensionaldimension} (ii), throughout the paper we fix the PBW monomial order $x\prec y\prec z$, so that the intended normal forms are $\Bbbk$-linear combinations of PBW monomials $x^i y^j z^k$. In this section, we implement a deterministic reduction procedure that repeatedly rewrites the leftmost forbidden adjacent pair among $yx, zy$ and $zx$. For the fifteen algebras classified in Proposition \ref{3-dimensionalClassification} Bergman's diamond lemma \cite{Bergman1978} is applied (c.f. \cite{ReyesSuarez20173D}), so the output of the code below coincides with the unique PBW normal form.

We rewrite $3$-dimensional skew polynomial algebras in the following form: $A$ is a $\Bbbk$-algebra generated by the indeterminates $x,y,z$ restricted to relations 
\begin{align}
yz-\alpha zy & =a_{\lambda}x+b_{\lambda}y+c_{\lambda}z+d_{\lambda}, \notag \\
zx-\beta xz & =a_{\mu}x+b_{\mu}y+c_{\mu}z+d_{\mu}, \quad {\rm and}\ \label{rel3skew} \\
xy-\gamma yx &= a_{\nu}x+b_{\nu}y+c_{\nu}z+d_{\nu}, \notag
\end{align} 
with the elements $a$'s, $b$'s, $c$'s, $d$'s belonging to $\Bbbk$ and $\alpha, \beta, \gamma \in \Bbbk^{\times}$.

Instead of writing separate \texttt{SageMath} scripts for each of these algebras and for each computational task, we implement a single PBW-reduction engine for the general defining relations \eqref{rel3skew}. In this way, every algebra in the list is obtained by specializing the parameters $\alpha, \beta, \gamma$ and $a_i,b_i,c_i,d_i \in \Bbbk$.

The next code implements a PBW reduction engine  that rewrites the forbidden pairs $yx,zy,zx$ using \eqref{rel3skew} to compute PBW normal forms, and hence products and powers after parameter specialization.

\begin{MyVerbatim}
class ThreeSkewPBW:
    def __init__(self, K, params):
        """
        K: a field where alpha,beta,gamma are invertible.
        params: dict with keys:
          alpha,beta,gamma, A1,B1,C1,D1, A2,B2,C2,D2, A3,B3,C3,D3
        """
        self.K = K
        self.p = params
        self.forbidden = {('y','x'), ('z','y'), ('z','x')}

    def _rewrite_pair(self, pair):
        K = self.K
        p = self.p

        alpha = p['alpha']; beta = p['beta']; gamma = p['gamma']
        A1,B1,C1,D1 = p['A1'],p['B1'],p['C1'],p['D1']
        A2,B2,C2,D2 = p['A2'],p['B2'],p['C2'],p['D2']
        A3,B3,C3,D3 = p['A3'],p['B3'],p['C3'],p['D3']

        # yx -> gamma^{-1}(xy - (A3 x + B3 y + C3 z + D3))
        if pair == ('y','x'):
            invg = K(1)/gamma
            return {
                ('x','y'): invg,
                ('x',):   -invg*A3,
                ('y',):   -invg*B3,
                ('z',):   -invg*C3,
                ():       -invg*D3
            }

        # zy -> alpha^{-1}(yz - (A1 x + B1 y + C1 z + D1))
        if pair == ('z','y'):
            inva = K(1)/alpha
            return {
                ('y','z'): inva,
                ('x',):    -inva*A1,
                ('y',):    -inva*B1,
                ('z',):    -inva*C1,
                ():        -inva*D1
            }

        # zx -> beta*xz + (A2 x + B2 y + C2 z + D2)
        if pair == ('z','x'):
            return {
                ('x','z'): K(beta),
                ('x',):    K(A2),
                ('y',):    K(B2),
                ('z',):    K(C2),
                ():        K(D2)
            }

        return None

    def reduce_lin(self, expr):
        """
        expr: dict word(tuple of 'x','y','z') -> coefficient in K.
        """
        K = self.K
        expr = {tuple(w): K(c) for w,c in expr.items() if c != 0}
        changed = True
        while changed:
            changed = False
            new = {}
            for word, coef in expr.items():
                idx = None
                for i in range(len(word)-1):
                    if (word[i], word[i+1]) in self.forbidden:
                        idx = i
                        break
                if idx is None:
                    new[word] = new.get(word, K(0)) + coef
                    continue

                changed = True
                pre = word[:idx]
                suf = word[idx+2:]
                repl = self._rewrite_pair((word[idx], word[i+1]))

                repl = self._rewrite_pair((word[idx], word[idx+1]))
                for w2, c2 in repl.items():
                    wfull = pre + w2 + suf
                    new[wfull] = new.get(wfull, K(0)) + coef*c2
            expr = {w:c for w,c in new.items() if c != 0}
        return expr

    def NF_word(self, word):
        return self.reduce_lin({tuple(word): self.K(1)})

    def add(self, A, B, sA=None, sB=None):
        K = self.K
        if sA is None: sA = K(1)
        if sB is None: sB = K(1)
        C = {}
        for w,c in A.items():
            C[w] = C.get(w,K(0)) + sA*c
        for w,c in B.items():
            C[w] = C.get(w,K(0)) + sB*c
        return {w:c for w,c in C.items() if c != 0}

    def mul(self, A, B):
        K = self.K
        prod = {}
        for w1,c1 in A.items():
            for w2,c2 in B.items():
                w = w1 + w2
                prod[w] = prod.get(w, K(0)) + c1*c2
        return self.reduce_lin(prod)

    def pow(self, A, n):
        K = self.K
        E = {(): K(1)}
        if n == 0:
            return E
        for _ in range(n):
            E = self.mul(E, A)
        return E

    def X(self): return {('x',): self.K(1)}
    def Y(self): return {('y',): self.K(1)}
    def Z(self): return {('z',): self.K(1)}

    def collect_PBW(self, NF):
        out = {}
        for w,c in NF.items():
            i=0
            while i<len(w) and w[i]=='x': i+=1
            j=0
            t=i
            while t<len(w) and w[t]=='y':
                j += 1
                t += 1
            k = len(w)-t
            out[(i,j,k)] = out.get((i,j,k), self.K(0)) + c
        return {e:v for e,v in out.items() if v != 0}
\end{MyVerbatim}

The next code provides reusable Sage routines for any specialization of the parameters.

\begin{MyVerbatim}
def word_pow(letter, n):
    return [letter]*n

def word_block(n,m,t):
    return (['x']*n) + (['y']*m) + (['z']*t)

# (1) y^n x^m, z^n x^m, z^n y^m
def item1_yx(A, n, m): return A.NF_word(word_pow('y',n) + word_pow('x',m))
def item1_zx(A, n, m): return A.NF_word(word_pow('z',n) + word_pow('x',m))
def item1_zy(A, n, m): return A.NF_word(word_pow('z',n) + word_pow('y',m))

# (2) (x^n y^m)^s, (x^n z^m)^s, (y^n z^m)^s
def item2_xnyms(A, n, m, s):
    base = A.NF_word(word_pow('x',n) + word_pow('y',m))
    return A.pow(base, s)

def item2_xnzms(A, n, m, s):
    base = A.NF_word(word_pow('x',n) + word_pow('z',m))
    return A.pow(base, s)

def item2_ynzms(A, n, m, s):
    base = A.NF_word(word_pow('y',n) + word_pow('z',m))
    return A.pow(base, s)

# (3) (xyz)^s
def item3_xyzs(A, s):
    base = A.NF_word(['x','y','z'])
    return A.pow(base, s)

# (4) (x^n y^m z^t)^s
def item4_block_power(A, n, m, t, s):
    base = A.NF_word(word_block(n,m,t))
    return A.pow(base, s)

# (5) Newton binomials
def item5_binomial(A, u, v, n):
    U = {'x':A.X(), 'y':A.Y(), 'z':A.Z()}[u]
    V = {'x':A.X(), 'y':A.Y(), 'z':A.Z()}[v]
    return A.pow(A.add(U, V), n)
\end{MyVerbatim}

This last code encodes each of the $15$ algebras as a parameter dictionary, so the same PBW engine and routines can be run by simply switching the selected case.

\begin{MyVerbatim}
R.<alpha,beta,gamma,a,b, A1,B1,C1,D1, A2,B2,C2,D2, A3,B3,C3,D3> = PolynomialRing(QQ)
K = FractionField(R)

def zero_params():
    return dict(alpha=K(1), beta=K(1), gamma=K(1),
                A1=K(0),B1=K(0),C1=K(0),D1=K(0),
                A2=K(0),B2=K(0),C2=K(0),D2=K(0),
                A3=K(0),B3=K(0),C3=K(0),D3=K(0))

CASES = {}

# (1) yz-alpha zy=0, zx-beta xz=0, xy-gamma yx=0 (|{alpha,beta,gamma}|=3)
p = zero_params()
p.update(alpha=K(alpha), beta=K(beta), gamma=K(gamma))
CASES['(1)'] = p

# (2) beta != alpha=gamma=1
p = zero_params(); p.update(beta=K(beta), C1=K(1), B2=K(1), A3=K(1))     ; CASES['(2)(i)']  = p
p = zero_params(); p.update(beta=K(beta), C1=K(1), D2=K(b), A3=K(1))     ; CASES['(2)(ii)'] = p
p = zero_params(); p.update(beta=K(beta), B2=K(1))                       ; CASES['(2)(iii)']= p
p = zero_params(); p.update(beta=K(beta), D2=K(b))                       ; CASES['(2)(iv)'] = p
p = zero_params(); p.update(beta=K(beta), C1=K(a), A3=K(1))              ; CASES['(2)(v)']  = p
p = zero_params(); p.update(beta=K(beta), C1=K(1))                       ; CASES['(2)(vi)'] = p

# (3) beta != alpha=gamma != 1
p = zero_params(); p.update(alpha=K(alpha), beta=K(beta), gamma=K(alpha),
                            B2=K(1), D2=K(b))                             ; CASES['(3)(i)']  = p
p = zero_params(); p.update(alpha=K(alpha), beta=K(beta), gamma=K(alpha),
                            D2=K(b))                                      ; CASES['(3)(ii)'] = p

# (4) alpha=beta=gamma != 1: yz-a zy = A1*x + D1, zx-a xz = B2*y + D2, xy-a yx = C3*z + D3
p = zero_params()
p.update(alpha=K(alpha), beta=K(alpha), gamma=K(alpha),
         A1=K(A1), D1=K(D1),  B2=K(B2), D2=K(D2),  C3=K(C3), D3=K(D3))
CASES['(4)'] = p

# (5) alpha=beta=gamma=1
p = zero_params(); p.update(A1=K(1), B2=K(1), C3=K(1))                    ; CASES['(5)(i)']  = p
p = zero_params(); p.update(C3=K(1))                                      ; CASES['(5)(ii)'] = p
p = zero_params(); p.update(D3=K(b))                                      ; CASES['(5)(iii)']= p
p = zero_params(); p.update(B1=K(-1), A2=K(1), B2=K(1))                    ; CASES['(5)(iv)'] = p
p = zero_params(); p.update(C1=K(a), A2=K(1))                              ; CASES['(5)(v)']  = p
\end{MyVerbatim}

\section{Future work}\label{Futurework}

As a natural continuation of the present work, we want to extend the combinatorial and computational analysis of PBW normal ordering developed here to other relevant classes of noncommutative algebras admitting structured rewriting systems, in particular to diffusion algebras \cite{Hinchcliffe2005, IsaevPyatovRittenberg2001, PyatovTwarock2002, RubianoReyes2026Diffusionalgebras} and bi-quadratic algebras on three generators with PBW basis \cite{Bavula2023, RubianoReyes2026Biquadraticalgebras}. Our aim is to determine whether these families also support a uniform normal ordering theory in terms of explicit recursive reduction procedures and coefficient arrays generalizing the Stirling-type patterns that arise in the skew polynomial setting.

\section{Declarations}

The authors have no conflict of interest to disclose.

\end{document}